\documentclass[a4paper,10pt,leqno]{amsart}
\usepackage[T1]{fontenc}
\usepackage[utf8]{inputenc}
\usepackage[margin=1.2in]{geometry}
\usepackage{lmodern}
\usepackage[english]{babel}
\usepackage{graphicx}
\usepackage{amsmath}
\usepackage{amsfonts}
\usepackage{amssymb} 
\usepackage{amsthm}
\usepackage{hyperref}
\usepackage{bbm}
\usepackage{bm} 
\usepackage[foot]{amsaddr}
\usepackage{subfig}
\usepackage{lipsum}
\usepackage{verbatim}
\usepackage{xcolor}
\usepackage{url}
\theoremstyle{plain}
\newtheorem{thm}{Theorem}[section]
\newtheorem{lem}[thm]{Lemma}
\newtheorem{prop}[thm]{Proposition}
\newtheorem{cor}[thm]{Corollary}
\newtheorem{conj}{Conjecture}[section]
\DeclareMathOperator\arctanh{arctanh}

\DeclareMathOperator\sign{sign}
\newtheorem{rem}{Remark}

\def \be {\begin{equation}}
\def \ee {\end{equation}}

\begin{document}

\title[A hierarchical mean field model of interacting spins]{A hierarchical mean field model of interacting spins}

\author{Paolo Dai Pra}
\author{Marco Formentin}
\author{Guglielmo Pelino}
\noindent \address[P. Dai Pra, M. Formentin, G. Pelino]{Department of Mathematics ``Tullio Levi-Civita'', \newline \indent University of Padua, \newline \indent Via Trieste 63, 35121 Padova, Italy.}
\address[P. Dai Pra]{Department of Computer Science, \newline \indent University of Verona, \newline \indent Strada Le Grazie 15, 37134 Verona, Italy.}
\address[M. Formentin]{Padova Neuroscience Center, \newline \indent University of Padua,\newline \indent via Giuseppe Orus 2, 35131 Padova, Italy.}
\vspace{1cm}
\email[P. Dai Pra]{daipra@math.unipd.it}
\email[M. Formentin]{marco.formentin@unipd.it}
\email[G. Pelino]{guglielmo.pelino@math.unipd.it}
\thanks{\textbf{Funding:} The authors acknowledge financial support through the project ``Large Scale Random Structures'' of the Italian Ministry of Education, Universities and Research (PRIN 20155PAWZB-004). The last author is partially supported by the PhD Program in Mathematical Science, Department of Mathematics, University of Padua (Italy), Progetto Dottorati - Fondazione Cassa di Risparmio di Padova e Rovigo. 
}


\date{\today}

%
%

\begin{abstract}
We consider a system of hierarchical interacting spins under dynamics of spin-flip type with a ferromagnetic mean field interaction, scaling with the hierarchical distance, coupled with a system of linearly interacting hierarchical diffusions of Ornstein-Uhlenbeck type. In particular, the diffusive variables enter in the spin-flip rates, effectively acting as dynamical magnetic fields. In absence of the diffusions, the spin-flip dynamics can be thought of as a modification of the Curie--Weiss model. We study the mean field and the two-level hierarchical model, in the latter case restricting to a subcritical regime, corresponding to high temperatures, obtaining macroscopic limits at different spatio-temporal scales and studying the phase transitions in the system. We also formulate a generalization of our results to the $k$-th level hierarchical case, for any $k$ finite, in the subcritical regime. We finally address the supercritical regime, in the zero-temperature limit, for the two-level hierarchical case, proceeding heuristically with the support of numerics.
\end{abstract}

\maketitle


\section{Motivation}
\label{motivation}
Hierarchical models are often employed in the literature for applications in population dynamics and genetics, where individuals naturally dispose in groups with a hierarchical structure (families, clans, villages, colonies, populations and so on).
A series of papers from the '90s - '00s (initiated with \cite{dawson_greven_hier} and \cite{dawson_greven_1lev} among others), nicely reviewed in \cite{denholl}, deals with different types of hierarchical mean field linearly interacting diffusions (the prototype being linear Wright-Fisher diffusions), where in most cases the macroscopic limits are retrieved at every spatio-temporal scale, and a renormalization map can be defined, allowing one to pass from one hierarchical level to the other. Moreover, the study of the fixed points of the renormalization map is in some cases fully worked out. Two crucial ingredients which allow for an iterative renormalization procedure are the linearity of the interactions, which in the above works is realized by considering linear drifts, of imitative type, which scale with the hierarchical distance, and some ergodicity properties of the individual dynamics.
The motivation for focusing on diffusive dynamics as building blocks for the hierarchical models stems from the fact that, with their choices, each individual non-interacting dynamics can itself be obtained as a continuum limit of a corresponding finite state space model of interacting particles: for example, the discrete prelimit counterpart of the Wright-Fisher diffusion is the \textit{Voter model} (see e.g.\! \cite{cox_longtime}).

When working directly on finite state models fewer results are known, due to the non-linearity of the microscopic interactions. Hierarchical Ising-type models for spin systems were introduced in \cite{dyson}. Since then, a literature on the hierarchical group and renormalization theory for spin systems was developed (e.g.\! \cite{bleher}, \cite{derrida}, \cite{desimoi}, \cite{hara}, \cite{kaufman}), but always studying equilibrium models. 
On the finite state space dynamics, we acknowledge the work \cite{athreya}, which studies contact processes on the hierarchical group, with a focus on deriving sufficient conditions on the speed of decay of the infection rates for obtaining a phase transition between extinction and survival.

Here we study a model of non-Markovian interacting spins with a hierarchical mean field structure; at the individual level, the full Markovian state is indeed given by a pair of variables: the spin and a continuous variable which evolves in a diffusive way, modeling some aggregated remaining characteristics of the individual. The main goal of our study is to obtain macroscopic limits at various spatio-temporal scales for both the mean field and the two-level hierarchical formulation of the model, analyzing the presence of phase transitions in the system. Note that in our model the interaction between the spins is highly non-linear. However, as we shall see, the linear diffusions \textit{drive} the system of spins, eventually allowing for a separation of spatio-temporal scales for the spin dynamics as well.  

The rest of the paper is organized as follows: in Section \ref{intro_hier} we formulate the dynamics for a general interaction graph, which we then specify to the two contexts of our interest: the mean field case - analyzed in Section \ref{mean_field_hier}, and the two-level hierarchical case - analyzed in Section \ref{hier_hier}. In particular, in Section \ref{mean_field_hier} we derive the macroscopic limit at the two characteristic timescales of the model for any value of the parameters, highlighting the presence of a phase transition, and studying the resulting effects on the dynamics at each timescale. In Section \ref{hier_hier}, we study the macroscopic limits at the three different timescales of the two-level hierarchical model, restricting ourselves to a range of interaction parameters which we refer to as \textit{subcritical}, corresponding to a high temperature regime. We also formulate a generalization of these results to the $k$-level hierarchical version of the model, for any $k \in \mathbb{N}$ finite (Section \ref{renormalization_theory}). In the \textit{supercritical} region we focus on the zero-temperature limit (Section \ref{the_limit_case}), where we give a description of the limit dynamics supported by numerics and heuristic arguments, allowing for a comparison with the mean field scenario. 

\section{Introducing the model}
\label{intro_hier}
Consider a set $V$ (possibly countably infinite), indexing individuals in a population. Each individual $r \in V$ is identified with a pair of variables $(\mu_r, x_r)$: a spin variable $\mu_r \in \left\{-1,1\right\}$, and a continuous one $x_r \in \mathbb{R}$, modeling some summary statistics of the remaining characteristics of the individual, and thus being naturally normally distributed by a central limit theorem.
The interaction between each pair of spin variables $\mu_r, \mu_s \in V$ is encoded in a (possibly random) variable $J_{rs} \in \mathbb{R}$. Analogously, $x_r$ and $x_s$ interact with a strength proportional to some variables $J'_{rs}\in \mathbb{R}$. The particles $(\mu_r,x_r)_{r \in V}$ follow stochastic dynamics given by
\begin{equation}
\label{eqn:graph}
\begin{cases}
\mu_r \mapsto  \ -\mu_r, \ \ \ \ \text{      with rate     }\ \ \  1 + \tanh\left[-\mu_r \sum_{s \in V}J_{rs}(\mu_s + x_s)\right],\\
d x_r  = - \sum_{s \in V}J'_{rs}(x_r - x_s)dt + \sigma dW_r(t),
\end{cases}
\end{equation} 
where $W_r(t)$'s are $|V|$ independent Brownian motions, and $\sigma > 0$ is the diffusion coefficient. 
The choice of the rate function $1+\tanh(\cdot)$ in \eqref{eqn:graph} might seem unusual. Note that it is alternative to the more common choice $e^{-\mu_r \sum_{s \in V}J_{rs}(\mu_s+x_s)}$. As the latter, in the case without diffusions, it defines a Glauber-type spin-flip dynamics with respect to which the Gibbs measure
$$
\pi(\bm{\mu}) \propto 1+\tanh\left(\sum_{r,s \in V}J_{rs}\mu_r\mu_s\right)
$$
is reversible. The reason for the alternative choice $1 + \tanh(\cdot)$ is technical, as the boundedness of the transition rates is convenient for the proofs, even though we believe it is not an essential ingredient.

We focus on two different choices for $V$ and (deterministic) interaction parameters $J_{rs}$ and $J'_{rs}$: 
\begin{itemize}
\item \textit{Ferromagnetic mean field case:} 
\begin{equation}
\label{eqn:mf}
\begin{aligned}
V &:= \left\{1,\dots,N\right\},\\
J_{rs} &= \frac{\beta}{N},\\
 J'_{rs} &= \frac{\alpha}{N},
\end{aligned}
\end{equation}
with $ \alpha, \beta \geq 0$.
\item \textit{Ferromagnetic two-level hierarchical case:} 
\begin{equation}
\label{eqn:hier}
\begin{aligned}
V:=  \left\{1,\dots,N\right\}\times\left\{1,\dots,N\right\},\\ 
\begin{cases}
J_{rs} =  \frac{\beta_1}{N},  \ \  J'_{rs} = \frac{\alpha_1}{N}, \ \ \ \text{  if } |r-s| \leq 1,\\
J_{rs} = \frac{\beta_2}{N^2}, \ \ J'_{rs} = \frac{\alpha_2}{N^3}, \ \ \ \text{ if } |r-s| = 2,
\end{cases}
\end{aligned}
\end{equation}
with $\alpha_1, \alpha_2, \beta_1,\beta_2 \geq 0$, where the distance $|\cdot|$ between $r := (i,j)$ and $s := (k,l)$ is defined by
\begin{equation}
\label{eqn:hierarchicaL_dista}
|r-s|:=
\begin{cases}
0, \ \ \ \text{ if } i = k, \ j = l\\
1, \ \ \ \text{ if } i \neq k, \  j = l \\
2, \ \ \ \text{ otherwise}.
\end{cases}
\end{equation}
\end{itemize}
The two-level hierarchical case can be thought of as a model for a collection of $N$ interacting populations, each of which is itself a mean field interacting particle system with $N$ particles. 
In the definition \eqref{eqn:hierarchicaL_dista} of the hierarchical distance $|r-s|$, the first index $i$ refers to the individual, while the index $j$ identifies the $j$-th population. Two individuals $r= (i,j)$ and $s = (k,l)$ are thus said to be at distance $1$ if $j = l$ (i.e.\! they belong to the same population); otherwise, they are at distance $2$. The choices in \eqref{eqn:hier} are such that the strength of the interaction decays with the hierarchical distance.
This construction can be reiterated a finite number of times to define a $k$-level hierarchical model, where $V := \left\{1,\dots,N\right\}^{k}$, $J_{rs} \propto \frac{1}{N^l}$, $J'_{rs} \propto \frac{1}{N^{2l -1}}$ for $|r-s| = l$, with $l=1,\dots,k$. See Section \ref{renormalization_theory} for details.
The main goal of this paper is to obtain a limit description of both the mean field and the two-level hierarchical formulation of dynamics \eqref{eqn:graph} at different spatio-temporal scales, analyzing the possible presence of phase transitions in the system.

\section{The mean field model}
\label{mean_field_hier}
In this section we study the mean field version of the model, i.e.\! the case of a single population of $N$ individuals with a mean field type interaction. 
We denote by $(\bm{\mu}, \bm{x}) := (\mu_j,x_j)_{j=1,\dots,N} \in (\left\{-1,1\right\} \times \mathbb{R})^N$ a configuration of the entire population. In the following, we interchangeably use the coordinates $(\mu_i, \lambda_i)$ and $(\mu_i, x_i)$, where $\lambda_i := \mu_i + x_i$ is the total local field of the $i$-th individua. Let
$$
m^N(t) := \frac{1}{N}\sum_{i=1}^N \mu_i(t)
$$ 
be the magnetization of the spin variables at time $t$, and $x^N(t)$ (resp.\! $\lambda^N(t)$) the analogous quantity for the $x_i(t)$'s (resp.\! $\lambda_i(t)$'s).
The dynamics is such that, at time $t$, the $i$-th spin flips with rate 
$$
\mu_i \mapsto -\mu_i, \ \ \ \text{with rate} \ \ \ 1 + \tanh({-\beta\mu_i(t)\lambda^N(t)}),
$$
where $\lambda^N(t)$ and $x^N(t)$ satisfy, substituting the mean field coupling constants \eqref{eqn:mf} in the general dynamics \eqref{eqn:graph},
\begin{equation}
\label{eqn:lambda}
\begin{cases}
d \lambda^N(t)  = dm^N(t) + dx^N(t),\\
d x^N(t)  = \frac{\sigma}{\sqrt{N}}dW^N(t), 
\end{cases}
\end{equation}
where $W^N:= \frac{1}{\sqrt{N}}\sum_{i=1}^N W_i$ is a Brownian motion and $\sigma > 0$ the diffusion coefficient. We stress that the law of $W^N$ does not depend on $N$, but we keep the notation $W^N$ to refer to the specific Brownian motion obtained by the aggregation of the single $W_i$'s.

From the definition of the spin-flip rates, we obtain the transition rates at time $t$ for the order parameter $m^N$
\begin{equation}
\label{eqn:m}
\begin{aligned}
m^N \mapsto & \ m^N + \frac{2}{N}, \ \ \ \ \text{      with rate     }\ \ \  N \frac{1 - m^N(t)}{2} \left[1 + \tanh({\beta \lambda^N(t)})\right]\\
m^N \mapsto & \ m^N - \frac{2}{N}, \ \ \ \ \text{      with rate     } \ \ \   N \frac{1+m^N(t)}{2} \left[1 - \tanh({\beta \lambda^N(t)})\right].
\end{aligned}
\end{equation}
We assume i.i.d.\! initial data for the single variables $x_i(0) \sim \mathcal{N}\left(x_0, \sigma^2\right)$, and $\mu_i(0) \sim \text{Ber}(p)$, for some $p \in [0,1]$. 

The infinitesimal generator associated to the dynamics \eqref{eqn:lambda} and \eqref{eqn:m}, applied to a function $f: \mathbb{R} \times [-1,1] \to \mathbb{R}$, is given by
\begin{equation}
\label{eqn:generator}
\begin{aligned}
\mathcal{L}^N & f(\lambda, m) = N\frac{1-m}{2} \left[1 + \tanh({\beta \lambda})\right] \left[f\left(\lambda + \frac{2}{N}, m + \frac{2}{N}\right) - f(\lambda, m)\right]\\
& +  N\frac{1+m}{2} \left[1 - \tanh({\beta \lambda})\right] \left[f\left(\lambda - \frac{2}{N}, m - \frac{2}{N}\right) - f(\lambda, m)\right] + \frac{\sigma^2}{2N}\frac{\partial^2}{\partial \lambda^2} f(\lambda, m).
\end{aligned}
\end{equation}
The rest of this section on the mean field case is organized as follows. In the next two subsections we motivate the expected limit behavior at the two different timescales characterizing the model: in Section \ref{deterministic_mf} we deduce the order $1$ timescale \textit{deterministic} limit dynamics for $N\to +\infty$, while in Section \ref{fluctuations_mf} we introduce the problem of studying the accelerated dynamics at a timescale of order $N$. We finally address rigorously the convergence problem in the so-called \textit{subcritical} regime in Section \ref{sub_mfcase}, and in the \textit{supercritical} regime in Section \ref{sup_mfcase}.

\subsection{Deterministic mean field limit}
\label{deterministic_mf}
At times of order $1$, where the fluctuations terms (i.e.\! the terms which tend to $0$ for $N\to +\infty$ in the generator \eqref{eqn:generator})  become negligible for $N\gg0$, the dynamics of the system is well approximated by the following system of two ODEs
\begin{equation}
\label{eqn:limit_sys}
\begin{cases}
\dot{\lambda}(t)  = 2 \tanh(\beta \lambda(t)) - 2m(t) \\
\dot{m}(t) = 2 \tanh(\beta \lambda(t)) - 2m(t)\\
\lambda(0) = \lambda_0 \in \mathbb{R},\\
m(0) = m_0 \in [-1,1],
\end{cases}
\end{equation}
which represents the mean field limit of the dynamics introduced at the beginning. Sys. \eqref{eqn:limit_sys} is easily derived by observing that the generator \eqref{eqn:generator} uniformly converges to 
\begin{equation}
\label{eqn:lim_gen}
\mathcal{L} f(\lambda, m)  = (2 \tanh(\beta \lambda) - 2m)\left[\frac{\partial}{\partial \lambda} f(\lambda, m) + \frac{\partial}{\partial m} f(\lambda,m)\right].
\end{equation}
From the uniform convergence of the generators we obtain the weak convergence of the stochastic processes $(\lambda^N(t), m^N(t))_{t \in [0,T]}$ satisfying dynamics \eqref{eqn:lambda} and \eqref{eqn:m} to the limit deterministic process $(\lambda(t),m(t))_{t \in [0,T]}$, for which Sys. \eqref{eqn:limit_sys} holds (see \cite{ethier} for a classic reference).  Note that if we choose initial conditions such that $\lambda_0 = m_0$, the above system restricts to the Curie--Weiss model for $m(t) \equiv \lambda(t)$, except for a missing multiplicative term in the vector field which does not modify the qualitative behavior of the dynamics. System \eqref{eqn:limit_sys} is such that its equilibria form a one-dimensional curve of fixed points, given by
$$
m = \tanh{\beta \lambda},
$$  
corresponding to the points $(\lambda, m)$ for which $(\dot{\lambda}, \dot{m}) = (0,0)$. By studying the sign of the two-dimensional vector field in \eqref{eqn:limit_sys}, which has a constant slope of $1$ since its components are equal, one can get convinced that the equilibrium curve is a global attractor for the dynamics. However, we can distinguish two regimes, depending on the value of the parameter $\beta$.

\begin{figure}   
\centering   
\includegraphics[width=6cm]{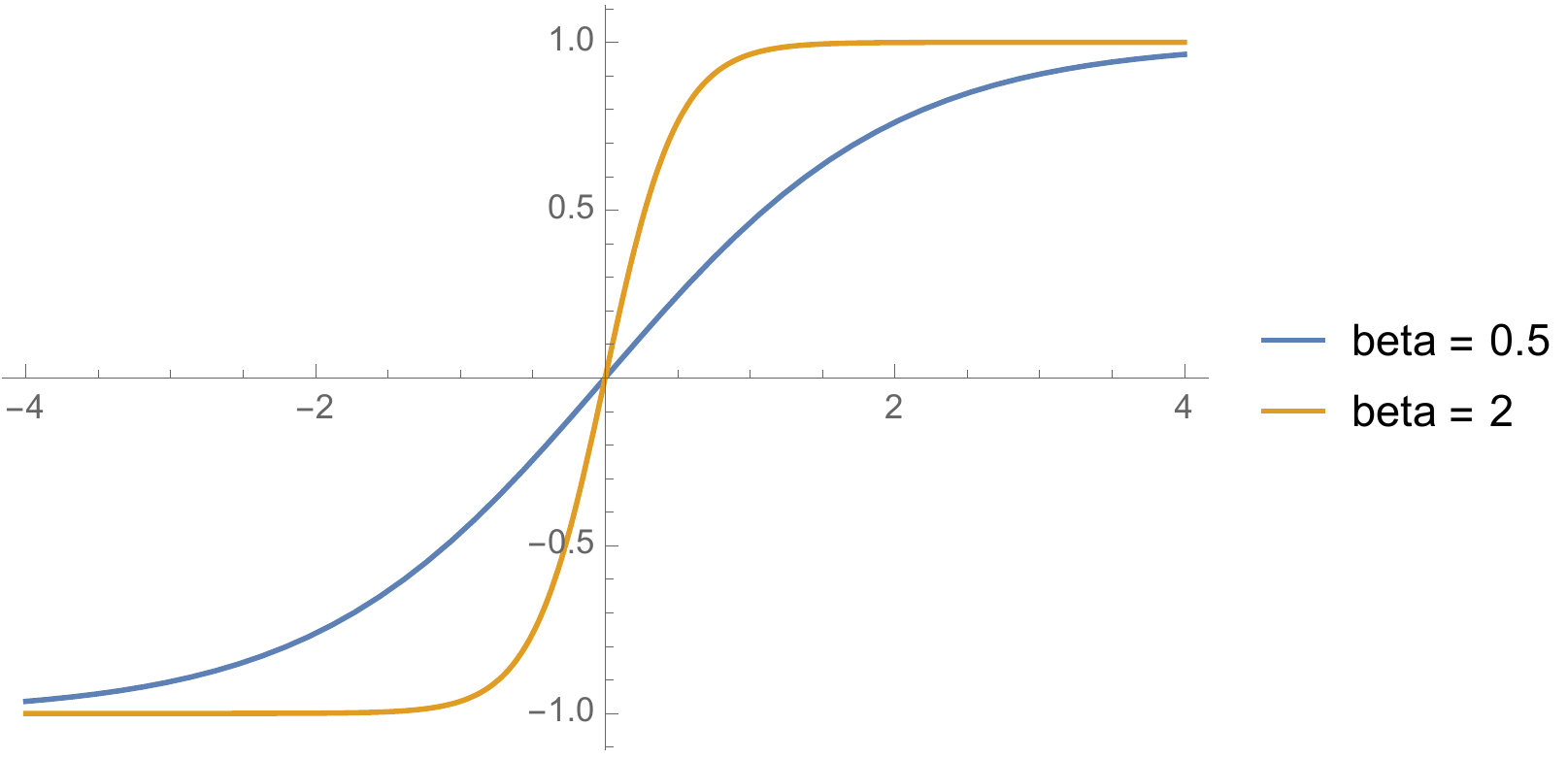}   
\caption{Invariant manifold for different values of $\beta$.}   
\label{figura1}   
\end{figure}   

Fig.\! \ref{figura1} and \ref{figura2} should highlight the qualitative behavior of the dynamics: for $\beta <1$, when the slope of the invariant curve is always smaller than the one of the vector field, the whole curve is a stable manifold; for $\beta > 1$ instead, the curve is stable in the two disjoint external intervals where the slope is less than $1$, while it shows an unstable behavior in the internal interval where the slope of the curve is greater than $1$. 
For $\beta > 1$, we denote the critical points where the curve has a slope equal to $1$ as $(\pm \lambda_a(\beta), \pm m_a(\beta))$, where
\begin{equation}
\label{eqn:critical_p}
\begin{aligned}
\lambda_a(\beta) & = \frac{1}{\beta}\text{arctanh}\left(\sqrt{1 - \frac{1}{\beta}}\right),\\
m_a(\beta) & = \sqrt{1 - \frac{1}{\beta}}.
\end{aligned}
\end{equation}

Thus, for some initial conditions close enough to the critical points, the dynamics will be soon attracted to the other branch of the curve, as shown in Fig.\! \ref{figura2}, where the vector field lines are also drawn in red. Consequently one can expect that, at the larger timescales where the diffusive smaller order terms are not negligible, the corresponding $N$-particle system might show an oscillating behavior between the two stable intervals, where the diffusion plays a role in driving the order parameters close enough to the endpoints of the stable intervals, thus determining a sudden change in the macroscopic variables.

\begin{figure}   
\centering   
\includegraphics[width=5.5cm]{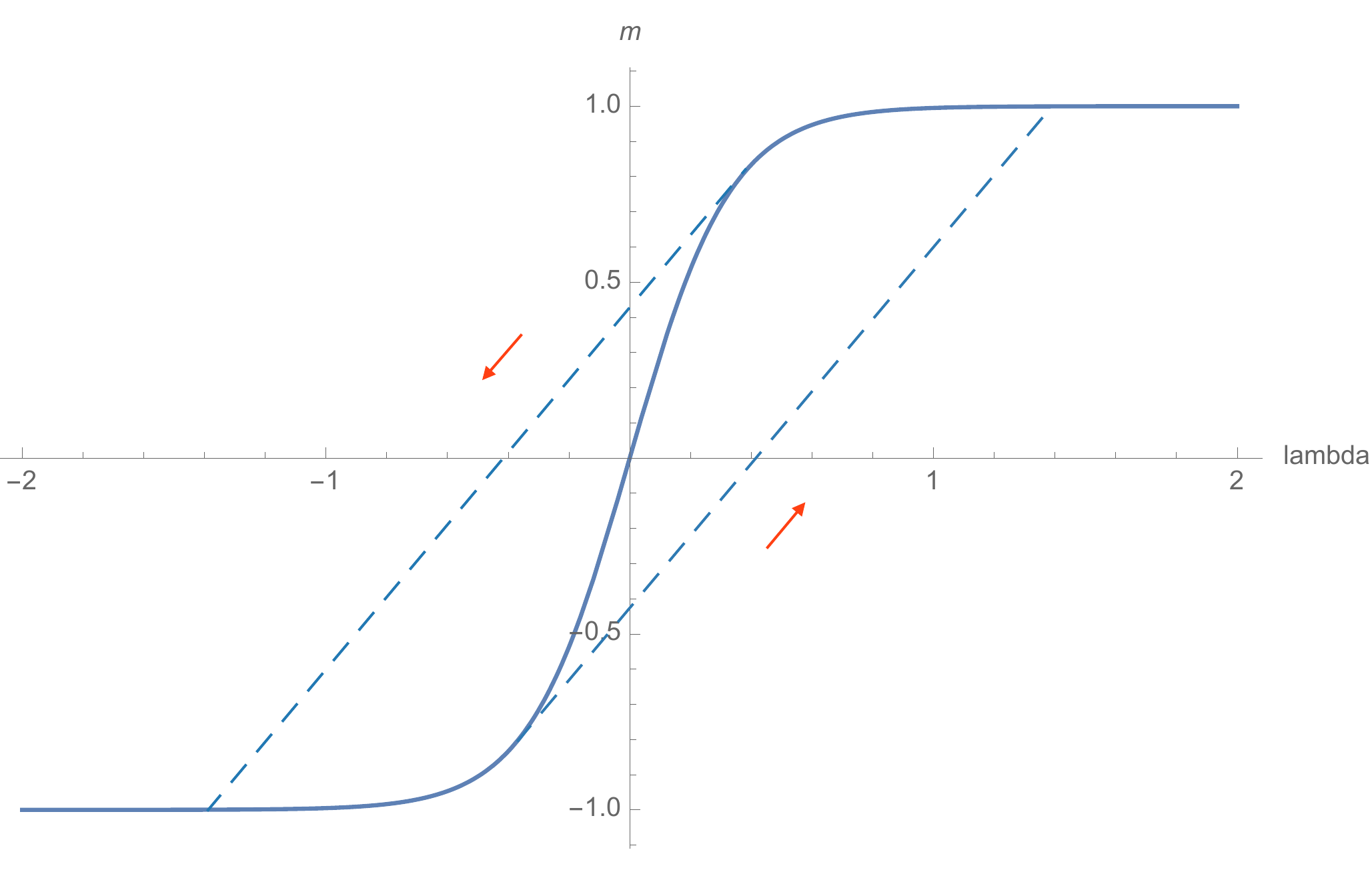}   
\caption{Qualitative behavior for $\beta > 1$.}   
\label{figura2}   
\end{figure}   

\subsection{Accelerated dynamics at times of order $N$}
\label{fluctuations_mf}
In order to rigorously understand the diffusive behavior with jumps which we qualitatively described in the previous section, we are led to study the accelerated $N$-particle dynamics and its relation with the order $1$ deterministic limit \eqref{eqn:limit_sys}, at a timescale where the diffusive smaller order terms are not negligible anymore in the limit $N \to +\infty$. In the following, for ease of notation, we still denote as $(\lambda^N(t),m^N(t))_{t\geq 0}$ the accelerated dynamics at a timescale of order $N$, i.e.\! $(\lambda^N(t),m^N(t)):=(\lambda^N(Nt),m^N(Nt))$, with the latter being the original process at a timescale of order $1$ (and the same notation $(x^N(t),m^N(t))$ for the alternative variables).
To motivate the presence of a limiting diffusive behavior at the accelerated timescale, we develop the jump terms in the generator \eqref{eqn:generator} at the second order, without considering the remainder terms of higher orders, yielding
\begin{equation}
\label{eqn:ap_gen}
\begin{aligned}
\mathcal{L}^N& f(\lambda, m) \approx (2 \tanh(\beta \lambda) - 2m)\left[\frac{\partial}{\partial \lambda} f(\lambda, m) + \frac{\partial}{\partial m} f(\lambda,m)\right]\\
& + \frac{1}{N}(2 - 2m\tanh(\beta\lambda))\left[2\frac{\partial^2}{\partial m\partial \lambda} f(\lambda, m) + \frac{\partial^2}{\partial \lambda^2} f(\lambda, m) + \frac{\partial^2}{\partial m^2}f(\lambda, m)\right]+ \frac{\sigma^2}{2N}\frac{\partial^2}{\partial \lambda^2}f(\lambda, m).
\end{aligned}
\end{equation}
The corresponding approximate dynamics features a strong drift, associated to the first order term in \eqref{eqn:ap_gen}, which grows with $N$ in both variables, and a bidimensional diffusion term which is of order $1$. The first fastly attracts the dynamics towards the curve $m = \tanh(\beta \lambda)$, on which the diffusive part then acts on a larger timescale. In the limit $N \to +\infty$, at an order $N$ timescale, one is then expecting to see an effective one-dimensional diffusive motion onto the curve $m = \tanh(\beta \lambda)$. Because of the difference in the stability properties for different values of the parameters, additional care must be put in the case $\beta > 1$, where one should retrieve a diffusive motion on the two stable intervals of the curve, $(-\infty, -\lambda_a(\beta))$ and $(\lambda_a(\beta),+\infty)$ (with $\lambda_a(\beta)$ as in \eqref{eqn:critical_p}), with jumps from one to the other component when the dynamics hits the critical points. Moreover, as we prove below, the arrival points of the jumps are also deterministic, and they are given by the intersection of the invariant curve with the tangent line passing through the critical points (see Fig.\! \ref{figura2}).  

An easy computation shows that our intuition is indeed correct: the accelerated $N$-particle \textit{exact} dynamics $(\lambda^N(t), m^N(t))_{t \geq 0}$ contracts the distance between $m$ and the invariant curve $\tanh(\beta \lambda)$, but only in the stable intervals $(-\infty, -\lambda_a(\beta))$ and $(\lambda_a(\beta),+\infty)$ when $\beta > 1$. Specifically, if we denote 
\begin{equation}
\label{eqn:y(t)}
y^N(t) := m^N(t)-\tanh(\beta \lambda^N(t)),
\end{equation}
we have the following
\begin{prop}
\label{mf_contr}
Let $y^N(t)$ be as in \eqref{eqn:y(t)}. Then, for any $T>0$, $k > 0$, $\beta < 1$,
\begin{equation}
\label{eqn:mf_contr}
\mathbb{E}\left[\sup_{t \in [0,T]} |y^N(t)|^k\right] \xrightarrow{N \to +\infty} 0.
\end{equation}
\end{prop}
\begin{proof}
If we apply the generator \eqref{eqn:generator} in the accelerated timescale to any power $k$ of the distance $|y^N(t)|$, we obtain
\begin{align*}
N \mathcal{L}^N |y^N(t)|^k & = N \mathcal{L}^N |m^N(t) - \tanh(\beta \lambda^N(t))|^k \\
&\leq - 2kN(m^N(t) - \tanh(\beta \lambda^N(t)))|m^N(t) - \tanh(\beta \lambda^N(t))|^{k-1} \times\\
& \times \text{sign}(m^N(t) - \tanh(\beta \lambda^N(t)))\left[-\frac{d}{d\lambda}(\tanh(\beta\lambda^N(t))) + 1\right] + O(1)\\
& = -2kN|m^N(t) - \tanh(\beta \lambda^N(t))|^k\left[-\beta(1-\tanh^2(\beta\lambda^N(t))) + 1\right] + O(1),
\end{align*}
where in the equality we have used $x \cdot \text{sign}(x) = |x|$. The $O(1)$ terms are estimated by exploiting the diffusive approximation \eqref{eqn:ap_gen}.
Observing that, for $\beta < 1$, the function $1-\beta(1-\tanh^2(\beta\lambda))$ has a global minimum in $0$ given by $1 - \beta$, we have found
\begin{equation}
\label{eqn:contracting}
N \mathcal{L}^N |y^N(t)|^k \leq - C(\beta,k)N|y^N(t)|^k + O(1),
\end{equation}
with $C(\beta,k) := 2k(1-\beta) > 0$. 
By definition of $\mathcal{L}^N$, \eqref{eqn:contracting} implies
\begin{equation*}
\frac{d}{dt} \mathbb{E}\Big[|y^N(t)|^k\Big] \leq - C(\beta,k)N\mathbb{E}\Big[|y^N(t)|^k\Big] + O(1),
\end{equation*}
which, integrating both sides gives
\begin{align*}
\mathbb{E}\Big[|y^N(t)|^k\Big] &\leq  e^{-C_1N t } \mathbb{E}\Big[|y^N(0)|^k\Big] -  \frac{C_2}{N} e^{-C_1N t} +  \frac{C_2}{N}.
\end{align*}
Thus, $\sup_{t \geq 0} \mathbb{E}\Big[|y^N(t)|^k\Big] \leq \mathbb{E}\Big[|y^N(0)|^k\Big] + \frac{C}{N}$. Note that by the assumptions on the initial data we have by a LLN that $\mathbb{E}\Big[|y^N(0)|^k\Big] \xrightarrow{N \to +\infty} 0$. 
For getting the stronger convergence \eqref{eqn:mf_contr} we refer to Section 4 of \cite{collet} for the diffusive case and to the Appendix of \cite{comets} for a  general proof for jump processes, where their results imply here that, for any $\delta > 0$,
$$
\mathbb{P}\left[\sup_{t \in [0,T]} |y^N(t)|^k > \delta \right] \xrightarrow{N \to +\infty} 0.
$$
Since $|y^N(t)|^k$ is uniformly bounded \eqref{eqn:mf_contr} follows.
\end{proof}

\begin{rem}
\label{beta>1_contracting}
For $\beta > 1$, when $\frac{d}{d\lambda}(\tanh(\beta\lambda)) < 1$ we can repeat the previous arguments to obtain an estimate as \eqref{eqn:contracting}. To be more precise, for any $\delta > 0$ we can find an $\varepsilon > 0$ such that $\frac{d}{d\lambda}\left[\tanh(\beta(\lambda_a(\beta) + \delta))\right] = \frac{d}{d\lambda}\left[\tanh(-\beta(\lambda_a(\beta) + \delta))\right]= 1- \varepsilon$, and $\frac{d}{d\lambda}\left[\tanh(\beta\lambda)\right] < 1-\varepsilon$ for any $\lambda \in (-\infty, -\lambda_a(\beta)-\delta) \cup (\lambda_a(\beta)+\delta,+\infty)$. Then, for any $(\lambda,m)$ satisfying the above conditions we have, denoting $y := m - \tanh(\beta \lambda)$,
\begin{equation}
\label{eqn:beta>1_contracting}
N \mathcal{L}^N |y|^k \leq - C(\delta,\beta, k,\varepsilon) N |y|^k + O(1),
\end{equation}
with $C(\delta,\beta,k,\varepsilon) > 0$ if and only if $\lambda \in (-\infty, -\lambda_a(\beta) - \delta) \cup (\lambda_a(\beta)+\delta,+\infty)$.
\end{rem}

\subsection{The subcritical case: $\beta < 1$}
\label{sub_mfcase}
In this section we employ the result of Proposition \ref{mf_contr} to obtain the convergence of the sequence of the accelerated processes $(\lambda^N(t), m^N(t))_{t \geq 0}$ to some limit random process $(\lambda(t),m(t))_{t \geq 0}$ in the subcritical case $\beta <1$. For convenience and coherence with the further analyses, we state the main result of the section (Proposition \ref{subcritical_chaos}) for the variables $(x^N(t),m^N(t))_{t \geq 0}$, whose infinitesimal accelerated generator can be obtained from \eqref{eqn:generator} through a change of coordinates and a multiplication by $N$. To be precise, $(x^N(t))_{t \geq 0}$ satisfies
\begin{equation}
\label{eqn:X}
\begin{cases}
dx^N(t) = \sigma dW^N(t),\\
x^N(0) \sim \mathcal{N}\left(x_0,\frac{1}{N}\sigma^2\right), 
\end{cases}
\end{equation}
with $W^N$ the Brownian motion $W^N(t) := \frac{1}{\sqrt{N}}\sum_{i=1}^N W_i(t)$, while $(m^N(t))_{t \geq 0}$ is given as in \eqref{eqn:m} but with rates multiplied by $N$, i.e.\!
\begin{equation}
\label{eqn:acc_jump}
\begin{cases}
m^N(t) \! \mapsto \!m^N(t) \pm \frac{2}{N}  \ \text{ rate }  \ \ N^2 \frac{1 \mp m^N(t)}{2}\! \left(1 \!\pm \tanh\!\left(\beta (x^N(t) + m^N(t))\right)\right),\\
m^N(0) = \frac{1}{N}\text{Bin}(N,p).
\end{cases}
\end{equation}
We show below that the limit process for the sequence $(x^N(t), m^N(t))_{t \geq 0}$ is given by 
\begin{equation}
\label{eqn:subcritical_mf}
\begin{cases}
m(t) = \tanh(\beta(x(t)+m(t))),\\
dx(t) = \sigma dW(t),\\
m(0) = m_0 \in [-1,1],\\
x(0) = x_0 \in \mathbb{R},
\end{cases}
\end{equation}
with $m_0 = 2p -1$ and $W$ a Brownian motion. 
In the subcritical case, Eq. \eqref{eqn:subcritical_mf} is well-posed. Indeed, for $\beta < 1$, the relation $m(t) = \tanh(\beta (x(t) + m(t)))$ can be made explicit so that $m(t) = \varphi(x(t))$ for some function $\varphi : \mathbb{R} \to [-1,1]$ (see also Proposition \ref{limit_sub_diff} below).

\begin{prop}[Subcritical order $N$ mean field limit dynamics]
\label{subcritical_chaos}
Let $T >0$ and $\beta < 1$. Then, $(x^N(t), m^N(t))_{t \in [0,T]}$ converges for $N \to +\infty$, in the sense of weak convergence of stochastic processes, to $(x(t),m(t))_{t \in [0,T]}$, the solution to \eqref{eqn:subcritical_mf}.
\end{prop}
\begin{proof}
We plug in the definition \eqref{eqn:X} of $x^N(t)$ the \textit{same} Brownian motion $W(t)$ appearing in the definition \eqref{eqn:subcritical_mf} of $x(t)$. We then prove, for the resulting processes
\begin{align}
\label{eqn:mlim}
\mathbb{E}\left[\sup_{t \in [0,T]} |m^N(t) - m(t)|\right] \xrightarrow{N \to +\infty} 0,\\
\label{eqn:xlim}
\mathbb{E}\left[\sup_{t \in [0,T]} |x^N(t) - x(t)|\right] \xrightarrow{N \to +\infty} 0.
\end{align}
Since $W^N \stackrel{\mathcal{D}}{=} W$ for every $N$, as they are both Brownian motions, \eqref{eqn:mlim} and \eqref{eqn:xlim} imply the desired convergence in distribution between the processes. 
Limit \eqref{eqn:xlim} is trivial, since the dynamics of $x^N(t)$ in the accelerated scale is
$$
x^N(t) = x^N(0) + \sigma \int_0^t dW(t),
$$
and $x^N(0) \to x(0)$ by a LLN. For \eqref{eqn:mlim}, we estimate
\begin{align*}
\mathbb{E}&\left[\sup_{t \in [0,T]}\Big|m^N(t) - m(t)\Big|\right] \leq \mathbb{E}\left[\sup_{t \in [0,T]}\Big|m^N(t) - \tanh(\beta(x^N(t) + m^N(t))) \Big|\right] \\
& \hspace{1cm} + \mathbb{E}\left[\sup_{t \in [0,T]} \Big|\tanh(\beta(x^N(t) + m^N(t))) - m(t)\Big|\right].
\end{align*}
The first term in the right hand side tends to $0$ thanks to \eqref{eqn:mf_contr} for $k=1$. For the second term, using Eq. \eqref{eqn:subcritical_mf} for $m(t)$, we have
\begin{align*}
\mathbb{E}&\left[\sup_{t \in [0,T]} \Big|\tanh(\beta(x^N(t) + m^N(t))) - m(t)\Big|\right] \\
& = \mathbb{E}\left[\sup_{t \in [0,T]} \Big|\tanh(\beta(x^N(t) + m^N(t))) - \tanh(\beta(x(t)+m(t)))\Big|\right]\\
& \leq \beta\mathbb{E}\left[\sup_{t \in [0,T]} |x^N(t) - x(t)|\right] + \beta \mathbb{E}\left[\sup_{t \in [0,T]} |m^N(t) - m(t)|\right],
\end{align*}
where in the inequality we have used the global Lipschitz continuity of $\tanh(\cdot)$.
Thus, recollecting the above estimates
\begin{align*}
(1-\beta)\mathbb{E}&\left[\sup_{t \in [0,T]}|m^N(t) - m(t)|\right] \leq \beta\mathbb{E}\left[\sup_{t \in [0,T]} |x^N(t) - x(t)|\right] \xrightarrow{N \to +\infty} 0.
\end{align*}
\end{proof}
We conclude this section by noting that, in the subcritical regime $\beta < 1$, we can furthermore obtain an explicit one-dimensional description of the limit process $m(t)$. Indeed, in the dynamics \eqref{eqn:subcritical_mf}, the only randomness is due to the diffusion $x(t)$, while $m(t)$ is slaved to be onto the invariant curve. A standard application of Itô's formula shows that
\begin{prop}[Limit diffusion]
\label{limit_sub_diff}
The process $(m(t))_{t \geq 0}$ defined in \eqref{eqn:subcritical_mf} is a strong solution to 
\begin{equation}
\label{eqn:limit_sub_diff}
\begin{cases}
dm(t) = - \frac{\beta^2 \sigma^2 m(t) \left(1-m^2(t)\right)}{\left(1-\beta(1-m^2(t))\right)^3}dt + \frac{\sigma \beta (1-m^2(t))}{1-\beta(1-m^2(t))}dW(t),\\
m(0) = m_0 \in [-1,1].
\end{cases}
\end{equation}
\end{prop}
\begin{proof}
By Eq. \eqref{eqn:subcritical_mf}, $m(t)$ can be written as an \textit{explicit} function of $x(t)$, and thus its dynamics must be of the form
$$
d m(t) = a(t,m(t))dt + b(t,m(t)) d W(t)
$$
for some functions $a,b : [0,\infty) \times [-1,1] \to \mathbb{R}$ to be determined, and $W(t)$ is the same Brownian motion appearing in the dynamics of $x(t)$.
By applying Itô's formula to the function $\tanh(\beta(x(t) + m(t))$, we find
\begin{align*}
dm(t) &= d\left\{\tanh{\beta(x(t) + m(t))}\right\} \\
& = \beta [1-\tanh^2{\beta(x(t) + m(t))}] (dx(t) + dm(t))\\
& - \beta^2\tanh{\beta(x(t)+m(t))}[1-\tanh^2{\beta(x(t) + m(t))}] (b(t,m(t)) + \sigma)^2dt\\
& = \big[\beta(1-m^2(t))a(t,m(t)) - \beta^2 m(t)(1-m^2(t))(b(t,m(t)) + \sigma)^2\big] dt \\
& + \beta(1-m^2(t))\big[\sigma + b(t,m(t))\big] dW(t).
\end{align*}
By reading the diffusion coefficient from the last line, we must have
$$
b(t,m(t)) = \beta(1-m^2(t))[\sigma + b(t,m(t))],
$$
and thus
$$
b(t,m(t)) = b(m(t)) = \frac{\sigma \beta(1-m^2(t))}{1-\beta(1-m^2(t))}.
$$
For the drift term instead
\begin{align}
\label{eqn:aux}
a(t,m(t)) = \beta(1-m^2(t))a(t,m(t)) - \beta^2 m(t)(1-m^2(t))[(b(t,m(t))+\sigma)^2].
\end{align}
Using the expression found for $b(t,m(t))$, we have that 
\begin{align*}
(b(t&,m(t))+\sigma)^2  = b^2(t,m(t)) + \sigma^2 + 2 \sigma b(t,m(t))  = \frac{\sigma^2}{(1-\beta(1-m^2(t)))^2},
\end{align*}
and thus, reading from \eqref{eqn:aux},
\begin{align*}
a(t,m(t))(1-\beta(1-m^2(t))) = - \beta^2 m(t) (1-m^2(t))\frac{\sigma^2}{(1-\beta(1-m^2(t)))^2},
\end{align*}
so that we can conclude.
\end{proof}

\begin{rem}
\label{subcritic_diffusion_is_well-posed}
For $\beta < 1$, the SDE \eqref{eqn:limit_sub_diff} is well-posed. Existence follows by Proposition \ref{limit_sub_diff}. Uniqueness follows by the Lipschitz properties of the drift and diffusion functions in $[-1,1]$. Indeed, note that Eq. \eqref{eqn:limit_sub_diff} defines a dynamics in $[-1,1]$, due to the sign of the drift at the borders of $(-1,1)$ and to the fact that the diffusion is zero at the borders of $(-1,1)$. 
\end{rem}

\subsection{The supercritical case: $\beta > 1$}
\label{sup_mfcase}
In this section we deal with the analysis of the supercritical case $\beta > 1$. The main result is the following convergence theorem:
\begin{thm}[Supercritical order $N$ mean field limit dynamics]
\label{thm1}
Fix $T >0$, $\beta > 1$, and let $(x^N(t),m^N(t))_{t \in [0,T]}$ be the accelerated processes defined in \eqref{eqn:X} and \eqref{eqn:acc_jump}, with $x^N(0)\stackrel{\mathcal{D}}{\to}  x_0 > \lambda_a(\beta) - m_a(\beta)$ and $m^N(0) \stackrel{\mathcal{D}}{\to} m_0 > m_a(\beta)$, or $x^N(0)\stackrel{\mathcal{D}}{\to}  x_0 < m_a(\beta) - \lambda_a(\beta)$ and $m^N(0)\stackrel{\mathcal{D}}{\to} m_0 < - m_a(\beta)$, with $(\lambda_a(\beta), m_a(\beta))$ as in \eqref{eqn:critical_p}.
Then, the accelerated sequence of processes $(m^N(t))_{t \in [0,T]}$ converges weakly in the sense of stochastic processes, for $N \to +\infty$, to the process which solves the following SDE 
\begin{align}
\label{eqn:lei}
d m(t) &= \mathbbm{1}_{|m(t)| > m_a} \left(- \frac{\beta^2 \sigma^2 m(t) \left(1-m^2(t)\right)}{\left(1-\beta(1-m^2(t))\right)^3}dt + \frac{\sigma \beta (1-m^2(t))}{1-\beta(1-m^2(t))}dW(t)\right) \\
&+  (m_b+m_a)\mathbbm{1}_{m(t) = - m_a} - (m_b+m_a)\mathbbm{1}_{m(t) = m_a},\nonumber
\end{align} 
with $m(0) = m_0$, and $m_b:= m_b(\beta)$ is the solution in $y$ to
\begin{equation}
\label{eqn:correct}
g(y):= 2\beta y - 2\beta(m_a(\beta) - \lambda_a(\beta)) - \log(1+y) + \log(1-y) = 0.
\end{equation}
\end{thm}
We want to derive a limit one-dimensional diffusion for each variable, which also contains the jump components illustrated in Fig.\! \ref{figura2} for $\beta >1$. 
As highlighted in Remark \ref{beta>1_contracting}, in this case the $N$-particle dynamics is contractive only in the union of the two intervals where $1-\beta(1-\tanh^2(\beta \lambda)) > 0$, i.e.\! for $\lambda > \lambda_a(\beta)$ or $\lambda < -\lambda_a(\beta)$, which we refer to as the \textit{stable} components of the invariant curve. As long as the dynamics does not hit the critical points, we expect the limit evolution to be the same as for the subcritical case in Proposition \ref{limit_sub_diff}. Note that both the drift and diffusion coefficients explode at the critical points of the invariant curve. In fact, when the dynamics hits the critical points, we expect to see an \textit{instantaneous} jump to the point given by the intersection between the vector field line passing through the critical point and the invariant curve.

Denoting with $c(\cdot)$ the drift function and with $\sqrt{g(\cdot)}$ the diffusion coefficient, we get that the global limiting accelerated one-dimensional dynamics, written in either of the two variables $m(t)$ or $\lambda(t)$, should be of the form
\begin{equation}
\label{eqn:lim}
d X(t) = \mathbbm{1}_{|X(t)| > a} \left(\sqrt{g(X(t))} dW(t) + c(X(t))dt \right)+ (b+a)\mathbbm{1}_{X(t) = - a} - (b+a)\mathbbm{1}_{X(t) = a}, 
\end{equation}
where the point $a = a(\beta)$ is the critical (positive) point on the invariant curve, and the point $b = b(\beta)$ (resp.\! $-b$) is the intersection between the curve and the vector field line passing through $-a$ (resp.\! $a$).

\begin{rem}[Limit case $\beta \to \infty$]
\label{limiting_case_b=infty}
When $\beta \to \infty$, the limit dynamics for the accelerated magnetization $(m(t))_{t \geq 0}$ is expected to be a spin-valued jump process $m(t) \in \left\{-1,1\right\}$  with non-exponentially distributed random interarrival jump times, with their distribution being the one of the hitting times of a Brownian motion with diffusion coefficient $\sigma > 0$. Indeed, the critical points (see Eq.\! \eqref{eqn:critical_p}) tend to $\pm 1$ in the $m$-variable, and to $0$ in the $\lambda$-variable, while the diagonal line $x(t) = \lambda(t) - m(t)$, determining when the process jumps, still evolves according to a Brownian motion with diffusion coefficient $\sigma > 0$.
\end{rem}

\subsubsection{The convergence argument}
\label{convergence_argument}
We now address the full proof of convergence to the limit dynamics for $\beta > 1$, given in Theorem \ref{thm1}. As we did above for the subcritical case, we consider the dynamics in the alternative variables $(x^N,m^N)$. Recall that the variable $x^N$, the intersection between the diagonal line (at $45$ degrees) passing through the point $(\lambda^N, m^N)$ and the $\lambda$-axis, follows a Brownian motion, while $m^N$ is a jump process depending on $x^N$: if we think of the latter as being deterministic and fixed, such motion is a unidimensional continuous-time Markov chain on the diagonal line parametrized by the fixed value $x^N = x$, which is attractive towards the invariant curve.
The limit dynamics is thus the projection of the combination of these two motions on the invariant curve. 
We divide the proof of Theorem \ref{thm1} in three lemmas. In the following proofs we assume that $x_0 > \lambda_a - m_a$ and $m_0 > m_a$. For the symmetry of the problem the case $x_0 < m_a - \lambda_a$ and $m_0 < -m_a$ is analogous.

\begin{lem}
\label{lemma_x}
Set 
\begin{equation}
\label{eqn:T_ma^e}
T_{m_a}^{\varepsilon} := \inf\left\{t \geq 0 : x^N(t) =\lambda_a- m_a - \varepsilon\right\},
\end{equation}
for $\varepsilon \in \mathbb{R}$.
Then,
\begin{equation}
\label{eqn:just_for_fin}
\mathbb{P}(T_{m_a}^{\varepsilon} < \infty) = 1.
\end{equation}
\end{lem}
\begin{proof}
Recall that $x^N$ evolves as in \eqref{eqn:X}. For the proof, we assume for simplicity that $x^N(0) = x_0$ (otherwise, we just add an additional term in the variance at time $t$, accounting for the initial variance - which is small in $N$).We thus have that $x^N(t) \sim \mathcal{N}\left(0, \sigma^2 t\right)$, and we can get explicitly the distribution of $T_{m_a}^{\varepsilon}$ in a classic way, using the reflection principle for the Brownian motion. Indeed, we have that for any $t \geq 0$, 
\begin{align*}
\mathbb{P}(T_{m_a}^{\varepsilon} \leq t) & = \mathbb{P}(\inf_{0 \leq s \leq t}x^N(s) \leq \lambda_a - m_a - \varepsilon) = 2 \mathbb{P}(x^N(t) \leq \lambda_a - m_a - \varepsilon)\\
& = \frac{2}{\sqrt{2 \pi \sigma^2 t}}\int_{-\infty}^{\lambda_a - m_a - \varepsilon}e^{-\frac{(x-x_0)^2}{2\sigma^2t}} dx = \left[z = \frac{x-x_0}{\sqrt{t}}\right]\\
& = \frac{2}{\sqrt{2\pi\sigma^2}}\int_{-\infty}^{\frac{\lambda_a - m_a - \varepsilon - x_0}{\sqrt{t}}} e^{-\frac{z^2}{2\sigma^2}} dz.
\end{align*}
By taking the derivative with respect to $t$ of the previous expression we get that $T_{m_a}^{\varepsilon}$ has density
\begin{equation*}
f_{T_{m_a}^{\varepsilon}}(t) = \frac{(\lambda_a - m_a - \varepsilon - x_0)}{\sqrt{2\pi\sigma^2}}\frac{1}{t^{3/2}}e^{-\frac{(\lambda_a - m_a - \varepsilon - x_0)^2}{2\sigma^2t}},
\end{equation*}
and, as one can check
\begin{equation*}
\mathbb{P}(T_{m_a}^{\varepsilon} < \infty) = \int_{0}^{\infty} f_{T_{m_a}^{\varepsilon}}(t) dt = 1,
\end{equation*}
so that \eqref{eqn:just_for_fin} is verified.
\end{proof}
Lemma \ref{lemma_x} tells us that, almost surely, the process $(x^N(t))_{t \geq 0}$ reaches in a finite time the point $\lambda_a- m_a - \varepsilon$, which corresponds - up to an $\varepsilon$ error - to the critical point on the invariant curve we discussed in the previous section.
The following two lemmas respectively describe the limit equation for the times preceding and following the hitting time $T_{m_a}^\varepsilon$. For $t < T_{m_a}^{-\delta}$, for some $\delta > 0$, we can proceed similarly as in Propositions \ref{subcritical_chaos} and \ref{limit_sub_diff} since the contraction estimates of Remark \ref{beta>1_contracting} are holding, while for $t > T_{m_a}^{\varepsilon}$ for some $\varepsilon > 0$ we capture the jumps via a direct estimate. We then conclude by the continuity with respect to $\varepsilon$ and $\delta$ of the hitting times distributions $T_{m_a}^\varepsilon$, $T_{m_a}^{-\delta}$.

\begin{lem}
\label{lem_small}
Fix $T,\delta > 0$. Set $T_{m_a}^{-\delta}:=\inf\left\{t \geq 0 : x^N(t) =\lambda_a- m_a + \delta \right\}$, and let $\Big(m^N(t \wedge T_{m_a}^{-\delta})\Big)_{t \in [0,T]}$ denote the accelerated stopped process, with initial conditions as in Theorem \ref{thm1}. Then, $\Big(m^N(t \wedge T_{m_a}^{-\delta})\Big)_{t \in [0,T]}$ converges weakly in the sense of stochastic processes, for $N \to +\infty$, to $\Big(m(t\wedge T_{m_a}^{-\delta})\Big)_{t \in [0,T]}$, with $(m(t))_{t \geq 0}$ the solution to \eqref{eqn:limit_sub_diff} with the same initial conditions as in Theorem \ref{thm1}, and $\Big(m(t\wedge T_{m_a}^{-\delta})\Big)_{t \in [0,T]}$ its stopped version.
\end{lem}
\begin{proof}
As in the proof of Proposition \ref{subcritical_chaos}, we plug in the definition \eqref{eqn:X} of $x^N(t)$ the \textit{same} Brownian motion $W(t)$ appearing in the definition \eqref{eqn:subcritical_mf} of $x(t)$. Let $T_{m_a}^{-\delta}$ be the resulting stopping time: we prove, 
\begin{equation}
\label{eqn:limit_for_m^N_stopped}
\mathbb{E}\left[\sup_{t \in [0,T]} \Big|m^N(t \wedge T_{m_a}^{-\delta}) - m(t\wedge T_{m_a}^{-\delta})\Big|\right] \xrightarrow{N \to +\infty} 0,
\end{equation}
which implies the result in distribution by reasoning as in Proposition \ref{subcritical_chaos}.
When $t < T_{m_a}^{-\delta}$ we have that $x^N(t) > \lambda_a - m_a + \delta$. Thus, we are in the stable component of the invariant curve.

From \eqref{eqn:xlim}, it follows that 
\begin{equation}
\label{eqn:limit_for_x^N_stopped}
\mathbb{E}\left[\sup_{t \in [0,T]} \Big|x^N(t \wedge T_{m_a}^{-\delta}) - x(t\wedge T_{m_a}^{-\delta})\Big|\right] \xrightarrow{N \to +\infty} 0.
\end{equation}
For \eqref{eqn:limit_for_m^N_stopped}, denoting the event $A:=\left\{\min_{t \in [0,T]}\lambda^N(t \wedge T_{m_a}^{-\delta}) > \lambda_a + \delta\right\}$, we estimate,
\begin{align}
\label{eqn:aux_stopped}
\mathbb{E}&\left[\sup_{t \in [0,T]}\Big|m^N(t\wedge T_{m_a}^{-\delta}) - m(t\wedge T_{m_a}^{-\delta})\Big|\right]= \mathbb{E}\left[\sup_{t \in [0,T]}\Big|m^N(t\wedge T_{m_a}^{-\delta}) - m(t\wedge T_{m_a}^{-\delta})\Big|\mathbbm{1}_A\right]\\
& \hspace{0.8cm}+ \mathbb{E}\Bigg[\sup_{t \in [0,T]}\Big|m^N(t\wedge T_{m_a}^{-\delta})- m(t\wedge T_{m_a}^{-\delta})\Big|\mathbbm{1}_{\exists t \in[0,T]  :   \lambda^N(t \wedge T_{m_a}^{-\delta}) < \lambda_a + \delta}\Bigg]\nonumber\\
&  \leq \mathbb{E}\Bigg[\sup_{t \in [0,T]}\Big|m^N(t\wedge T_{m_a}^{-\delta}) - \tanh(\beta(x^N(t\wedge T_{m_a}^{-\delta}) + m^N(t\wedge T_{m_a}^{-\delta}))) \Big|\mathbbm{1}_A\Bigg] \nonumber \\
&\hspace{0.8cm}+ \mathbb{E}\Bigg[\sup_{t \in [0,T]} \Big|\tanh(\beta(x^N(t\wedge T_{m_a}^{-\delta}) + m^N(t\wedge T_{m_a}^{-\delta}))) - m(t\wedge T_{m_a}^{-\delta})\Big|\mathbbm{1}_A\Bigg].\nonumber\\
& \hspace{1.6cm}+ 2 \mathbb{P}\Big(\exists t \in[0,T]  :  \lambda^N(t \wedge T_{m_a}^{-\delta}) < \lambda_a + \delta\Big) \nonumber,
\end{align}
where in the last line we have used the boundedness of the integrands.
The first term in the right hand side of the above inequality tends to $0$ thanks to estimate \eqref{eqn:beta>1_contracting} of Remark \ref{beta>1_contracting}  for $k = 1$, which can be applied for any $\lambda > \lambda_a + \delta$, and to the same argument used for the proof of Proposition \ref{mf_contr}. 
For the second term in the right hand side of inequality \eqref{eqn:aux_stopped}, using Eq. \eqref{eqn:subcritical_mf} for $m(t \wedge T_{m_a}^{-\delta})$, we have
\begin{align*}
\mathbb{E}&\left[\sup_{t \in [0,T]} \Big|\tanh(\beta(x^N(t\wedge T_{m_a}^{-\delta} ) + m^N(t\wedge T_{m_a}^{-\delta}))) - m(t\wedge T_{m_a}^{-\delta})\Big|\mathbbm{1}_A\right]\\
& = \mathbb{E}\Bigg[\sup_{t \in [0,T]} \Big|\tanh(\beta(x^N(t\wedge T_{m_a}^{-\delta}) + m^N(t\wedge T_{m_a}^{-\delta}))) - \tanh(\beta(x(t\wedge T_{m_a}^{-\delta})+m(t\wedge T_{m_a}^{-\delta})))\Big|\mathbbm{1}_A\Bigg]\\
& \leq (1-\varepsilon)\mathbb{E}\!\left[\sup_{t \in [0,T]}\! \Big|x^N(t\wedge T_{m_a}^{-\delta}) - x(t\wedge T_{m_a}^{-\delta})\Big|\right] \!+ (1-\varepsilon) \mathbb{E}\!\left[\sup_{t \in [0,T]} \! \Big|m^N(t\wedge T_{m_a}^{-\delta}) - m(t\wedge T_{m_a}^{-\delta})\Big|\right],
\end{align*}
where in the first inequality we have used that, by the properties of $\tanh(\cdot)$ and by definition of $\lambda_a$, there exists an $\varepsilon > 0$ such that $\frac{d}{d\lambda} \tanh(\beta \lambda) < 1-\varepsilon$ for every $\lambda > \lambda_a + \delta$.
Finally, the third term in the right hand side of \eqref{eqn:aux_stopped} can be estimated as follows
\begin{align}
\label{eqn:estimate_lambda_process}
2 \mathbb{P}&\Big(\exists t \in[0,T] : \lambda^N(t \wedge T_{m_a}^{-\delta}) < \lambda_a + \delta\Big)= 2\mathbb{P}\Big(\exists t \in[0,T] : x^N(t \wedge T_{m_a}^{-\delta}) + m^N(t \wedge T_{m_a}^{-\delta}) < \lambda_a + \delta\Big)\\
& = 2\mathbb{P}\Big(\exists t \in[0,T] : x^N(t \wedge T_{m_a}^{-\delta})  < \lambda_a - m^N(t \wedge T_{m_a}^{-\delta}) + \delta\Big)\nonumber\\
&\leq 2\mathbb{P}\Big(\exists t \in[0,T] : m^N(t \wedge T_{m_a}^{-\delta})  < m_a\Big)\nonumber,
\end{align}
where the inequality follows by the definition of $T_{m_a}^{-\delta}$.
To bound the latter, we introduce an auxiliary process $(\tilde{m}^N(t))_{t\in[0,T]}$, coupled with $(x^N(t),m^N(t))_{t \in [0,T]}$, with dynamics
\begin{equation*}
\begin{cases}
\tilde{m}^N(t) \! \mapsto \!\tilde{m}^N(t) \pm \frac{2}{N}  \ \text{ rate }  \ \ N^2 \frac{1 \mp \tilde{m}^N(t)}{2}\! \left(1 \!\pm \tanh\!\left(\beta (\lambda_a - m_a + \delta + \tilde{m}^N(t))\right)\right),\\
\tilde{m}^N(0) =m^{N}(0),
\end{cases}
\end{equation*}
and consider its stopped version $\Big(\tilde{m}^N(t \wedge T_{m_a}^{-\delta})\Big)_{t \in [0,T]}$. Since, by definition of $T_{m_a}^{-\delta}$, it holds $x^N(t \wedge T_{m_a}^{-\delta}) \geq \lambda_a - m_a + \delta$, we have that the rate of increase of $m^N(t \wedge T_{m_a}^{-\delta})$ is bigger than the rate of increase of $\tilde{m}^N(t \wedge T_{m_a}^{-\delta})$; symmetrically, the rate of decrease of $m^N(t \wedge T_{m_a}^{-\delta})$ is smaller than the rate of decrease of $\tilde{m}^N(t \wedge T_{m_a}^{-\delta})$. We thus have, for any $ t \in [0,T]$, $N \in \mathbb{N}$, $\overline{m} \in [-1,1]$,
\begin{equation}
\label{eqn:obs_tilde_aux}
\mathbb{P}\Big(m^N(t \wedge T_{m_a}^{-\delta}) < \overline{m}\Big) \leq \mathbb{P}\Big(\tilde{m}^N(t \wedge T_{m_a}^{-\delta}) < \overline{m}\Big).
\end{equation}
Moreover, note that $\tilde{m}^N(t \wedge T_{m_a}^{-\delta})$ is a jump process with rates independent of $x^N$, starting above $m_a$ with probability tending to $1$ for $N \to +\infty$, and that it gets fastly attracted, for $N \to +\infty$, to the point $m^*$ on the invariant curve identified by
\begin{equation*}
\begin{cases}
x = \lambda_a - m_a + \delta,\\
m = \tanh(\beta(x + m)),
\end{cases}
\end{equation*}
for which it holds by construction $m_a < m^*$. Thus $\mathbb{P}\Big(\exists t \in[0,T] : \tilde{m}^N(t \wedge T_{m_a}^{-\delta}) < m_a\Big) \leq C(N)$, with $C(N) \xrightarrow{N\to +\infty} 0$. By the above observation \eqref{eqn:obs_tilde_aux}, this implies the same bound for $m^N$ in the last line of the right hand side of \eqref{eqn:estimate_lambda_process}.
Finally, recollecting the above estimates from \eqref{eqn:aux_stopped},
\begin{align*}
\mathbb{E}&\left[\sup_{t \in [0,T]}\Big|m^N(t\wedge T_{m_a}^{-\delta}) - m(t\wedge T_{m_a}^{-\delta})\Big|\right] \leq \frac{(1-\varepsilon)}{\varepsilon}\mathbb{E}\left[\sup_{t \in [0,T]} \Big|x^N(t\wedge T_{m_a}^{-\delta}) - x(t\wedge T_{m_a}^{-\delta})\Big|\right] \\
&+ C(N) \leq \frac{C(N)}{\varepsilon} \xrightarrow{N \to +\infty} 0.
\end{align*}
\end{proof}
The next lemma deals with the times which follow the hitting time $T_{m_a}^{\varepsilon}$. Using the strong Markov's property, we can restart the dynamics from the point reached at the hitting time, assuming that we are above the invariant curve.
\begin{lem}
\label{lem_jump}
Fix $\varepsilon > 0$ such that $\lambda_a - \varepsilon > 0$. Let $(x^N(t), m^N(t))_{t \geq 0}$ be the accelerated processes, with initial data $(x^N(0), m^N(0)) = (x_0, m_0)$, such that $x_0 = \lambda_a- m_a - \varepsilon$ and $m_0 > \tanh{\beta(x_0 + m_a)}$. 
Let 
$$
T_{\varepsilon/2} := \inf\left\{t > 0 : x^N(t) = \lambda_a- m_a - \frac{\varepsilon}{2}\right\}, \quad T_{m_b} := \inf\left\{t> 0 : m^N(t) \leq m_b\right\},
$$ 
with $m_b$ as in \eqref{eqn:correct}.
Then,
\begin{equation}
\label{eqn:jump}
\lim_{N\to \infty}\mathbb{P}(T_{m_b} < T_{\varepsilon/2}) = 1.\\
\end{equation}
\end{lem}
\begin{proof}
The proof makes extensive use of $(\nu^N(t))_{t \geq 0 }$, an auxiliary CTMC - coupled with $(m^N(t))_{t \geq 0}$ - with the same initial datum $m_0$, whose transition rates are given by
\begin{align}
\label{eqn:nu}
\nu^N \mapsto & \ \nu^N + \frac{2}{N} \ \ \ \ \text{      with rate     }\ \ \ \ N^2 \frac{1 - \nu^N(t)}{2} \left[1 + \tanh({\beta (\nu^N(t) + \lambda_a - m_a - \varepsilon/2))}\right]\\
\nu^N \mapsto & \ \nu^N - \frac{2}{N} \ \ \ \ \text{      with rate     } \ \ \ \  N^2 \frac{1+\nu^N(t)}{2} \left[1 - \tanh({\beta (\nu^N(t) + \lambda_a - m_a - \varepsilon/2))}\right].\nonumber
\end{align}
Note that $(\nu^N(t))_{t \geq 0 }$ is independent of $(x^N(t))_{t \geq 0}$. Setting $\tilde{T}_{m_b} := \inf\left\{t > 0 : \nu^N(t) \leq m_b\right\}$, we have
\begin{equation}
\label{eqn:ineq}
\mathbb{P}(T_{m_b} < T_{\varepsilon/2}) \geq \mathbb{P}(\tilde{T}_{m_b} < T_{\varepsilon/2}).
\end{equation}
Indeed, it is easy to check that for $t \leq T_{\varepsilon/2}$, for which $x^N(t) \leq \lambda_a - m_a - \varepsilon/2$, the rate of increase in the dynamics of $\nu^N(t)$ is greater than that of $m^N(t)$, while the opposite is true for the rate of decrease. Since $m_b < m_0$, \eqref{eqn:ineq} follows.
Consider now the slowed version of the process $\nu^N(t)$, i.e.\! $\tilde{\nu}^N(t) := \nu^N(tN^{-1})$, whose generator is
\begin{align*}
\mathcal{L}^Nf(\tilde{\nu}) &:= N \frac{1+\tilde{\nu}}{2} \left[1 - \tanh({\beta (\tilde{\nu} + \lambda_a - m_a - \varepsilon/2))}\right]\left[f\left(\tilde{\nu} - \frac{2}{N}\right) - f(\tilde{\nu})\right] \\
& + N \frac{1 - \tilde{\nu}}{2} \left[1 + \tanh({\beta (\tilde{\nu} + \lambda_a - m_a - \varepsilon/2))}\right]\left[f\left(\tilde{\nu}+\frac{2}{N}\right) - f(\tilde{\nu})\right].\nonumber
\end{align*}
Expanding it to the first order, we find, up to terms of order $O\left(\frac{1}{N}\right)$,
\begin{align*}
\mathcal{L}^Nf(\tilde{\nu})  \approx \left[- 2 \tilde{\nu} + 2\tanh(\beta(\tilde{\nu} + \lambda_a - m_a - \varepsilon/2)) \right]f'(\tilde{\nu}).
\end{align*}
This implies that, in the limit $N \to +\infty$, the process $(\tilde{\nu}^N(t))_{t \geq 0}$ weakly converges to the solution of the following ODE 
\begin{equation}
\label{eqn:approx_lim}
\begin{cases}
\frac{d}{dt}\overline{m}(t) = v(\overline{m})= - 2 \overline{m}(t) + 2\tanh(\beta(\overline{m}(t) + \lambda_a - m_a - \varepsilon/2)) \\
\overline{m}(0) = m_0.
\end{cases}
\end{equation}
The vector field $v(m)$ in \eqref{eqn:approx_lim} is positive if and only if
\begin{equation}
\label{eqn:ap}
f(m) := 2\beta m - 2\beta(m_a - \lambda_a) - \log(1+m) + \log(1-m) - \beta\varepsilon  > 0.
\end{equation}
Indeed, $v(m)$ is positive if and only if $m < \tanh(\beta(m + \lambda_a - m_a - \varepsilon/2))$, which is equivalent to 
$$
\frac{1}{\beta}\arctanh(m) < m + \lambda_a - m_a - \varepsilon/2.
$$
By using the identity $\arctanh(m) = \frac{1}{2} \log\left(\frac{1+m}{1-m}\right)$ we get the desired inequality \eqref{eqn:ap}. Analogous steps motivate the expression for $g$ given in Theorem \ref{thm1}, obtained for $\varepsilon = 0$, which we recall for the ease of the reader
$$
g(y):= 2\beta y - 2\beta(m_a(\beta) - \lambda_a(\beta)) - \log(1+y) + \log(1-y) = 0.
$$
Recall that by our choice $m_0 > 0$. First of all, it is easy to see that $f(m) < 0$ whenever $m \geq 0$. Indeed, $f(0) < 0$, $f$ has a local maximum in $m = m_a= \sqrt{1 - \frac{1}{\beta}}$ for which $f\left(\sqrt{1 - \frac{1}{\beta}}\right) = -\beta \varepsilon < 0$, and $f(m) \to -\infty$ for $m \to +1$. Moreover, we have that $f(m) = g(m) - \beta\varepsilon$, so that $f(m) < g(m)$ for all $m \in [-1,1]$. Since $f'(m) = g'(m) = 2\beta - \frac{1}{1+m} - \frac{1}{1-m}$ we have that $g$ has a local maximum at $m = m_a$, for which we have $g(m_a) = 0$, while $g(m) < 0$ for all $m > 0, \ m \neq m_a$.
We also observe that:
\begin{itemize}
\item $\exists! \ m^*_{f,b}$ such that $f(m^*_{f,b}) = 0$;
\item $g(m_b) = 0$ and $g(m) \neq 0 \ \ \forall m \neq m_a, m_b$; 
\item $g(m) > 0$ if $ m < m_b$, $g(m) < 0$ if $m > m_b$;
\item $f(m) > 0$ if $m < m^*_{f,b}$, $f(m) < 0$ if $m >  m^*_{f,b}$;
\item $ m^*_{f,b} < m_b$; 
\item $m_b \to -1$ when $\beta \to \infty$.
\end{itemize}
In order to check the claims, we note that, when $m \leq 0$,
$$
f'(m) = g'(m) > 0 \text{ iff } m < -m_a = - \sqrt{1- \frac{1}{\beta}},
$$
and $-m_a$ is a local minimum, for which $f(-m_a), \ g(-m_a) <0$. Moreover, $f(m), \ g(m) \to +\infty$ for $m \to -1$.
Combining these with the above considerations for $m \geq 0$, we deduce the first four bullet points. For the fact that $f(m) < g(m)$ we get the fifth claim, while for the last it is sufficient to observe that $m_b < - \sqrt{1- \frac{1}{\beta}} \to -1$ for $\beta \to \infty$.
The above facts and the convergence of $(\tilde{\nu}^N(t))_{t\geq 0}$ to the deterministic process $(\overline{m}(t))_{t \geq 0}$ imply that, if we define $\bar{T}_{m_b} := \inf\left\{t > 0: \tilde{\nu}^N(t) \leq m_b\right\}$ and $\bar{T}_{ m^*_{f,b}} :=  \inf\left\{t > 0: \tilde{\nu}^N(t) \leq m^*_{f,b}\right\}$, there exists a $C >0$, independent of $N$, such that
\begin{equation}
\label{eqn:concl}
\mathbb{P}(\bar{T}_{m_b} \leq C) \geq \mathbb{P}(\bar{T}_{ m^*_{f,b}} \leq C) \xrightarrow{N \to +\infty} 1.
\end{equation}
Indeed, for the deterministic process $\overline{m}(t)$ the arrival time in $m^*_{f,b}$ (which is greater than the one for arriving in $m_b$) is for sure limited by a constant, because of the sign of the vector field of \eqref{eqn:approx_lim}.
If we now consider the original auxiliary process $(\nu^N(t))_{t\geq0}$, i.e.\! the sped up version of $\tilde{\nu}^N(t)$, we get that, defining $\tilde{T}_{ m^*_{f,b}} :=  \inf\left\{t > 0: \nu^N(t) \leq m^*_{f,b}\right\}$, 
\begin{equation}
\label{eqn:fin}
\mathbb{P}(\tilde{T}_{m_b} \leq C(N)) \geq \mathbb{P}(\tilde{T}_{ m^*_{f,b}} \leq C(N)) \xrightarrow{N \to +\infty} 1,
\end{equation}
with $C(N) \xrightarrow{N \to +\infty} 0$, by means of \eqref{eqn:concl}.

We can finally conclude the proof of \eqref{eqn:jump}, by estimating
\begin{equation*}
\mathbb{P}(T_{m_b} < T_{\varepsilon/2}) \geq \mathbb{P}(\tilde{T}_{m_b} < T_{\varepsilon/2}) \geq \mathbb{P}(\tilde{T}_{ m^*_{f,b}} < T_{\varepsilon/2}) \to 1,
\end{equation*}
as $N \to +\infty$. The last limit is deduced by \eqref{eqn:fin} and by the fact that $T_{\varepsilon/2}$ has an explicit distribution - independent of $N$ - which can be found through the reflection principle for the Brownian motion, in the same way we did in Lemma \ref{lemma_x}, for which we have $\mathbb{P}(T_{\varepsilon/2} \leq \delta) \xrightarrow{\delta \to 0} 0$.
\end{proof}

\begin{proof}[Proof of Theorem \ref{thm1}]
Apply Lemmas \ref{lemma_x}, \ref{lem_small} and \ref{lem_jump} to $(x^N(t),m^N(t))_{t \in [0,T]}$ for fixed $\varepsilon,\delta > 0$. Observe that the density of $T_{\varepsilon/2}$ is smooth with respect to $\varepsilon$, and of course $T_{\varepsilon/2} \to 0$ for $\varepsilon \to 0$. Indeed, repeating analogous computations as in Lemma \ref{lemma_x}, we find, for $t \geq 0$,
\begin{align*}
\mathbb{P}(T_{\varepsilon/2} \leq t) = \frac{\varepsilon}{2\sqrt{2\pi\sigma^2}}\frac{1}{t^{3/2}}e^{-\frac{\varepsilon^2}{8\sigma^2t}}.
\end{align*}
The same is true for both $T_{m_a}^\varepsilon, T_{m_a}^{-\delta} \to T_{m_a}^0$, when $\varepsilon,\delta \to 0$. Sending first $N \to +\infty$ and then $\varepsilon,\delta \to 0$, we get the convergence in distribution for all the times $t \leq T_{m_b}$. Once we are in $m_b$, we can restart the dynamics by the strong Markov property and repeat the arguments above for the symmetric negative component of the invariant curve. Inductively, we can find a sequence of almost surely finite stopping times $(T_k)_{k \in \mathbb{N}}$ (the alternate arrival times in the two symmetric critical points), such that $[0,T] = \cup_k \left\{[T_k,T_{k+1}]\cap[0,T]\right\}$. This is enough to deduce the weak convergence of $(m^N(t))_{t \in [0,T]}$ to the process with instantaneous deterministic jumps described by SDE \eqref{eqn:lei}.
\end{proof}

\begin{figure}
    \centering
    \subfloat[$(x^N(t),m^N(t))$ subcritical case.]{{\includegraphics[width=5.5cm]{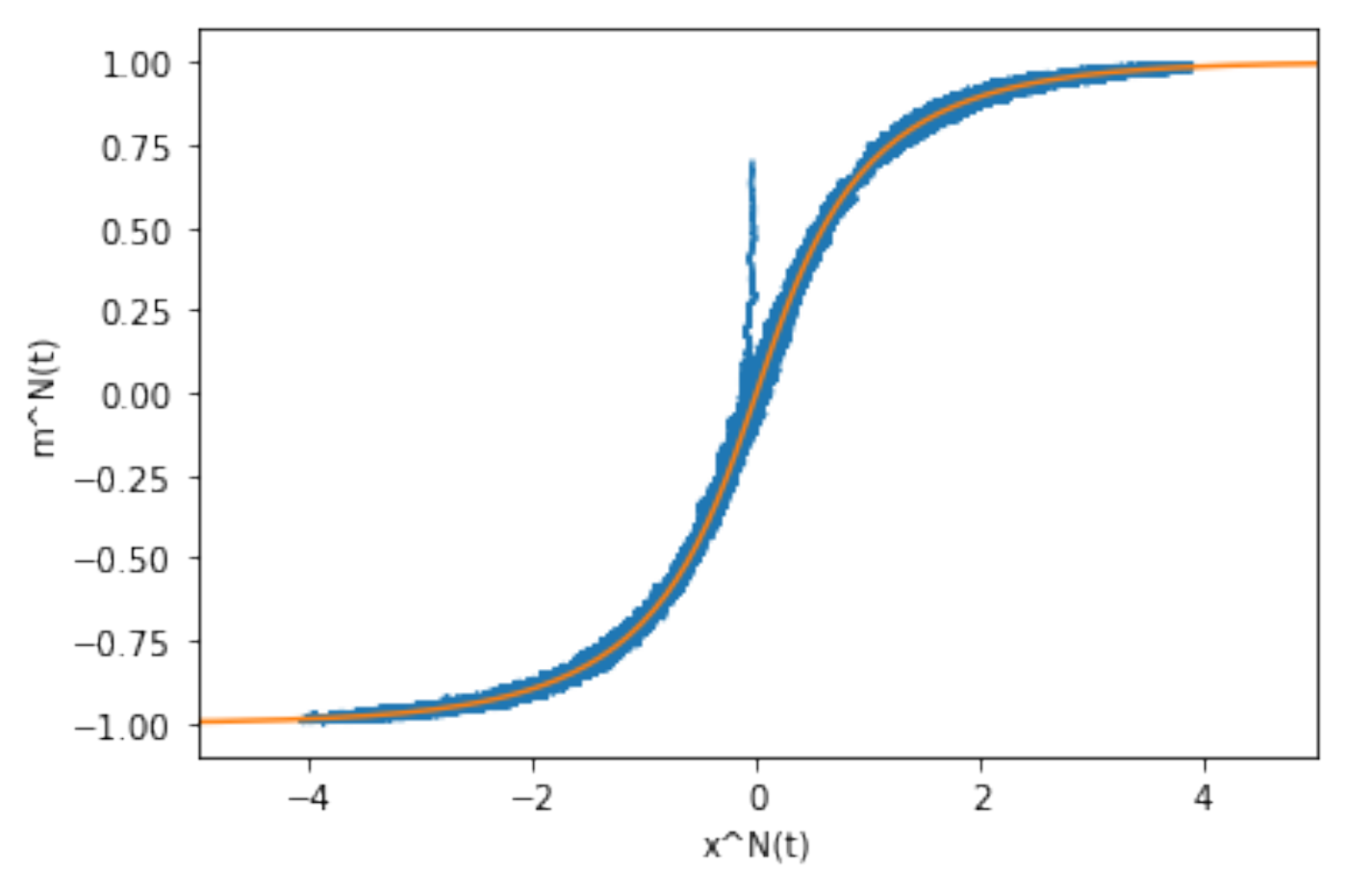} }}%
    \qquad
    \subfloat[$(x^N(t),m^N(t))$ supercritical case.]{{\includegraphics[width=5.5cm]{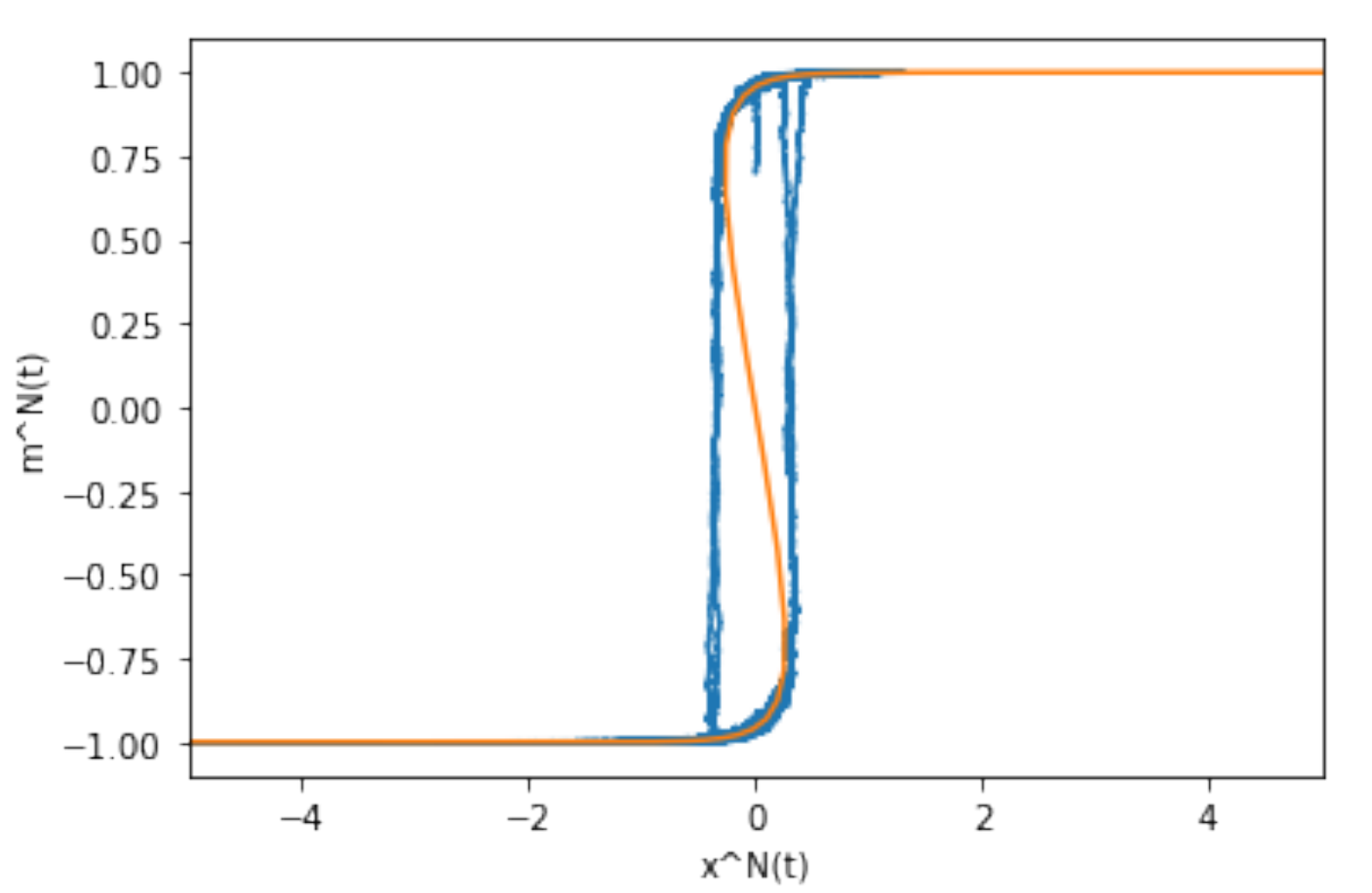} }}%
    \caption{Simulation of the finite $N$ dynamics, for $N = 2000$, $\sigma = 2$, $\beta = 0.5$ (left), and $\beta = 2$ (right).}
    \label{mf-sim}
\end{figure}

In Fig.\! \ref{mf-sim} we show a comparison between two prelimit trajectories in the subcritical and supercritical case for the same initial conditions, where we used the coordinates $(x,m)$ instead of $(\lambda,m)$, which were instead employed in Fig.\! \ref{figura1} and \ref{figura2}. These plots will come useful for a qualitative comparison with the two-level hierarchical case.

\section{The hierarchical model}
\label{hier_hier}
In this section we study the two-level hierarchical version of the previous model. We consider $N$ interacting populations, each of which consists of $N$ mean field interacting particles. 
We denote with a subscript $(i,j)$ the $i$-th individual in the $j$-th population, for $i,j = 1,\dots,N$. The collective state is identified by $N^2$ pairs of variables $(x_{ij}, \mu_{ij})$ (equivalently $(\lambda_{ij}, \mu_{ij})$, with $\lambda_{ij} := \mu_{ij} + x_{ij}$ the total local field), where the $\mu_{ij}$'s are the spins, and the $x_{ij}$'s represent the aggregated remaining characteristics of the individual. As above, we define 
$$
m_j^N(t) := \frac{1}{N}\sum_{i=1}^N \mu_{ij}(t),
$$
the magnetization of the $j$-th population, and the analogous definition for $x_j^N(t)$ and $\lambda_j^N(t)$. Moreover, we define the two-level magnetization as
$$
M^N(t) := \frac{1}{N^2}\sum_{i,j=1}^N \mu_{ij}(t) = \frac{1}{N}\sum_{j=1}^N m_j^N(t),
$$
and the analogous quantities $X^N(t) := \frac{1}{N^2}\sum_{ij}x_{ij}(t) = \frac{1}{N}\sum_{j=1}^Nx_j^N(t)$ (resp.\! $\Lambda^N(t)$) for the $x$ (resp.\! $\lambda$) variables.
Ideally, we want to describe the dynamics at the different hierarchical levels as a projection of a diffusion process onto an invariant curve, as we did for the one population scenario.

With the choices specified in \eqref{eqn:hier}, the stochastic dynamics \eqref{eqn:graph} becomes
\begin{equation}
\label{eqn:hierarch}
\begin{cases}
\mu_{ij} \mapsto -\mu_{ij} \   \text{ rate }  \ 1\!+\tanh\!\left[{-\beta_1 \mu_{ij}(t)(x_{j}^N(t)+ m_{j}^N(t))- \beta_2\mu_{ij}(t)(X^N(t) + M^N(t))}\right],\\
d x_{ij}(t) = \sigma d W_{ij}(t) - \alpha_1\left[x_{ij}(t) - x_j^N\left(t\right)\right]dt - \frac{\alpha_2}{N}\left[x_{ij}(t) - X^N\left(t\right)\right]dt,\\
\mu_{ij}(0) \sim \text{Ber}(p),\\
x_{ij}(0) \sim \mathcal{N}(0,1),
\end{cases}
\end{equation}
for $\beta_1, \beta_2, \sigma,\alpha_1,\alpha_2 > 0$, with the $W_{ij}(t)$'s being $N^2$ independent one-dimensional Brownian motions. 
In terms of the alternative variables $(\mu_{ij},\lambda_{ij})$ and their corresponding macroscopic quantities, the above can be rewritten as
\begin{equation*}
\begin{cases}
\mu_{ij} \mapsto -\mu_{ij} \ \ \ \text{ with rate } \ \ \ 1+\tanh\!\left[{-\beta_1 \mu_{ij}(t)\lambda_j^N(t)- \beta_2\mu_{ij}(t)\Lambda^N(t)}\right],\\
d \lambda_{ij}(t) = d \mu_{ij}(t) + \sigma d W_{ij}(t) - \alpha_1\left[(\lambda_{ij}(t) - \mu_{ij}(t)) - \left(\lambda_j^N(t)- m_j^N(t)\right)\right]dt\\
\ \ \ \ \ \ \ \ \ \ \  - \frac{\alpha_2}{N}\left[(\lambda_{ij}(t) - \mu_{ij}(t)) - \left(\Lambda^N(t) - M^N(t)\right)\right]dt,\\
\mu_{ij}(0) \sim \text{Ber}(p),\\
\lambda_{ij}(0) \sim \text{Ber}(p)*\mathcal{N}(0,1),
\end{cases}
\end{equation*}
where the $*$ denotes the convolution between the two distributions.
Thanks to the linearity of the dynamics for the $x_{ij}$'s, it follows directly from \eqref{eqn:hierarch} that
\begin{equation}
\label{eqn:diff}
\begin{cases}
d x_j^N(t) = -\frac{\alpha_2}{N}\left[x_j^N(t) - X^N(t)\right]dt +\frac{\sigma}{\sqrt{N}}dW_{j}^N(t),\\
x_j^N(0) \sim \mathcal{N}\left(0,\frac{1}{N}\right).
\end{cases}
\begin{cases}
d X^N(t)  = \frac{\sigma}{N} dW^N(t),\\
X^N(0) \sim \mathcal{N}\left(0,\frac{1}{N^2}\right),
\end{cases}
\end{equation}
where $W_j^N := \frac{1}{\sqrt{N}}\sum_{i=1}^N W_{ij}$ are $N$ independent Brownian motions, and $W^N := \frac{1}{\sqrt{N}}\sum_{j=1}^N W_j^N$ 
is another Brownian motion. Note that the laws of $(W_j^N(t))_{t \geq 0 }$ and $(W^N(t))_{t \geq 0}$ are independent of $N$, but we keep the dependency on $N$ in the notation to refer to the specific Brownian motions. 
As we did for the mean field case, we describe each population through the order parameters $(m_j^N(t),x_j^N(t))_{t \geq 0}$. The collective behavior of the system can be studied in terms of the infinitesimal generator of the dynamics applied to a function $f = f\left((m_1,x_1),(m_2,x_2),\dots,(m_N,x_N)\right) =: f(\bm{m},\bm{x})$, $f : [-1,1]^N\times\mathbb{R}^N \to \mathbb{R}$, which is given by
\begin{equation}
\label{eqn:h_gen}
\begin{aligned}
&\mathcal{L}^N f(\bm{m},\bm{x}) \\
& \hspace{0.5cm}:= \sum_{j=1}^N \!\left\{N \frac{1 + m_j}{2} \left(1-\tanh\!\left[{\beta_1 (x_j + m_j)+ \beta_2 (X^N + M^N)}\right]\right)\left[f\!\left(x_j,m_j - \frac{2}{N}\right) \!-\! f(x_j, m_j)\right]\right.\\
& \hspace{0.5cm} \left.+ N \frac{1 - m_j}{2} \left(1+\tanh\left[{\beta_1 (x_j + m_j)+ \beta_2 (X^N + M^N)}\right]\right) \left[f\left(x_j,m_j+ \frac{2}{N}\right) - f(x_j, m_j)\right]\right.\\
& \hspace{0.5cm} \left. + \frac{1}{2N}\sigma^2 \frac{\partial^2}{\partial x_j^2} f\left(x_j,m_j\right) - \frac{\alpha_2}{N}\left(x_j - X^N\right)\frac{\partial}{\partial x_j} f\left(x_j,m_j\right)\right\}.
\end{aligned}
\end{equation}
The rest of the paper is organized as follows: in Section \ref{heuristic_hier} we develop some heuristics to present the expected limit behaviors; in Section \ref{order1} we study the convergence at times of order $1$; we then restrict to the subcritical regime for studying rigorously the convergence to the limit dynamics at times of order $N$ and $N^2$ (respectively addressed in Sections \ref{orderN} and \ref{orderN^2}); in Section \ref{renormalization_theory} we generalize the results giving a conjecture on the $k$-level hierarchical case, for any $k$ finite; finally, in Section \ref{the_limit_case} we study heuristically, with the help of numerics, the zero-temperature limit case $\beta_1=\beta_2=+\infty$, highlighting the presence of a phase transition tuned by the diffusion parameters.

\subsection{Heuristics}
\label{heuristic_hier}
At the first hierarchical level we are interested in describing the limit behavior of the order parameters of each population, i.e. the convergence of the sequences $(m_j^N(t), x_j^N(t))_{t \geq 0}$, both at a timescale of order $1$ and $N$. 
At times of order $1$, by \eqref{eqn:diff} it follows that $d x_j^N(t) \to 0$ and thus $x_j^N(t) \to 0$, that is the mean of the initial condition. The same holds for the sequence $X^N(t) \to 0$. Expanding the generator \eqref{eqn:h_gen} at the first order in the variables $m_j$'s, similarly to what we did for the one population case, we find that $m_j^N(0) \to m(0) = 2p -1$, $m_j^N(t) \to m(t)$, and $M^N(t) \to m(t)$ for $N \to +\infty$, where $(m(t),x(t))_{t \geq 0}$ solves the ODE
\begin{equation}
\label{eqn:ord_1}
\begin{cases}
\dot{m}(t) = 2 \tanh((\beta_1+\beta_2)m(t)) - 2m(t),\\
\dot{x}(t) = 0,\\
m(0) = 2p -1,\\
x(0) = 0.
\end{cases}
\end{equation}
Eq. \eqref{eqn:ord_1} is the mean field equation for the Curie--Weiss model with inverse temperature parameter $\beta_1 + \beta_2$, i.e.\! the corresponding two-level hierarchical version of the deterministic mean field limit Eq. \eqref{eqn:limit_sys}. The equilibria of the above ODE are either just one ($m = 0$), when $\beta_1 + \beta_2 \leq 1$, or three when $\beta_1 + \beta_2 > 1$: two stable (the polarized ones) and one unstable (the disordered one), where the asymptotic one is one of the two polarized states, determined by the sign of the initial magnetization.

At times of order $N$, the diffusions $x_j^N$'s are now subject to non-trivial dynamics. Indeed, denoting again - with an abuse of notation - the sped up processes as $x_j^N(t) := x_j^N(Nt)$, $X^N(t) := X^N(Nt)$, equations \eqref{eqn:diff} become
\begin{equation}
\begin{cases}
d x_j^N(t) = -\alpha_2\left[x_j^N(t) - X^N(t)\right]dt + \sigma dW_{j}^N(t),\\
x_j^N(0) \sim \mathcal{N}\left(0,\frac{\sigma^2}{2\alpha_2}\frac{1}{N}\right).
\end{cases}
\begin{cases}
d X^N(t)  = \frac{\sigma}{\sqrt{N}} dW^N(t),\\
X^N(0) \sim \mathcal{N}\left(0,\frac{\sigma^2}{2\alpha_2}\frac{1}{N^2}\right),
\end{cases}
\end{equation}
where the initial data are given by the long-time limit of the diffusions at the timescale of order $1$.
In this timescale we thus find $x_j^N(t) \to x(t)$, $X^N(t) \to 0$, where $x(t)$ follows the Ornstein-Uhlenbeck dynamics
\begin{equation*}
\begin{cases}
d x(t) = -\alpha_2x(t)dt + \sigma dW(t),\\
x(0) =0,
\end{cases}
\end{equation*}
with $W$ a Brownian motion.
As in the mean field case, the accelerated approximate diffusive generator can give us intuition on the limit dynamics for the magnetization processes at a timescale of order $N$. Indeed, expanding up to the second order the jump terms of the dynamics in $m_j$ in \eqref{eqn:h_gen}, we get
\begin{equation}
\label{eqn:expanded}
\begin{aligned}
N \mathcal{L}^N & f\left(m_j,  x_j\right)\approx N\left[2 \tanh(\beta_1 (x_j+m_j) + \beta_2(X^N+ M^N))- 2 m_j\right]\frac{\partial}{\partial m_j}f(m_j,x_j)\\
& +\left[2- 2m_j\tanh(\beta_1 (x_j+m_j) + \beta_2(X^N+ M^N))\right]\frac{\partial^2}{\partial m_j^2}f(x_j, m_j) \\
&+ \frac{\sigma^2}{2} \frac{\partial^2}{\partial x_j^2} f\left(x_j,m_j\right) - \alpha_2\left(x_j - X^N\right)\frac{\partial}{\partial x_j} f\left(x_j,m_j\right).
\end{aligned}
\end{equation}
Assuming that a propagation of chaos property holds, the presence of the strong drift in the above generator should be such that the limit of the magnetizations processes $m_j^N(t)$'s is a (mean field) process laying on the curve $m = \tanh(\beta_1(x+m)) + \beta_2 M)$, where the dynamics is driven by the evolution of the Ornstein-Uhlenbeck limit process $x(t)$. Moreover, the limit mean field $M(t)$ should be proved to be the mean of $m(t)$ with respect to the distribution of $x(t)$. Specifically, denoting with $\mu_t(dx)$ the distribution of the O-U process at time $t$, we should find that each pair of accelerated processes  $(x_j^N(t),m_j^N(t))_{t \geq 0}$, for $j=1,\dots,N$, at times of order $N$, converges to 
\begin{equation}
\label{eqn:first_lev}
\begin{cases}
m(t) := m(t)(x(t)),\\
d x(t) = \sigma d W(t) - \alpha_2 x(t) dt,\\
m(0) = 2p-1,\\
x(0) = 0,\\
M(t) = \int_{\mathbb{R}}m(t)(x) \mu_{t}(dx),
\end{cases}
\end{equation}
where $m(t)(x) := \tanh[\beta_1(x+ m(t)(x)) + \beta_2 M(t)]$.
The study of \eqref{eqn:first_lev} is hard to perform for general choices of the parameters. Indeed, the behavior of the dynamics can drastically change, depending on $\beta_1, \beta_2, \alpha_2, \sigma$ and the initial conditions. By analogy with the mean field case, one can expect to recognize a radical difference between the case where one has uniqueness of the equilibrium for the dynamics at order $1$ \eqref{eqn:ord_1}, and the case where multiple equilibria appear.  

At the second hierarchical level, we write the infinitesimal generator for a function $f(M, X)$ by averaging over the different populations, 
\begin{align*}
\mathcal{L^N}& f(M,X) = \\
& N \sum_{j=1}^N \frac{1 + m_j}{2} \left(1-\tanh\left[\beta_1 (x_j + m_j)+ \beta_2 (X + M)\right]\right)\left[f\left(M - \frac{2}{N^2}, X\right) - f(M, X)\right]\\
& + N\sum_{j=1}^N \frac{1 - m_j}{2} \left(1+\tanh\left[{\beta_1 (x_j + m_j)+ \beta_2 (X + M)}\right]\right)\left[f\left(M + \frac{2}{N^2}, X\right) - f(M, X)\right] \\
&+ \frac{1}{2}\frac{\sigma^2}{N^2} \frac{\partial^2}{\partial X^2} f(M, X).
\end{align*}
With analogous expansions as above for the jump components, we find
\begin{align*}
\mathcal{L^N} & f(M,X) \approx \frac{1}{N} \sum_{j=1}^N \Big[2 \tanh\left[\beta_1 (x_j + m_j)+ \beta_2 (X + M)\right] - 2 m_j\Big]\frac{\partial }{\partial M}f(M,X) \\
& + \frac{1}{N^3}\sum_{j=1}^N \Big[2 - 2 m_j \tanh\left[\beta_1 (x_j + m_j)+ \beta_2 (X + M)\right]\Big]\frac{\partial^2}{\partial M^2} f(M,X) \\
& + \frac{1}{2N^2}\sigma^2 \frac{\partial^2}{\partial X^2} f(M, X).
\end{align*}
In the drift component we can recognize the empirical average of the drifts of the single magnetizations.
It is reasonable to ask for a description of the limit dynamics of $M^N(t)$ at any timescale. As we already motivated heuristically, at a timescale of order $1$ the limit $M(t)$ of the macroscopic magnetization is the same as the magnetization of each population, which follows a Curie--Weiss ODE. For long times (but still of order $1$), the value of $M^N(t)$ should converge to the stable equilibrium of the C--W ODE, which, depending on the value of $\beta_1+\beta_2$ may be the disordered or a polarized state. Once we consider a scale of order $N$, we expect the single magnetizations to be close to their invariant curves. However, the evolution of $M^N(t)$ can change drastically depending on the interaction and diffusion parameters. We expect to find a regime of the parameters for which $M^N(t)$ does not move much from the equilibrium reached at times of order $1$, eventually starting to move only at a scale of order $N^2$, when the macroscopic diffusion $X^N(t)$ starts to evolve non-trivially.
At least in this regime, we expect the $N^2$ accelerated second-level process $M^N(t)$, conditionally on $X^N(t) \approx X$, to converge, for every fixed $t \geq 0$, to the deterministic value
\begin{equation}
\label{eqn:limi_cond}
\begin{cases}
M(t) = \int_{\mathbb{R}} \tanh(\beta_1(x + m(t)(x)) + \beta_2(X + M(t)))\mu_{\infty}(dx;X),\\
M(0) = 2p -1, \\
\end{cases}
\end{equation}
where $\mu_{\infty}(dx;X)$ is the stationary distribution of the process
$$
dx(\xi) = -\alpha_2(x(\xi) - X) d\xi + \sigma d W(\xi),
$$
where $X$ enters as a parameter (it must be intended as the \textit{current} fixed value of $X(t)$), and $m(t)(x)$ is the solution to 
$$
m(t)(x) = \tanh(\beta_1(x+m(t)(x)) + \beta_2(X + M(t))).
$$
In turns, the limit process $X(t)$, $X^N(t) \to X(t)$, evolves as
\begin{equation}
\label{eqn:limi_diff}
\begin{cases}
dX(t) = \sigma d B(t),\\
X(0) = 0,
\end{cases}
\end{equation}
where $B$ is a Brownian motion.
In order to obtain a full description of the law of the limit process $M(t)$, one then needs to consider a combination of the conditional dynamics \eqref{eqn:limi_cond} and \eqref{eqn:limi_diff}, which takes into account the diffusive motion of $X(t)$ (see Section \ref{orderN^2} for details). 

For a rigorous treatment (Sections \ref{order1}-\ref{orderN^2}) we restrict to the subcritical case $\beta_1 + \beta_2 < 1$ (except for the order $1$ timescale, analyzed in Section \ref{order1}, where the argument works for any choice of the parameters), while we give solid heuristics and numerics for the supercritical zero-temperature limit regime $\beta_1 = \beta_2 \to +\infty$, analyzing the relevance of the diffusion parameters $\alpha_2$ and $\sigma$ for obtaining a phase transition already at a timescale of order $N$ (see Section \ref{the_limit_case} below). Moreover, in Section \ref{renormalization_theory} we conjecture a generalization of the results on the subcritical regime to the $k$-level hierarchical version of the model.

\subsection{Propagation of chaos at times of order $1$}
\label{order1}
In this section we prove the convergence of the empirical processes $(m_j^N(t),x_j^N(t))_{j=1,\dots,N}$ to the deterministic limit dynamics given by \eqref{eqn:ord_1}, for any choice of the parameters. Our proof works as well for random i.i.d.\! initial data $x_j^N(0) \sim \mu(dx)$, when $\mu(dx)$ is a normal distribution $\mathcal{N}(0,(\sigma^*)^2)$ (in our particular case we have $\sigma^* = \frac{1}{\sqrt{N}}$, so that randomness is deleted in the limit), with the resulting modification of the limit dynamics,
\begin{equation}
\label{eqn:gen_ord1}
\begin{cases}
\dot{m}(t)(x) = 2 \tanh(\beta_1(x+m(t)(x))+\!\beta_2 M(t)) - 2m(t)(x),\\
m(0)(x) \equiv 2p -1,\\
M(t) = \int_{\mathbb{R}}m(t)(x)\mu(dx).
\end{cases}
\end{equation}
Considering random initial data also for the limit dynamics will be useful for the analyses of the longer timescales. For clarity we recall the dynamics of the empirical processes $(x_j^N(t),m_j^N(t))_{t \geq 0}$,
\begin{equation}
\label{eqn:empirical_ord1}
\begin{cases}
m_j^N \! \mapsto \!m_j^N \pm \frac{2}{N} \text{ rate } N \frac{1 \mp m_j^N(t)}{2}\! \left(1 \!\pm \tanh\!\left[\beta_1 (x_j^N(t) \!+\! m_j^N(t)) \!+ \beta_2 (X^N(t) \!+ \! M^N(t))\right]\right),\\
m_{j}^N(0) = m_j \sim \frac{1}{N}\text{Bin}(N,p),\\
d x_j^N(t) = -\frac{\alpha_2}{N}\left[x_j^N(t) - X^N(t)\right]dt +\frac{\sigma}{\sqrt{N}} dW_{j}^N(t),\\
x_{j}^N(0) = x_j \sim \mathcal{N}\left(0, (\sigma^*)^2\right).
\end{cases}
\end{equation}
Since the magnetizations are not appearing in the diffusion dynamics, the propagation of chaos property for the $x_j^N(t)$'s is trivially true for any finite time interval. Indeed, every diffusion is converging to its initial datum due to the decaying factors in front of the drift and diffusion coefficients. The i.i.d.\! processes $(\tilde{m}_j(t))_{j=1,\dots,N}$ to which the $m_j^N(t)$'s will be proved to converge are denoted as $\tilde{m}_j(t) := m(t)(x_j)$,
where the $x_j$'s coincide with the initial data for the diffusions, and $m(t)(x)$ is the solution to \eqref{eqn:gen_ord1}.
\begin{thm}[Propagation of chaos at order $1$]
\label{chaos_ord1}
Fix $T > 0$. For any $\beta_1,\beta_2,\alpha_1,\alpha_2,\sigma > 0$, and any $j = 1,\dots,N$, we have
\begin{equation}
\label{eqn:chaos1}
\lim_{N\to \infty}\mathbb{E}\Bigg[\sup_{t \in [0,T]}\big| m_j^N(t) - \tilde{m}_j(t)\big|\Bigg]=  0.
\end{equation}
\end{thm}
Before proving Theorem \ref{chaos_ord1} we need to assess the well-posedness of Eq. \eqref{eqn:gen_ord1}. We rewrite the dynamics with a generic initial datum 
\begin{equation}
\label{eqn:gen_ord1_generic}
\begin{cases}
\dot{m}(t)(x) = 2 \tanh(\beta_1(x+m(t)(x))+\!\beta_2 M(t)) - 2m(t)(x),\\
m(0)(x) = m_0(x),\\
M(t) = \int_{\mathbb{R}}m(t)(x)\mu(dx),
\end{cases}
\end{equation}
with $m_0 : \mathbb{R} \to [-1,1]$, $m_0 \in C(\mathbb{R})$.
\begin{prop}[Well-posedness at order $1$]
\label{wp-ord1}
For any $T>0$, Eq. \eqref{eqn:gen_ord1_generic} has a unique solution $m : [0,T] \times \mathbb{R} \to [-1,1]$ such that $m(t)(\cdot) \in C(\mathbb{R})$ for any $t \in [0,T]$.
\end{prop}
\begin{proof}
The vector field $f : \mathbb{R} \times C(\mathbb{R}) \to C(\mathbb{R})$,
\begin{equation}
\label{eqn:vect_field_ord1}
f(x,m):= 2 \tanh(\beta_1(x+m) + \beta_2 M) - 2m
\end{equation} 
is globally Lipschitz continuous for any $\beta_1, \beta_2 > 0$, thus existence and uniqueness of a solution to \eqref{eqn:gen_ord1_generic}, with $m(t)(\cdot) \in C(\mathbb{R})$ for any $t \in [0,T]$, is standard. Moreover, studying the sign of the vector field \eqref{eqn:vect_field_ord1}, we see that \eqref{eqn:gen_ord1_generic} defines a dynamics such that $m(t) : \mathbb{R} \to [-1,1]$, provided the initial datum $m_0 :  \mathbb{R} \to [-1,1]$ has the same property. Indeed, at a point $\overline{x} \in \mathbb{R}$ for which $m(t)(\overline{x}) = 1$, we have that $\frac{d}{dt}m(t)(x)\Big|_{x = \overline{x}} \leq 0$, and symmetrically if $m(t)(\overline{x}) = -1$ it holds $\frac{d}{dt}m(t)(x)\Big|_{x = \overline{x}} \geq 0$.
\end{proof}
For the proof of Theorem \ref{chaos_ord1}, we make use of a representation of the jump processes $m_j^N(t)$'s in terms of SDEs, by employing Poisson random measures (see \cite{Graham}), as follows
\begin{equation}
\label{eqn:sde_hierarch}
m_j^N(t)  = m_j^N(0) + \int_0^t \int_\Xi f( m_j^N(s^{-}), \xi, M^N(s^-),x_j^N(s),X^N(s))\mathcal{N}_j(ds,d\xi),
\end{equation}
for $j= 1,\dots,N$, where each $m_j^N(t)$ takes values in $\Sigma = \left\{-1,-1+\frac{2}{N},\dots,1-\frac{2}{N},1\right\}$; the $\mathcal{N}_j$'s are $N$ i.i.d.\! stationary Poisson random measures on $[0,T] \times \Xi$ with intensity measure $\nu$ on $\Xi := [0,\infty)^{|\Sigma|} \subset \mathbb{R}^{|\Sigma|}$ given by
\begin{equation}
\label{nu_hierarch}
\nu(E) := \sum_{i=1}^{|\Sigma|} \ell(E \cap \Xi_i),
\end{equation}
for any $E$ in the Borel $\sigma$-algebra $\mathcal{B}(\Xi)$ of $\Xi$, where $\Xi_j := \left\{ u \in \Xi \ : \ u_i = 0 \ \ \forall \ i \neq j\right\}$ is viewed as a subset of $\mathbb{R}$, and $\ell$ is the Lebesgue measure on $\mathbb{R}$.  
We fix a probability space $(\Omega, \mathcal{F}, \mathbb{P})$ and denote by 
$\mathbb{F}=(\mathcal{F}_t)_{t\in [0,T]} $ the filtration generated by the Poisson measures. The function $f$, modeling the possible jumps of the process, is given by
\begin{equation*}
f(m,\xi,M,x,X) :=\sum_{y \in \Sigma}(y-m)\mathbbm{1}_{]0, \lambda_{my}[} (\xi_y),
\end{equation*}
where $\lambda_{my}$ denotes the rate of jumping from state $m$ to state $y$. Denoting by 
$$
\lambda_{\pm}(m,M,x,X) := N \frac{1 \mp m}{2} \left(1 \pm \tanh\left[\beta_1 (x + m)+ \beta_2 (X + M)\right]\right)
$$ 
the rate of going from $m$ to $m \pm \frac{2}{N}$, in our case the function $f$ further simplifies to 
\begin{equation}
\label{f_hierarch}
f(m,\xi,M,x,X) = \frac{2}{N}\mathbbm{1}_{]0, \lambda_+[} (\xi_{m + \frac{2}{N}}) - \frac{2}{N}\mathbbm{1}_{]0,\lambda_-[}(\xi_{m - \frac{2}{N}}),
\end{equation}
since the only possible jumps are the ones from $m$ to $m \pm \frac{2}{N}$ with rates $\lambda_{\pm}$.
The above definitions of $f$ and $\nu$ ensure that $\lambda_{\pm}$ are exactly the transition rates of the continuous time Markov chains $m_j^N(t)$'s, and that $\pm \frac{2}{N}$ are the only possible jumps allowed at every time. Indeed, it is easy to prove that with our choices \eqref{eqn:sde_hierarch} is equivalent to 
\begin{equation*}
\mathbb{P}\left[m_j^N(t+h) \!= m \!\pm \frac{2}{N} \Bigg|  \! m_j^N(t) \!=\! m, M^N(t) \!=\! M, x_j^N(t) \!=\! x, X^N(t) \!=\! X\!\right] \!=\! \lambda_{\pm}(m, M,x,X) h + o(h).
\end{equation*}
By the smoothing formula of Poisson calculus (see \cite[Ch. 9]{bremaud}), we have 
\begin{equation}
\label{eqn:smoothing}
\begin{aligned}
\mathbb{E}&\left[m_j^N(t)\right] = \mathbb{E}\left[m_j^N(0)\right] + \mathbb{E}\left[\int_0^t\int_\Xi f( m_j^N(s^{-}), \xi, M^N(s^-),x_j^N(s),X^N(s))ds \nu(d\xi)\right]\\
& = \mathbb{E}\left[m_j^N(0)\right] + \mathbb{E}\left[\int_0^t\int_\Xi\left[\frac{2}{N}\mathbbm{1}_{]0, \lambda_+[} (\xi_{m + \frac{2}{N}}) - \frac{2}{N}\mathbbm{1}_{]0,\lambda_-[}(\xi_{m - \frac{2}{N}})\right]ds \nu(d\xi)\right]\\
& = \mathbb{E}\left[m_j^N(0)\right] + \mathbb{E}\Bigg[\int_0^t\Big[ 2 \tanh\left(\beta_1(x_j^N(s)+m_j^N(s))+\beta_2 (X^N(s) +M^N(s))\right)-  2m_j^N(s)\Big]ds\Bigg].
\end{aligned}
\end{equation}
\begin{proof}[Proof of Theorem \ref{chaos_ord1}]
First, we observe that, by the dynamics \eqref{eqn:sde_hierarch} with the choice (\ref{f_hierarch}) for $f$, we can write
\begin{align*}
\sup_{s \in [0,t]}&|m_j^N(s) - \tilde{m}_j(s)| = \sup_{s \in [0,t]}\left|\int_0^s \int_\Xi f( m_j^N(r^{-}), \xi, M^N(r^-),x_j^N(r),X^N(r))\mathcal{N}_j(dr,d\xi) - \tilde{m}_j(s)\right|.
\end{align*}
Taking the expectation and using formula \eqref{eqn:smoothing} and the limit dynamics \eqref{eqn:gen_ord1}, we can estimate 
\begin{align*}
\mathbb{E}&\Bigg[\sup_{s \in [0,t]}\big|m_j^N(s) - \tilde{m}_j(s) \big|\Bigg] \leq \mathbb{E}\bigg[\big|m_j^N(0) - (2p - 1)\big|\bigg] + \mathbb{E}\Bigg[\int_0^t \Big| 2 m_j^N(s)- 2 \tilde{m}_j(s) \Big|ds\Bigg]\\
&\!\!\!\!+\! \mathbb{E}\Bigg[\!\int_0^t \!\Big| 2\tanh\left(\beta_1(x_j^N(s) \!+\! m_j^N(s)) \!+\! \beta_2(X^N(s) \!+\! M^N(s))\right)\!-\! 2\tanh\left(\beta_1 (x_j \!+\! \tilde{m}_j(s)) \!+\! \beta_2 M(s)\right)\Big|ds\Bigg] \\
& \leq \mathbb{E}\bigg[\big|m_j^N(0) - (2p - 1)\big|\bigg]+ \mathbb{E}\Bigg[\int_0^t \Big| 2\tanh\left(\beta_1(x_j^N(s) + m_j^N(s)) + \beta_2(X^N(s) + M^N(s))\right)\\
& - 2\tanh\left(\beta_1 (x_j + \tilde{m}_j(s)) + \beta_2 M(s)\right)\Big|ds\Bigg] + C\mathbb{E}\Bigg[\int_0^t \sup_{r \in [0,s]}\Big| m_j^N(r)- \tilde{m}_j(r) \Big|ds\Bigg].
\end{align*}
By LLN on the initial data we have $\mathbb{E}\bigg[\big|m_j^N(0) - (2p - 1)\big|\bigg] \leq C(N)$, with $C(N) \xrightarrow{N \to +\infty} 0$. We now focus on estimating the first of the two integrals.
Using the globally Lipschitz continuity of $\tanh(\cdot)$, we have
\begin{align*}
\mathbb{E}&\Bigg[\int_0^t \Big| 2\tanh\left(\beta_1(x_j^N(s) + m_j^N(s)) + \beta_2(X^N(s) + M^N(s))\right)\\
& - 2\tanh\left(\beta_1 (x_j + \tilde{m}_j(s)) + \beta_2 M(s)\right)\Big| ds \Bigg] \leq C \mathbb{E}\Bigg[\int_0^t \Big|x_j^N(s) - x_j\Big| ds \Bigg] + C \mathbb{E}\Bigg[\int_0^t \Big| X^N(s)\Big|ds \Bigg] \\
& + C\mathbb{E}\Bigg[\int_0^t \Big| m_j^N(s) - \tilde{m}_j(s)\Big|ds\Bigg] + C\mathbb{E}\Bigg[\int_0^t \Big|M^N(s) - M(s)\Big|ds\Bigg]\\
& \leq C \mathbb{E}\Bigg[\int_0^T \Big|x_j^N(s) - x_j\Big| ds \Bigg] + C \mathbb{E}\Bigg[\int_0^T \Big| X^N(s)\Big|ds \Bigg] + C\mathbb{E}\Bigg[\int_0^t \sup_{r \in [0,s]}\Big| m_j^N(r) - \tilde{m}_j(r)\Big|ds\Bigg] \\
&+ C\mathbb{E}\Bigg[\int_0^t \sup_{r \in [0,s]}\Big|M^N(r) - M(r)\Big|ds\Bigg],
\end{align*}
where the constants are allowed to change from line to line.
By the propagation of chaos for the diffusions, we have
$$
 \mathbb{E}\Bigg[\int_0^T \Big|x_j^N(s) - x_j\Big| ds \Bigg] +  \mathbb{E}\Bigg[\int_0^T \Big| X^N(s)\Big|ds \Bigg] \leq C(N),
$$
for some $C(N) \to 0$ when $N \to +\infty$.
For the last integral, denoting $\tilde{M}^N(t) := \frac{1}{N}\sum_{i=1}^N \tilde{m}_i(t)$, we estimate
\begin{align*}
\mathbb{E}&\Bigg[\int_0^t  \sup_{r \in [0,s]}\Big|M^N(r) - M(r)\Big|ds\Bigg] \leq \mathbb{E}\Bigg[\int_0^t  \sup_{r \in [0,s]}\Big|M^N(r) - \tilde{M}^N(r)\Big|ds\Bigg] \\
&+ \mathbb{E}\Bigg[\int_0^t  \sup_{r \in [0,s]} \Big|\tilde{M}^N(r) - M(r)\Big|ds\Bigg]\leq C\frac{1}{N}\sum_{i=1}^N\mathbb{E}\Bigg[\int_0^t \sup_{r \in [0,s]}\Big|m_i^N(r) - \tilde{m}_i(r)\Big|ds\Bigg] + C(N)\\
& = C \mathbb{E}\Bigg[\int_0^t \sup_{r \in [0,s]}\Big| m_j^N(r) - \tilde{m}_j(r)\Big|ds\Bigg] + C(N),
\end{align*}
where the $C(N) \to 0$ when $N \to +\infty$ by LLN, and the last equality is a consequence of the exchangeability of the processes $(m_i^N(t), \tilde{m}_i(t))_{i=1,\dots,N}$.
Recollecting all the above observations and estimates, we have found
\begin{equation*}
\mathbb{E}\Bigg[\sup_{s \in [0,t]}\big|m_j^N(s) - \tilde{m}_j(s) \big|\Bigg] \leq C(N) + C \mathbb{E}\Bigg[\int_0^t \sup_{r \in [0,s]}\Big| m_j^N(r) - \tilde{m}_j(r)\Big|ds\Bigg],
\end{equation*}
with $C(N)$ going to $0$ for $N \to +\infty$. Denoting $\varphi(t) := \mathbb{E}\Bigg[\sup_{s \in [0,t]}\big|m_j^N(s) - \tilde{m}_j(s) \big|\Bigg]$, the last estimate implies $\varphi(t) \leq C(N) + \int_0^t \varphi(s) ds$.
Thus, the propagation of chaos follows by the Gronwall's lemma.
\end{proof}
\begin{rem}
\label{true_chaos1}
Note that the strong convergence \eqref{eqn:chaos1} implies the convergence (in e.g.\! $1$-Wasserstein distance $\bm{d_1}$, uniform in time) of the associated empirical measures $\mu^N(t) := \frac{1}{N}\sum_{j=1}^N \delta_{m_j^N(t)}$ and $\tilde{\mu}^N(t) := \frac{1}{N}\sum_{j=1}^N \delta_{\tilde{m}_j(t)}$ to the deterministic measure $\mu(t)$, the distribution of the i.i.d.\! processes $\tilde{m}_j(t)$. Indeed, one has that, almost surely, $||\mu^N - \tilde{\mu}^N||_{\bm{d_1}} \leq \frac{1}{N}\sum_{i=1}^N \mathbb{E}\left[\sup_{t \in [0,T]}|m_i^N(t) - \tilde{m}_i(t)|\right]$, which tends to zero  because of \eqref{eqn:chaos1}, while $||\tilde{\mu}^N - \mu||_{\bm{d_1}} \to 0$ as $N \to +\infty$ is standard (by LLN). This in turns implies the propagation of chaos in the classic sense.
\end{rem}

The following proposition assesses the long-time behavior of the deterministic limit dynamics. Specifically, we show the convergence to a unique symmetric stationary profile $\overline{m}(x)$, regardless of the initial datum $m_0(x)$.
\begin{prop}[Long-time subcritical limit behavior]
\label{longtime1}
For $\beta_1 + \beta_2 < 1$, the solution $m(t)(\cdot)$ to \eqref{eqn:gen_ord1_generic} is such that
\begin{equation}
\label{eqn:time_limit}
\mathbb{E}\left[|m(t)(\xi) - \overline{m}(\xi)|^2\right] \rightarrow 0,
\end{equation}
for $t \to \infty$, with $\xi \sim \mathcal{N}(0,\sigma^*)$ and $\overline{m}(\cdot)$ is the unique solution to
\begin{equation}
\label{eqn:asympt}
\overline{m}(x) = \tanh(\beta_1(x+\overline{m}(x))).
\end{equation}
\end{prop} 
\begin{proof}
The uniqueness of solution to Eq. \eqref{eqn:asympt} follows by considering any two solutions $m(x), n(x)$ and observing that
$$
|m(x) - n(x)| \leq \beta_1 |m(x) - n(x)| \leq \dots \leq \beta_1^k  |m(x) - n(x)|,
$$
for any $x \in \mathbb{R}$, so that we can conclude by a contraction argument. For the proof of \eqref{eqn:time_limit}, consider any two solutions $m(t)$ and $n(t)$ with different initial data. It holds
\begin{align}
\label{eqn:est}
\frac{1}{2}\frac{d}{dt}\int_{\mathbb{R}}\big(m(t)(x) - n(t)(x)\big)^2\mu(dx) \leq -2(1-(\beta_1+\beta_2)) \int_{\mathbb{R}}\big(m(t)(x) - n(t)(x)\big)^2 \mu(dx),
\end{align}
which is negative for $\beta_1 + \beta_2 < 1$, thus implying \eqref{eqn:time_limit} because of the well-posedness of \eqref{eqn:gen_ord1_generic}. Indeed $\overline{m}(x)$, the unique solution to Eq. \eqref{eqn:asympt}, is always a solution to \eqref{eqn:gen_ord1_generic} with initial datum $m_0(x) = - m_0(-x)$ and $M(t) = 0$ for every $t$. In order to verify \eqref{eqn:est}, we use Eq. \eqref{eqn:gen_ord1_generic} to compute
\begin{align*}
\frac{1}{2}&\frac{d}{dt}\int_{\mathbb{R}}(m(t)(x) - n(t)(x))^2\mu(dx) = \int_{\mathbb{R}}(\dot{m}(t)(x) - \dot{n}(t)(x))(m(t)(x) - n(t)(x))\mu(dx)\\
& = -2\int_{\mathbb{R}}(m(t)(x) - n(t)(x))^2\mu(dx) \\
&+\! 2 \!\int_{\mathbb{R}}\!\Big[\tanh(\beta_1(m(t)(x) + x) + \beta_2 M(t)) - \tanh(\beta_1(n(t)(x) + x) + \beta_2N(t))\Big]\times\\
& \times(m(t)(x) - n(t)(x))\mu(dx)\\
& \leq -2\int_{\mathbb{R}}(m(t)(x) - n(t)(x))^2\mu(dx)  + 2(\beta_1 + \beta_2)\int_{\mathbb{R}}(m(t)(x) - n(t)(x))^2 \mu(dx),
\end{align*}
where in the last step we have used the Lipschitz properties of $\tanh(\cdot)$ and the definitions of $M(t)$ and $N(t)$.
\end{proof}

\begin{rem}
\label{modified_ord1}
Theorem \ref{chaos_ord1} and Propositions \ref{wp-ord1}, \ref{longtime1} can be generalized to the case of Gaussian initial data not centered around zero. The limit equation becomes
\begin{equation}
\label{eqn:modified_ord1}
\begin{cases}
\dot{m}(t)(x) = 2 \tanh(\beta_1(x+m(t)(x))+ \beta_2 (\overline{X} + M(t))) - 2m(t)(x),\\
m(0)(x)= m_0(x),\\
M(t) = \int_{\mathbb{R}}m(t)(x)\mu(dx;\overline{X}),
\end{cases}
\end{equation}
with $\mu(dx;\overline{X}) = \mathcal{N}\left(\overline{X},\rho^2\right)$. The equilibrium solution to \eqref{eqn:modified_ord1} is given by 
\begin{equation}
\label{eqn:modified_longtime}
\begin{cases}
\overline{m}_{\overline{X}}(x) = \tanh\Big(\beta_1(x+\overline{m}_{\overline{X}}(x)) + \beta_2(\overline{X} + \overline{M})\Big),\\
\overline{M} = \int_{\mathbb{R}}\overline{m}_{\overline{X}}(x)\mu(dx;\overline{X}),
\end{cases}
\end{equation}
whose well-posedness can be obtained by a contraction argument as in Proposition \ref{longtime1}.
\end{rem}

We conclude the section noting that the processes $x_j^N$'s and $m_j^N$'s are close to their i.i.d.\! limits for any fixed time ranging in an interval which is allowed to grow with $N$ with a certain speed. 
\begin{thm}[Long-time subcritical particles behavior]
\label{erg1}
For any $T > 0$, $\beta_1 + \beta_2 < 1$, $\varepsilon >0$ and $j =1,\dots,N$, we have 
\begin{itemize}
\item[(i)] For any $A \in \mathcal{B}(\mathbb{R})$,
$$
\sup_{t \in [0,TN^{2-\varepsilon}]}\Big| \mathbb{P}\Big(x_j^N(t) \in A\Big) - \mathbb{P}\Big(x_j(t) \in A\Big)\Big| \xrightarrow{N \to +\infty} 0.
$$ 
\item[(ii)] For any $A \in\mathcal{B}([-1,1])$, 
$$
\sup_{t \in [0,TN^{2/3 - \varepsilon}]}\Big| \mathbb{P}\Big(m_j^N(t) \in A \Big) - \mathbb{P}\Big(\tilde{m}_j(t) \in A\Big)\Big| \xrightarrow{N \to +\infty} 0,
$$ 
\end{itemize}
where $\tilde{m}_j(t) := m(t)(x_j(t))$, with
\begin{equation}
\label{eqn:xjN_h}
\begin{cases}
d x_j^N(t) = -\frac{\alpha_2}{N}(x_j^N(t) - X^N(t)) dt + \frac{\sigma}{\sqrt{N}} d W_j^N(t),\\
x_j^N(0) = x_j \sim \mathcal{N}\left(0,\frac{1}{N}\right),
\end{cases}
\end{equation}
and
\begin{equation}
\label{eqn:xj_h}
\begin{cases}
d x_j(t) = -\frac{\alpha_2}{N}(x_j(t) - \mathbb{E}[x_j(t)]) dt + \frac{\sigma}{\sqrt{N}} d W_j(t),\\
x_j(0) = x_j  \sim \mathcal{N}\left(0,\frac{1}{N}\right),
\end{cases}
\end{equation}
with $X^N(t) := \frac{1}{N}\sum_{k=1}^N x_k^N(t)$ and $W_j(t)$ is a Brownian motion. 
\end{thm}
\begin{proof}
We realize the process $x_j^N(t)$ by plugging in \eqref{eqn:xjN_h} the \textit{same} Brownian motion $W_j(t)$ of the definition of $x_j(t)$ in \eqref{eqn:xj_h}. Then, for the resulting processes we prove 
\begin{equation}
\label{eqn:unif_timex}
\sup_{0 \leq t \leq T N^{2-\varepsilon} }\mathbb{E}\Bigg[\Big(x_j^N(t) - x_j(t)\Big)^2\Bigg] \xrightarrow{N \to +\infty} 0,
\end{equation}
\begin{equation}
\label{eqn:unif_timem}
\sup_{0 \leq t \leq T N^{2/3-\varepsilon}}\mathbb{E}\Bigg[\Big|m_j^N(t) - \tilde{m}_j(t)\Big|\Bigg] \xrightarrow{N \to +\infty} 0,
\end{equation}
which imply the limits in distribution $(i)$ and $(ii)$.
First of all we observe that, for any $t \geq 0$, we have $\mathbb{E}[x_j(t)] = 0$ and $X^N(t) = \frac{\sigma}{N}W(t)$, with $W(t) := \frac{1}{\sqrt{N}}\sum_{k=1}^N W_k(t)$. For \eqref{eqn:unif_timex}, by Itô's formula, we compute
\begin{equation}
\label{eqn:ultima_a}
\begin{aligned}
&\mathbb{E}\Big[(x_j^N(t) - x_j(t))^2\Big] = \mathbb{E}\Big[(x_j^N(0) - x_j(0))^2\Big]- \frac{2 \alpha_2}{N} \int_0^t\mathbb{E}\Big[(x_j^N(s) - x_j(s))^2\Big]ds \\
& \hspace{0.1cm} - \frac{2 \alpha_2}{N}\int_0^t \mathbb{E}\Big[(x_j^N(s)-x_j(s)) X^N(s)\Big]ds\\
&  \hspace{0.1cm} \leq \! \mathbb{E}\Big[(x_j^N(0) - x_j(0))^2\Big] \!-\! \frac{2 \alpha_2}{N} \int_0^t\mathbb{E}\Big[(x_j^N(s) - x_j(s))^2\Big]ds \!+ \frac{2 \alpha_2}{N}\int_0^t \mathbb{E}\Big[|x_j^N(s)-x_j(s)|| X^N(s)|\Big]ds\\
&  \hspace{0.1cm} \leq \mathbb{E}\Big[(x_j^N(0) - x_j(0))^2\Big]- \frac{2 \alpha_2}{N} \int_0^t\mathbb{E}\Big[(x_j^N(s) - x_j(s))^2\Big]ds + \frac{\alpha_2}{N}\int_0^t\mathbb{E}\Big[(x_j^N(s) - x_j(s))^2\Big]ds\\
&  \hspace{0.1cm} + \frac{\alpha_2}{N}\int_0^t\mathbb{E}\Big[(X^N(s))^2\Big]ds,
\end{aligned}
\end{equation}
where in the last estimate we have used $ab \leq \frac{a^2}{2} + \frac{b^2}{2}$. By definition, we have
\begin{equation}
\label{eqn:ult_auxxx}
\int_0^t \mathbb{E}[(X^N(s))^2] ds =  \frac{\sigma^2}{N^2}\int_0^t \mathbb{E}[(W(s))^2] ds  = \frac{1}{N^2}\sigma^2 \frac{t^2}{2}.
\end{equation}
Using \eqref{eqn:ult_auxxx} in the right hand side of \eqref{eqn:ultima_a}, we have found
$$
\mathbb{E}[(x_j^N(t) - x_j(t))^2]  \leq \mathbb{E}[(x_j^N(0) - x_j(0))^2] - \frac{\alpha_2}{N} \int_0^t\mathbb{E}[(x_j^N(s)- x_j(s))^2]ds  + \frac{\alpha_2}{N^3}\sigma^2 \frac{t^2}{2},
$$
which, denoting with $c(t) := \mathbb{E}[(x_j^N(t) - x_j(t))^2]$, in differential form reads
$$
\dot{c}(t) \leq - \frac{\alpha_2}{N} c(t)+ \frac{\alpha_2}{N^3}\sigma^2 t.
$$
By solving the differential equation on the right hand side of the inequality, we deduce 
\begin{equation}
\label{eqn:c_t_N2}
c(t) \leq e^{-\frac{\alpha_2}{N} t} c(0) + \frac{\sigma^2}{N \alpha_2}(e^{-\frac{\alpha_2}{N} t}  - 1) + \frac{\sigma^2}{N^2} t.
\end{equation}
Note that $c(0) =0$ because of our choices of initial data. When we take the supremum over $t$ in the above expression the dominant term is $ \frac{\sigma^2}{N^2} t$, which still tends to $0$ with $N$ going to infinity, if the supremum is taken over $0 \leq t \leq T N^{2-\varepsilon}$, so that \eqref{eqn:unif_timex} is proved.

Moreover, we have 
\begin{equation}
\label{eqn:useful}
\int_0^t\mathbb{E}\Bigg[|x_j^N(s) - x_j(s)|\Bigg] ds \leq  \frac{C}{N}t^{\frac{3}{2}}.
\end{equation}
Indeed, by Jensen and Hölder inequalities and by \eqref{eqn:c_t_N2}, we estimate
\begin{align*}
\Bigg(\int_0^t \mathbb{E}&\Bigg[\Big|x_j^N(s) - x_j(s)\Big|\Bigg]ds\Bigg)^2 = \Bigg(\frac{t}{t}\int_0^t \mathbb{E}\Bigg[\Big|x_j^N(s) - x_j(s)\Big|\Bigg]ds\Bigg)^2\\
 \leq t \int_0^t \mathbb{E}&\Bigg[\Big|x_j^N(s) - x_j(s)\Big|^2\Bigg]ds \leq t^2 \sup_{s \in [0,t]}\left[\frac{\sigma^2}{N\alpha_2}(e^{-\frac{\alpha_2}{N}s} - 1) + \frac{\sigma^2}{N^2}s\right] \leq \frac{\sigma^2}{N^2}t^3
 \end{align*}
so that \eqref{eqn:useful} follows by taking the square root. 
Note also that
\begin{align*}
\int_0^t & \mathbb{E}\Bigg[|X^N(s)|\Bigg] ds = \frac{\sigma}{N} \int_0^t \mathbb{E}\big[|W(s)|\big]ds \leq \frac{C}{N} t^{\frac{3}{2}},
\end{align*}
since $|X^N(s)| = \frac{1}{N}|W(s)|$, and $\mathbb{E}[|W(s)|] \leq C \sqrt{s}$.

Finally, for proving \eqref{eqn:unif_timem} we compute (using $\text{sign}(x)\cdot x = |x|$), 
\begin{align*}
\mathbb{E}&\Bigg[\Big|m_j^N(t) - \tilde{m}_j(t)\Big|\Bigg] = \mathbb{E}\Bigg[\Big|m_j^N(0) - \tilde{m}_j(0)\Big|\Bigg] -2\int_0^t \mathbb{E}\Bigg[\Big|m_j^N(s) - \tilde{m}_j(s)\Big|\Bigg]ds \\
& + 2 \int_0^t \mathbb{E}\Bigg[\text{sign}(m_j^N(s) - \tilde{m}_j(s))\Big(\tanh(\beta_1(x_j^N(s) + m_j^N(s)) + \beta_2 (M^N(s) + X^N(s)) - \\
& -\tanh(\beta_1(x_j(s) + \tilde{m}_j(s)) + \beta_2 M(s))  \Big)\Bigg]ds .
\end{align*}
Using the Lipschitz properties of $\tanh(\cdot)$ and the boundedness of the magnetizations processes we can estimate
\begin{align*}
\mathbb{E}&\Bigg[\Big|m_j^N(t) - \tilde{m}_j(t)\Big|\Bigg] \leq  \mathbb{E}\Bigg[\Big|m_j^N(0) - \tilde{m}_j(0)\Big|\Bigg] -2(1-\beta_1)\int_0^t \mathbb{E}\Bigg[\Big|m_j^N(s) - \tilde{m}_j(s)\Big|\Bigg]ds \\
& + 2\beta_1 \int_0^t \mathbb{E}\Bigg[\Big|x_j^N(s) - x_j(s)\Big|\Bigg]ds+ 2\beta_2 \int_0^t \mathbb{E}\Bigg[\Big|M^N(s) - M(s)\Big|\Bigg]ds+ 2 \beta_2\int_0^t \mathbb{E}\Bigg[|X^N(s)|\Bigg]ds.
\end{align*}
Denoting $\tilde{M}^N(t) := \frac{1}{N}\sum_{j=1}^N \tilde{m}_j(t)$, and $\mu^N(x_1,\dots,x_N) := \frac{1}{N}\sum_{j=1}^N \delta_{x_j}$, we have 
\begin{align*}
\mathbb{E}&\Bigg[\Big|\tilde{M}^N(t) - M(t)\Big|\Bigg] = \mathbb{E}\Bigg[\Bigg|\int_{\mathbb{R}} m(t)(x) (\mu^N - \mu)(dx)\Bigg|\Bigg]\leq ||\mu^N(t) - \mu(t)||_{d_1} \leq \frac{C}{\sqrt{N}},
\end{align*}
where $d_1$ is the $1$-Wasserstein metric, and the estimate follows by LLN. Furthermore, we have
\begin{align*}
\mathbb{E}&\Bigg[\Big|M^N(s) - M(s)\Big|\Bigg] \leq \mathbb{E}\Bigg[\Big|M^N(s) - \tilde{M}^N(s)\Big|\Bigg] +\mathbb{E}\Bigg[\Big|\tilde{M}^N(s) - M(s)\Big|\Bigg] \\
& \leq \mathbb{E}\Bigg[\Big|m_j^N(s) - \tilde{m}_j(s)\Big|\Bigg] + \mathbb{E}\Bigg[\Big|\tilde{M}^N(s) - M(s)\Big|\Bigg],
\end{align*}
where in the last estimate we have used the exchangeability of the magnetizations processes.
Finally, we can collect all the previous estimates to get
\begin{align*}
\mathbb{E}&\Bigg[\Big|m_j^N(t) - \tilde{m}_j(t)\Big|\Bigg]  \leq  \mathbb{E}\Bigg[\Big|m_j^N(0) - \tilde{m}_j(0)\Big|\Bigg] -2(1-\beta_1-\beta_2)\int_0^t \mathbb{E}\Bigg[\Big|m_j^N(s) - \tilde{m}_j(s)\Big|\Bigg]ds\\
&+ \frac{C_1}{N}t^{3/2} + C_2\frac{t}{\sqrt{N}}.
\end{align*}
In differential form, with $c(t) := \mathbb{E}\Bigg[\Big|m_j^N(t) - \tilde{m}_j(t)\Big|\Bigg]$, $k:= 2(1- \beta_1-\beta_2) > 0$, the previous estimate reads
$$
\dot{c}(t) \leq -k c(t) +  \frac{C_1}{N}t^{1/2} + \frac{C_2}{\sqrt{N}},
$$
implying
$$
c(t) \leq e^{-kt}c(0) +\frac{C}{N}t^{3/2} + \frac{C}{\sqrt{N}}.
$$
Recalling that $c(0) \xrightarrow{N \to +\infty} 0$ by a LLN, we obtain \eqref{eqn:unif_timem} when we take the supremum for $0\leq t \leq TN^{2/3 - \varepsilon}$.
\end{proof}


\subsection{Propagation of chaos at times of order $N$: the subcritical case}
\label{orderN}
In this section we adapt the proof of the propagation of chaos to times of order $N$ for the case $\beta_1 + \beta_2 < 1$. Thanks to Theorem \ref{erg1}, in this scale we can assume that the initial data for the processes are given by the long-time limit at the previous timescale of order $1$. For the diffusions it holds $x_j^N(0) = x_j \sim \mathcal{N}\left(0,\frac{1}{N} \frac{\sigma^2}{2\alpha_2}\right)$ for any $j=1,\dots,N$, while the magnetizations are starting the dynamics in the long-time limit symmetric equilibrium $\overline{m}(x)$.
For ease of notation we still denote the sped up processes by $x_j^N(t) := x_j^N(Nt)$, and $m_j^N(t) := m_j^N(Nt)$. 
They evolve according to:
\begin{equation}
\label{eqn:empirical_ordN}
\begin{cases}
m_j^N \! \mapsto \!m_j^N \pm \frac{2}{N} \text{ rate }  N^2 \frac{1  \mp\! m_j^N(t)}{2}\! \left(1 \! \pm \tanh\!\left[\beta_1 (x_j^N(t) \!+\! m_j^N(t))\!+\! \beta_2 (X^N(t) \!+\! M^N(t))\right]\right)\!,\\
m_{j}^N(0) = \overline{m}(x_j),\\
d x_j^N(t) = -\alpha_2\left[x_j^N(t) - X^N(t)\right]dt +\sigma dW_{j}^N(t),\\
x_{j}^N(0) = x_j \sim \mathcal{N}\left(0,\frac{1}{N} \frac{\sigma^2}{2\alpha_2}\right).
\end{cases}
\end{equation}
The limit i.i.d.\! processes to which the sped up processes at order $N$ will be proved to converge are denoted as $(\tilde{x}_j(t), \tilde{m}_j(t))_{j=1,\dots,N}$, 
where $\tilde{m}_j(t) := m(t)(\tilde{x_j}(t))$, with
\begin{equation}
\begin{cases}
\label{eqn:chaos_diff_N}
d \tilde{x}_j(t) = -\alpha_2 \tilde{x}_j(t) dt + \sigma d W_j(t),\\
\tilde{x}_j(0) = 0,
\end{cases}
\end{equation}
with $W_j$'s $N$ independent Brownian motions, and $m(t)(x)$ solves 
\begin{equation}
\label{eqn:iid_N}
\begin{cases}
m(t)(x) = \tanh\left(\beta_1(x + m(t)(x)) + \beta_2 M(t)\right),\\
m(0)(x) \equiv \overline{m}(x),\\
M(t) = \int_{\mathbb{R}} m(t)(x) \mu_t(dx),
\end{cases}
\end{equation}
where $\mu_t(dx)$ is the distribution at time $t$ of the Ornstein-Uhlenbeck i.i.d.\! processes $\tilde{x}_j(t)$'s, and $\overline{m}(x)$ is the solution to Eq. \eqref{eqn:asympt}. Once again, the propagation of chaos for the diffusion processes is standard at this scale (for any fixed interval of time). What we need to prove is the same property for the magnetizations processes,
\begin{thm}[Propagation of chaos at order $N$]
\label{chaos_ordN}
Fix $T > 0$. For any $\beta_1 + \beta_2 < 1$, $\alpha_1,\alpha_2,\sigma >0$, and any $j = 1,\dots,N$, $\Big(m_j^N(t)\Big)_{t \in [0,T]}$ converges weakly in the sense of stochastic processes, for $N \to +\infty$, to $\Big(\tilde{m}_j(t)\Big)_{t \in [0,T]}$.
\end{thm}
Before addressing the proof, we must check that Eq. \eqref{eqn:iid_N} is well-posed. In fact, the limit dynamics \eqref{eqn:iid_N} is trivial at this scale.
\begin{prop}[Well-posedness at order $N$]
\label{wp_N}
For any $\beta_1+\beta_2 < 1$, Eq. \eqref{eqn:iid_N} has a unique classical solution $m : [0,T] \times \mathbb{R} \to [-1,1]$ such that $m(t)(\cdot) \in C(\mathbb{R})$ for any $t \in [0,T]$. Moreover, we have $m(t)(x) = \overline{m}(x)$ and $M(t) = 0$ for any $t \in [0,T]$.
\end{prop}
\begin{proof}
The non-explosiveness of Eq. \eqref{eqn:iid_N} is obvious by construction. Indeed, $m(t)(x) \in [-1,1]$ for any $t \in [0,T]$, $x \in \mathbb{R}$. For the uniqueness, define $F(m)(t)(x):= \tanh\left(\beta_1(x + m(t)(x)) + \beta_2 M(t)\right)$, and consider two solutions $m(t)(\cdot), m'(t)(\cdot) \in C(\mathbb{R})$. Then, we have
\begin{align*}
|F(m) - F(m')|(t)(x) & \leq \max_{\xi \in \mathbb{R}}|1- \tanh^2(\xi)| \Big[\beta_1|m(t)(x) - m'(t)(x)| + \beta_2|M(t) - M'(t)|\Big]\\
& \leq \beta_1|m(t)(x) - m'(t)(x)| + \beta_2|M(t) - M'(t)|.
\end{align*}
By taking the sup over $x \in \mathbb{R}$, $||F(m)(t) - F(m')(t)||_{\infty} \leq (\beta_1+\beta_2)||m(t) - m'(t)||_{\infty}$, since $|M(t) - M'(t)| \leq \int_{\mathbb{R}}|m(t)(x) - m'(t)(x)| \mu_t(dx) \leq ||m(t) - m'(t)||_{\infty}$.
Thus, we can conclude the uniqueness of solution by a contraction argument when $\beta_1 + \beta_2 < 1$. Moreover, the triviality of the dynamics is due to the symmetry around zero of the distribution $\mu_t(dx) \sim \mathcal{N}\left(0, \frac{\sigma^2}{2\alpha_2}(1- e^{-2\alpha_2 t})\right)$, for which we have that $M(t) \equiv 0$ for any $t$, and thus that $m(t)(x) \equiv \overline{m}(x)$ is the unique solution to the dynamics in this regime.
\end{proof}
While the requirement $\beta_1 + \beta_2 < 1$ ensures the uniqueness of solution to the limit dynamics of order $N$, the crucial observation - working for $\beta_1 < 1$ independently of $\beta_2$ - which allows to adapt the previous proof is the following 
\begin{prop}[Contraction estimates]
\label{cont_est}
Let $(x_j^N(t),m_j^N(t))_{j=1,\dots,N}$ the empirical sped up processes at a timescale of order $N$. Let
$$
y_j(t) := m_j^N(t) - \tanh\left(\beta_1(x_j^N(t) + m_j^N(t)) + \beta_2(X^N(t) + M^N(t))\right).
$$
Then, for any $\beta_1 < 1$, $k > 0$, $j = 1,\dots,N$,
\begin{equation}
\label{eqn:h_contr}
N \mathcal{L}^N |y_j(t)| ^k \leq -C N |y_j(t)|^k + O(1),
\end{equation}
for some $C:= C(\beta_1,k) > 0$, where $O(1)$ is uniform in time and space and $\mathcal{L^N}$ is given by \eqref{eqn:h_gen}.
\end{prop}
\begin{proof}
The proof uses analogous arguments to the ones used in the mean field case for obtaining \eqref{eqn:contracting}. For simplicity, we use the coordinates $(\lambda_j, m_j)$ instead of $(x_j, m_j)$.
Applying the accelerated generator in the other coordinates to the function $y_j^k(t)$, and expanding to the second order in $(m_j,\lambda_j)$, we get
\begin{align*}
N\mathcal{L}^N |y_j(t)|^k & \leq- 2N \left[m_j^N(t) - \tanh\left(\beta_1 \lambda_j^N(t) + \beta_2\Lambda^N(t)\right)\right]\left[\frac{\partial}{\partial m_j} |y_j(t)|^k + \frac{\partial}{\partial \lambda_j} |y_j(t)|^k\right] + O(1)\\
& = - 2 N y_j(t) \left[\frac{\partial}{\partial m_j} |y_j(t)|^k + \frac{\partial}{\partial \lambda_j} |y_j(t)|^k\right] + O(1).
\end{align*}
The $O(1)$ follows from the fact that both $y_j$ and the coefficients appearing in the higher order terms of the generator are uniformly bounded by some constant $C$ not depending on time nor space. Indeed, the dominating remainder terms of the development are the second order terms, which in the accelerated timescale of order $N$ are of order $1$. Computing
\begin{equation*}
\frac{\partial}{\partial m_j} |y_j|^k \!+\! \frac{\partial}{\partial \lambda_j} |y_j|^k \!=\! k \left|y_j\right|^{k-1} \text{sign}(y_j)\left[1 - \left(\beta_1 + \frac{\beta_2}{N}\right)\!\left( 1\! -\! \tanh^2\left(\beta_1\lambda_j^N \!+\! \beta_2 \Lambda^N\right)\right)\right], 
\end{equation*}
we see that the factor $\frac{\beta_2}{N}$ can be included in the terms of order $O(1)$.
Thus, using that $x \cdot \text{sign}(x) = |x|$, we have
\begin{align*}
N\mathcal{L}^N |y_j(t)|^k & \leq  - 2kN |y_j(t)|^k\left[1 - \beta_1\left( 1 - \tanh^2\left(\beta_1\lambda_j^N(t) + \beta_2 \Lambda^N(t)\right)\right)\right] + O(1).
\end{align*}
Finally, observing that the function $f(\lambda_j) := \left[1 - \beta_1\left( 1 - \tanh^2\left(\beta_1\lambda_j^N + \beta_2 \Lambda^N\right)\right)\right]$ is always positive for $\beta_1 < 1$ and has a unique minimum for $\lambda_j^* = -\frac{k}{\beta_1 + \frac{\beta_2}{N}}$, with $k = \beta_2 \frac{1}{N}\sum_{k \neq j} \lambda_k$ such that $f(\lambda_j^*) = 1 -\beta_1$, we can conclude by choosing $C(\beta_1,k) := k (1-\beta_1)$.
\end{proof}

\begin{rem}
\label{any_scale}
Proposition \ref{cont_est} can be trivially generalized to any timescale of order $N^m t$, yielding
$$
N^m \mathcal{L}^N |y_j^m(t)| ^k \leq -C N^m |y_j^m(t)|^k + O(N^{m-1}),
$$ 
with $y_j^m(t) := m_j^N(N^m t) - \tanh\left(\beta_1(x_j^N(N^m t) \!+\! m_j^N(N^m t)) \!+\! \beta_2(X^N(N^m t) \!+\! M^N(N^m t))\right)$.
\end{rem}

\begin{cor}
\label{cor_coll}
Let $y_j^m(t)$ be defined as in Remark \ref{any_scale}. Then, for any $T > 0$, $k > 0$, $m =1,2$ 
\begin{equation}
\label{eqn:strong_collapse}
\mathbb{E}\left[\sup_{t \in [0,T]} |y_j^m(t)|^k\right] \leq C(N,m,k),\\
\end{equation}
with $C(N,m,k) \xrightarrow{N \to +\infty} 0$.
\end{cor}
\begin{proof}
Observing that the infinitesimal generator of the processes $(x^N_j, m^N_j)$ at a timescale of order $N^m$ is $N^m\mathcal{L}^N$, from the contraction estimates \eqref{eqn:h_contr} generalized as in Remark \ref{any_scale} it follows
$$
\frac{d}{dt}\mathbb{E}\Big[|y_j(t)|^k\Big] \leq - C N^m \mathbb{E}\Big[|y_j(t)|^k\Big] + O(N^{m-1}).
$$ 
Integrating both sides with respect to time we then get claim for any time $t \in [0,T]$, provided that the assertion is true for the initial datum. More precisely, the previous estimate implies
\begin{align*}
\mathbb{E}\Big[|y_j(t)|^k\Big] &\leq  e^{-C_1N^{m} t } \mathbb{E}\Big[|y_j(0)|^k\Big] - C_2 \frac{N^{m-1}}{N^m} e^{-C_1N^{m} t} + C_2 \frac{N^{m-1}}{N^m} \\
& = e^{-C_1N^{m} t } \mathbb{E}\Big[|y_j(0)|^k\Big] - \frac{C_2}{N} e^{-C_1N^{m} t} + \frac{C_2}{N}.
\end{align*}
Thus, $\sup_{t \geq 0} \mathbb{E}\Big[|y_j(t)|^k\Big] \leq \mathbb{E}\Big[|y_j(0)|^k\Big] + \frac{C}{N}$. Note that by the assumptions on the initial data we have by a LLN that $\mathbb{E}\Big[|y_j(0)|^k\Big] \xrightarrow{N \to +\infty} 0$. This works both at a timescale of order $N$ and $N^2$.
For getting the stronger convergence \eqref{eqn:strong_collapse} we again refer to Section 4 of \cite{collet} for the diffusive case and to the Appendix of \cite{comets} for a  general proof for jump processes. We can then conclude as we did in the proof of Proposition \ref{mf_contr} for the mean field case.
\end{proof}

\begin{proof}[Proof of Theorem \ref{chaos_ordN}]
As we repeatedly did above, we plug in the definition of the sped up diffusions $x_j^N(t)$ the \textit{same} Brownian motion $W_j(t)$ appearing in the definition of the limit process $\tilde{x}_j(t)$ in \eqref{eqn:chaos_diff_N}. The weak convergence in distribution is then implied by
\begin{equation}
\label{eqn:chaosN}
\lim_{N\to \infty}\mathbb{E}\Bigg[\sup_{t \in [0,T]}\big| m_j^N(t) - \tilde{m}_j(t)\big|\Bigg]= 0,
\end{equation}
for the resulting processes, since $W^N_j \stackrel{\mathcal{D}}{=} W_j$ for $j = 1,\dots,N$.
First, we estimate
\begin{align*}
\mathbb{E}&\Bigg[\!\sup_{s \in [0,t]}\!\!\big|m_j^N(s) \!-\! \tilde{m}_j(s) \big|\Bigg] \!\leq \mathbb{E}\Bigg[\!\sup_{s \in [0,t]} \! \Big|m_j^N(s) - \tanh\left(\beta_1(x_j^N(s) \!+\! m_j^N(s)) + \beta_2(X^N(s) \!+\! M^N(s))\right)\Big|\Bigg] \\
& \hspace{0.2cm}+  \mathbb{E}\Bigg[\sup_{s \in [0,t]} \Big|\tanh\left(\beta_1(x_j^N(s) + m_j^N(s)) + \beta_2(X^N(s) + M^N(s))\right)\\
& \hspace{0.7cm} - \tanh\left(\beta_1 (\tilde{x}_j(s) + \tilde{m}_j(s)) + \beta_2 M(s)\right)\Big|\Bigg].
\end{align*}
The first term in the right hand side of the above inequality is dealt with the contraction estimates of Corollary \ref{cor_coll} for $m = k = 1$. For the other term we use the global Lipschitz continuity of $\tanh(\cdot)$ in the following way:
\begin{align*}
\mathbb{E}&\Bigg[\!\sup_{s \in [0,t]} \!\Big|\tanh\left(\beta_1(x_j^N(s) \!+\! m_j^N(s)) + \beta_2(X^N(s) \!+\! M^N(s))\right)- \tanh\left(\beta_1 (\tilde{x}_j(s) \!+\! \tilde{m}_j(s)) \!+\! \beta_2 M(s)\right)\Big|\Bigg] \\
& \leq \beta_1 \mathbb{E}\Bigg[\sup_{s \in [0,t]} \Big|x_j^N(s) - \tilde{x}_j(s)\Big|\Bigg] + \beta_1\mathbb{E}\Bigg[\sup_{s \in [0,t]} \Big|m_j^N(s) - \tilde{m}_j(s)\Big|\Bigg] \\
& \hspace{0.2cm}+ \beta_2 \mathbb{E}\Bigg[\sup_{s \in [0,t]} \Big| X^N(s)\Big|\Bigg] + \beta_2\mathbb{E}\Bigg[\sup_{s \in [0,t]}\Big|M^N(s) - M(s)\Big|\Bigg].
\end{align*}
For standard arguments of propagation of chaos for the interacting diffusions we have
$$
\mathbb{E}\Bigg[\sup_{s \in [0,t]} \Big|x_j^N(s) - \tilde{x}_j(s)\Big|\Bigg] \leq C_1(N), \quad \mathbb{E}\Bigg[\sup_{s \in [0,t]} \Big| X^N(s)\Big|\Bigg] \leq C_2(N),
$$
with $C_{1,2}(N) \xrightarrow{N \to +\infty}  0$. For the term $\mathbb{E}\Bigg[\sup_{s \in [0,t]}\Big|M^N(s) - M(s)\Big|\Bigg]$ we proceed by a coupling as in the proofs of Theorem \ref{chaos_ord1}, to get
$$
\mathbb{E}\Bigg[\sup_{s \in [0,t]}\Big|M^N(s) - M(s)\Big|\Bigg] \leq C(N) + \mathbb{E}\Bigg[\sup_{s \in [0,t]}\big|m_j^N(s) - \tilde{m}_j(s) \big|\Bigg].
$$
Recollecting all the estimates, we have found 
\begin{align*}
(1 - \beta_1 - \beta_2)\mathbb{E}\Bigg[\sup_{s \in [0,t]}&\big|m_j^N(s) - \tilde{m}_j(s) \big|\Bigg] \leq C(N) \xrightarrow{N \to +\infty}  0.
\end{align*}
Thanks to the hypothesis $\beta_1 + \beta_2 < 1$ we get \eqref{eqn:chaosN}.
\end{proof}

\begin{rem}
\label{true_chaosN}
The analogous to Remark \ref{true_chaos1} holds here, implying the propagation of chaos in the classic sense.
\end{rem}

In words, we have found that in the subcritical regime $\beta_1 + \beta_2 < 1$ the equilibrium that the dynamics reaches for long times of order $1$ is the same as the equilibrium of the dynamics at long times of order $N$. The limit dynamics is thus a process moving across the equilibria, due to the movement of the limit diffusion $x(t)$. In particular, define the limit order $N$ dynamics as the pair of processes $(x(t),m(t))_{t\geq 0}$ satisfying
\begin{equation}
\label{eqn:subcritical_hierarchical_diff}
\begin{cases}
m(t) = \tanh(\beta_1(x(t) + m(t)) + \beta_2 M(t)),\\
dx(t) = -\alpha_2 x(t) + \sigma dW(t),\\
M(t) = \mathbb{E}[m(t)],\\
m(0) = 0,\\
x(0) = 0,
\end{cases}
\end{equation}
for which it holds $\Big(\tilde{m}_j(t)\Big)_{t \in [0,T]}\!\!\!\! \stackrel{\mathcal{D}}{=} \Big(m(t)\Big)_{t \in [0,T]}$ for any $j=1,\dots,N$. Then, we have the analogous to Proposition \ref{limit_sub_diff}:
\begin{prop}
\label{1d_diff_ordN}
The process $\Big(m(t)\Big)_{t \geq 0}$ defined in \eqref{eqn:subcritical_hierarchical_diff} is a strong solution to
\begin{equation}
\label{eqn:limit_sub_diff_hier}
\begin{cases}
dm(t) = \left[-\frac{\alpha_2\beta_1(1-m^2(t))\left(\frac{1}{\beta_1}\arctanh(m(t)) - m(t)\right)}{1-\beta_1(1-m^2(t))} - \frac{\beta_1^2 \sigma^2 m(t) \left(1-m^2(t)\right)}{\left(1-\beta_1(1-m^2(t))\right)^3}\right]dt \\
\hspace{1.5cm}+\frac{\sigma \beta_1 (1-m^2(t))}{1-\beta_1(1-m^2(t))}dW(t),\\
m(0) = 0.
\end{cases}
\end{equation}
\end{prop}
\begin{proof}
By Proposition \ref{wp_N} it follows that $M(t) \equiv 0$. Thus, by Eq. \eqref{eqn:subcritical_hierarchical_diff} we have that $m(t)$ can be written as an \textit{explicit} function of $x(t)$. We can then perform analogous computations as in the proof of Proposition \ref{limit_sub_diff}, with the only difference that now $(x(t))_{t \geq 0}$ is an Ornstein-Uhlenbeck process instead of a Brownian motion. Still, $m(t)$ must be of the form 
$$
d m(t) = a(t,m(t))dt + b(t,m(t)) d W(t)
$$
for some functions $a,b : [0,\infty) \times [-1,1] \to \mathbb{R}$ to be determined, and $W(t)$ is the same Brownian motion appearing in the dynamics of $x(t)$ as in \eqref{eqn:subcritical_hierarchical_diff}.
By applying Itô's formula to the function $\tanh(\beta_1(x(t) + m(t))$ we find \eqref{eqn:limit_sub_diff_hier}.
\end{proof}
\begin{rem}
\label{subcritic_diffusion_is_well-posed_hierar}
The analogous statement to Remark \ref{subcritic_diffusion_is_well-posed} holds: for $\beta_1 < 1$, the SDE \eqref{eqn:limit_sub_diff_hier} is well-posed. Indeed, note that \eqref{eqn:limit_sub_diff_hier} differs from \eqref{eqn:limit_sub_diff} only by an additional drift, following by the O-U dynamics of $x(t)$, which is regular and tends to $0$ at the borders of $(-1,1)$ (observe that $(1-x^2)\arctanh(x) \to 0$ when $x \to \pm 1$).
\end{rem}

\begin{rem}
\label{modified_ordN}
Analogously to Remark \ref{modified_ord1} for the order $1$ case, we can generalize Proposition \ref{wp_N} and Theorem \ref{chaos_ordN} to the case where the initial data for the diffusions are centered around $\overline{X} \neq 0$. The limit order $N$ equation becomes
\begin{equation}
\label{eqn:modified_ordN}
\begin{cases}
m(t)(x) = \tanh\left(\beta_1(x + m(t)(x)) + \beta_2 (\overline{X}+M(t))\right),\\
m(0)(x) \equiv \overline{m}_{\overline{X}}(x),\\
M(t) = \int_{\mathbb{R}} m(t)(x) \mu_t(dx;\overline{X}),
\end{cases}
\end{equation}
for some $\overline{X}\in \mathbb{R}$, where $\overline{m}_{\overline{X}}(x)$ is the solution to \eqref{eqn:modified_longtime}, and $\mu_t(dx;\overline{X})$ is a normal distribution with mean $\overline{X}$ and variance depending on time (the distribution of the Ornstein-Uhlenbeck diffusions). 
Note that in this case dynamics \eqref{eqn:modified_ordN} is not trivial: $M(t)$ fluctuates around an equilibrium point due to the time-dependent variance of the Ornstein-Uhlenbeck diffusions, where the equilibrium point depends both on the given $\overline{X}$ and on the parameters of the diffusions $\sigma$ and $\alpha_2$. In the long run, $M(t) \to  M(\infty) := \int_{\mathbb{R}} m(t)(x) \mu_\infty(dx;\overline{X})$, with $\mu_\infty = \mathcal{N}\left(\overline{X}, \frac{\sigma^2}{2\alpha_2}\right)$.
An analogous equation to \eqref{eqn:limit_sub_diff_hier} can also be written, by adding an additional drift term following by the fact that $x(t) = \frac{1}{\beta_1}\arctanh(m(t)) - m(t) - \frac{\beta_2}{\beta_1}(M(t) + \overline{X})$. Due to the term $M(t) = \mathbb{E}[m(t)]$ the resulting equation is a diffusion of McKean-Vlasov type.
\end{rem}

\subsection{Dynamics at times of order $N^2$: the subcritical case}
\label{orderN^2}
At this timescale a refined study of the interacting diffusions is needed to describe the limit dynamics. Denoting with $t$ the macroscopic time of order $N^2$, the single $x_j^N$'s evolve at a much faster timescale with respect to the current value of their empirical mean $X^N(t)$, which is not anymore zero but evolves randomly as a Brownian motion with constant diffusion coefficient $\sigma$. Thus, one can expect that in an infinitesimal time $dt$ of order $N^2$ the single diffusions become asymptotically independent and reach their equilibrium distribution \textit{given} the current value of $X^N(t) = \overline{X}$. In turns, in the same $dt$ the magnetization's processes are also asymptotically i.i.d.\! and reach an equilibrium given by a macroscopic magnetization $\overline{M}$, whose value can be read off from \eqref{eqn:modified_ordN} in Remark \ref{modified_ordN}, substituting $\mu_t$ with $\mu_{\infty}$, the ergodic measure of the Ornstein-Uhlenbeck processes. The reiteration of this procedure for any $dt$ describes the dynamics at the order $N^2$. In particular, the latter does not propagate chaos, unless we condition it with respect to $X^N(t)$.

As before, we still denote the sped up processes under the same notation, $x_j^N(t) := x_j^N(N^2t)$, and $m_j^N(t):= m_j^N(N^2 t)$,
using as initial data the long-time limit at the previous timescale of order $N$. For clarity we write them again:
\begin{equation}
\begin{cases}
\label{eqn:ord2_diff}
d x_j^N(t) = -N\alpha_2(x_j^N(t) - X^N(t)) dt + \sqrt{N} \sigma d W_j^N(t),\\
x_j^N(0) = x_j \sim \mathcal{N}\left(0,\frac{\sigma^2}{2\alpha_2}\right),
\end{cases}
\end{equation}
with $X^{N}(t) := \frac{1}{N}\sum_{k=1}^N x_k^N(t)$.
The dynamics of the magnetizations is now given by
\begin{equation}
\label{eqn:empirical_ordN^2}
\begin{cases}
m_j^N \! \mapsto \!m_j^N \pm \frac{2}{N}  \text{ rate }  N^3 \frac{1 \mp m_j^N(t)}{2}\! \left(\!1 \!\pm \tanh\!\left[\beta_1 (x_j^N(t) \!+\! m_j^N(t))\!+\! \beta_2 (X^N(t) \!+\! M^N(t))\right]\!\right),\\
m_{j}^N(0) = \overline{m}(x_j).
\end{cases}
\end{equation} 
At this level, we aim to prove that the conditional distribution  of the empirical macroscopic magnetization $M^N(t)$ with respect to $X^N(t)$ converges to the conditional distribution of $M(t)$ given $X(t)$ (which is actually a delta), with
\begin{equation}
\label{eqn:iid_N^2}
\begin{cases}
m(t)(x) = \tanh\left(\beta_1(x + m(t)(x)) + \beta_2(X(t) + M(t))\right),\\
m(0)(x) \equiv \overline{m}_{X(0)}(x),\\
M(t) = \int_{\mathbb{R}} m(t)(x) \mu_\infty(dx;X(t)),
\end{cases}
\end{equation}
where $\mu_\infty(dx;X(t)) = \mathcal{N}\left(X(t), \frac{\sigma^2}{2\alpha_2}\right)$ must be intended as a conditional distribution \textit{given} the current realization of $X(t)$, whose random evolution is 
\begin{equation}
\label{eqn:diffX_N2}
\begin{cases}
d X(t) = \sigma d W(t),\\
X(0) = 0,
\end{cases}
\end{equation}
with $W$ a Brownian motion. Moreover, denoting with 
\begin{equation}
\label{eqn:transition_kernel}
Q_t(0,dX) = \frac{1}{\sqrt{2\pi \sigma^2 t}}e^{-\frac{X^2}{2\sigma^2 t}} dX
\end{equation}
the transition kernel's density at time $t$ associated to the limit diffusion \eqref{eqn:diffX_N2}, we also prove the convergence of the full law of $M^N(t)$ to the law of the process $M(t)$ defined by 
\begin{equation}
\label{eqn:full_law_M(t)}
\begin{cases}
m(t)(x) = \tanh\left(\beta_1(x + m(t)(x)) + \beta_2(X(t) + M(t))\right),\\
m(0)(x) \equiv \overline{m}_{X(0)}(x),\\
M(t) = \int_{\mathbb{R}} m(t)(x) \tilde{\mu}^t(dx),
\end{cases}
\end{equation}
with 
\begin{equation}
\label{eqn:mu_tilde_final}
\tilde{\mu}^t(\cdot) := \int_\mathbb{R}Q_t(0,d\overline{X})\mu_\infty(\cdot;\overline{X}).
\end{equation}
\begin{thm}[Limit dynamics at order $N^2$]
\label{chaos_ordN^2}
For any $T > 0$, $\beta_1,\beta_2 > 0$ such that $\beta_1 + \beta_2 < 1$ and $\alpha_1,\alpha_2,\sigma >0$
\begin{itemize}
\item[(i)] For all the finite dimensional distributions of the form $(t_1,\dots, t_k) \in [0,T]^k$, it holds
\begin{equation}
\label{eqn:conv_fixed_t_M}
\text{Law}\Big(M^N(t_1),\dots,M^N(t_k)\Big) \xrightarrow{N \to +\infty} \text{Law}\Big(M(t_1),\dots,M(t_k)\Big),
\end{equation}
with $M(t)$ the process defined by \eqref{eqn:full_law_M(t)} and \eqref{eqn:mu_tilde_final}.
\item[(ii)] For every $t \in [0,T]$,
\begin{equation}
\label{eqn:convergenza_a_delta}
\text{Law}\Big(M^N(t)\Big| |X^N(t) - X| \leq \varepsilon_N\Big) \xrightarrow{N \to +\infty} \delta_{M(t)},
\end{equation}
with $M(t)$ the (deterministic) variable defined by \eqref{eqn:iid_N^2} with $X(t) = X$, and $\varepsilon_N \xrightarrow{N \to +\infty} 0$.
\item[(iii)](Conditional propagation of chaos) For every $t \in [0,T]$ and every $k$-tuple of distinct indexes $j_1,\dots,j_k \in \left\{1,\dots,N\right\}^k$, we have 
\begin{equation}
\label{eqn:cond_prop_chaos_m}
\text{Law}\Big(m_{j_1}^N(t),\dots,m_{j_k}^N(t) \Big| |X^N(t) - X| \leq \varepsilon_N\Big) \xrightarrow{N \to +\infty} \text{Law}\Big(\tilde{m}_{j_1}(t),\dots,\tilde{m}_{j_k}(t)\Big) = \text{Law}\Big(\tilde{m}_{j_1}(t)\Big)^k,
\end{equation}
where $\tilde{m}_{j_i}(t):= m(t)(x_{j_i})$, with $m(t)(x)$ given by \eqref{eqn:iid_N^2} with $X(t) = X$, the $x_{j_i}$'s are i.i.d.\! random variables distributed as $x \sim \mu_\infty(dx;X) = \mathcal{N}\left(X,\frac{\sigma^2}{2\alpha_2}\right)$, and $\varepsilon_N \xrightarrow{N \to +\infty} 0$.
\end{itemize}
\end{thm}

Note that the well-posedness of the limit dynamics \eqref{eqn:iid_N^2} and \eqref{eqn:full_law_M(t)} can be proved in the same way as we did for the order $N$ case in Proposition \ref{wp_N}, since any two solutions $m(t)$ and $n(t)$ share the same $X(t)$. Moreover, we point out that we expect property $(i)$ to hold in the stronger sense of weak convergence of stochastic processes, though we did not work out a proof yet.
The main ingredients for proving the convergence to the limit at this timescale are provided by Lemmas \ref{coordinates} and \ref{averaging}. The first establishes a handy distributional representation of the interacting diffusions in terms of a combination of (fast) stationary independent Ornstein-Uhlenbeck processes plus a (slow) independent Brownian motion and a small interaction term. Lemma \ref{averaging} involves a sort of Law of Large Numbers/averaging property for non-linear implicit functions of the magnetizations and of the diffusions. In what follows we strongly rely on the Gaussianity of the interacting processes \eqref{eqn:ord2_diff}. Before stating the next result, we need to introduce the following processes. Let $\Big(\xi_j^N(t)\Big)_{j=1,\dots,N}$ be defined as,  
\begin{equation}
\label{eqn:ou_auxiliary}
\begin{cases}
d \xi_j^N(t) = -\alpha_2N \xi_j^N(t) dt + \sigma \sqrt{N} dW_j(t),\\
\xi_j^N(0) \sim \mathcal{N}\left(0,\frac{\sigma^2}{2\alpha_2}\right),
\end{cases}
\end{equation}
with $W_j$'s independent Brownian motions, and set $\overline{\xi}_N(t) := \frac{1}{N}\sum_{j=1}^N \xi_j(t)$. Moreover, let the process $\Big(U_N(t)\Big)_{t \geq 0}$ be defined as
\begin{equation}
\label{eqn:U_N}
\begin{cases}
d U_N(t) = \sigma^2 dW(t),\\
U_N(0) \sim \mathcal{N}\left(0,\frac{\sigma^2}{2\alpha_2 N}\right),
\end{cases}
\end{equation}
with $W$ a Brownian motion independent of all the $W_j$'s. Note that the dependence on $N$ in $U_N(t)$ is only through the initial datum. 
\begin{lem}
\label{coordinates}
Let $(x_j^N(t))_{j=1,\dots,N}$ be as in \eqref{eqn:ord2_diff}. Then, for any $T > 0$, we have that
\begin{itemize}
\item[(i)] For every $j=1,\dots,N$ and every $N \in \mathbb{N}$,
\begin{equation}
\label{eqn:equality_in_law}
\text{Law}\Big((x_j^N(t))_{t \in [0,T]}\Big) = \text{Law}\Big(\big(\xi_j^N(t) - \overline{\xi}_N(t) + U_N(t)\big)_{t \in [0,T]}\Big).
\end{equation}
\item[(ii)] For every $k$-tuple of distinct indexes $(j_1,\dots,j_k) \in \left\{1,\dots,N\right\}^k$ and every fixed $t \in [0,T]$,
\begin{equation}
\label{eqn:conv_in_law_vector}
\text{Law}\Big(x_{j_1}^N(t),\dots,x_{j_k}^N(t)\Big)(dx) = \int_{\mathbb{R}}Q_t(0,dX)\mu_{\infty}^k(dx; X) =: \tilde{\mu}^{t,k}(dx), 
\end{equation}
for every $N \in \mathbb{N}$, with $\mu_{\infty}^k(dx;X) = \mu_{\infty}(dx_1;X)\times \dots \times \mu_{\infty}(dx_k;X)$.
\item[(iii)](Conditional propagation of chaos) For every $k$-tuple of distinct indexes $(j_1,\dots,j_k) \in \left\{1,\dots,N\right\}^k$ and every fixed $t \in [0,T]$, 
\begin{equation}
\label{eqn:conv_in_law_vector_conditional}
\text{Law}\Big(x_{j_1}^N(t),\dots,x_{j_k}^N(t) \Big| |X^N(t) - X| \leq \varepsilon_N\Big)(dx) \xrightarrow{N \to +\infty} \mu_{\infty}^k(dx; X), 
\end{equation}
with $\varepsilon_N \xrightarrow{N \to +\infty} 0$.
\end{itemize}
\end{lem}
\begin{proof}
Because of the Gaussianity of the processes $(x_j^N(t))_{t \geq 0}$, $(\xi_j^N(t))_{t \geq 0}$ and $(U_N(t))_{t \geq 0}$ we can check assertion $(i)$ just by studying the covariance functions.
For a fixed $t \geq 0$, denote $A(t) := \mathbb{E}[(x_j^N(t))^2]$ and $B(t) := \mathbb{E}[x_j^N(t) x_i^N(t)]$. Because of the exchangeability of the processes $(x_j^N(\cdot))_{j=1,\dots,N}$ we have that $A$ and $B$ do not depend on $j$ nor $i$. Applying Itô's formula to $f(x_j^N(t)) = (x_j^N(t))^2$ and to $f(x_j^N(t),x_i^N(t)) = x_j^N(t) x_i^N(t)$, and then taking the expectation, we obtain a system of two ODEs for $A(t)$ and $B(t)$, whose solution is given by
\begin{equation}
\label{eqn:rel_imp}
A(t) = \frac{\sigma^2(1+2\alpha_2 t)}{2\alpha_2}, \qquad B(t) =\sigma^2 t.
\end{equation}
Now, fix any $s,t \geq 0$ with $t > s$. Denote $A_N(s,t) := \mathbb{E}[x_j^N(s)x_j^N(t)]$ and $B_N(s,t) := \mathbb{E}[x_j^N(s)x_i^N(t)]$. Clearly, we have $A_N(s,s) = A(s)$ and $B_N(s,s) = B(s)$. The evolution in $t$ of the above quantities can be obtained by applying Itô's formula to $x_j^N(s)x_j^N(t)$ and $x_j^N(s)x_i^N(t)$ on the time interval $[s,t]$. As above, we obtain a system of two ODEs in $t \in [s,+\infty)$, with initial data provided by \eqref{eqn:rel_imp}, whose solution is
\begin{equation}
\label{eqn:relations_impo}
A_N(s,t) = \frac{\sigma^2}{2\alpha_2 N} \left[1 - e^{-\alpha_2N(t-s)}\right] + \frac{\sigma^2}{2\alpha_2} e^{-\alpha_2 N(t-s)} + \sigma^2 s,\quad B_N(s,t) = A_N(s,t) - \frac{\sigma^2}{2\alpha_2}e^{-\alpha_2N(t-s)}.
\end{equation}
Now, denote $Y_j(t) := \xi_j^N(t) - \overline{\xi}_N(t) + U_N(t)$.
For any $t \geq 0$ we have
$$
\mathbb{E}[Y_j^2(t)] \!=\! \left(1 + \frac{1}{N}\right)\!\mathbb{E}[(\xi_j^N(t))^2] \!+\! \mathbb{E}[U_N^2(t)] \!-\! \frac{2}{N}\mathbb{E}[(\xi_j^N(t))^2], \quad \mathbb{E}[Y_i(t)Y_j(t)] \!=\! \mathbb{E}[Y_j^2(t)] - \mathbb{E}[(\xi_j^N(t))^2].
$$
For any $t > s$ we get
$$
\mathbb{E}[Y_j(s)Y_j(t)] = \left(1-\frac{2}{N}\right)\mathbb{E}[\xi_j^N(s)\xi_j^N(t)] + \frac{1}{N}\mathbb{E}[\xi_j^N(s)\xi_j^N(t)]  + \mathbb{E}[U_N(t)U_N(s)],
$$
and
$$
\mathbb{E}[Y_j(s)Y_i(t)] = \mathbb{E}[Y_j(s)Y_j(t)] - \mathbb{E}[\xi_j^N(s)\xi_j^N(t)]. 
$$
Note that for the stationary Ornstein-Uhlenbeck processes $\xi_j^N(t)$ we have, for any $t \geq 0$ and $t > s$ respectively
$$
\mathbb{E}[(\xi_j^N(t))^2] = \frac{\sigma^2}{2\alpha_2},\qquad \mathbb{E}[(\xi_j^N(t)\xi_j^N(s)]  = \frac{\sigma^2}{2\alpha_2}e^{-\alpha_2N(t-s)}.
$$
Moreover, by the independence between the $\xi_j^N(t)$'s, for any $t \geq 0$ and $t > s$ respectively,
$$
\mathbb{E}[(\overline{\xi}_N(t))^2] = \frac{1}{N}\mathbb{E}[(\xi_j^N(t))^2] = \frac{1}{N}\frac{\sigma^2}{2\alpha_2},\qquad \mathbb{E}[\overline{\xi}_N(t) \overline{\xi}_N(s)] = \frac{1}{N}\frac{\sigma^2}{2\alpha_2}e^{-\alpha_2N(t-s)}.
$$
For $U_N$ we get,
$$
\mathbb{E}[U_N^2(t)] = \sigma^2t + \frac{\sigma^2}{2\alpha_2 N},\qquad \mathbb{E}[U_N(t)U_N(s)] = \sigma^2 t + \frac{\sigma^2}{2\alpha_2 N}.
$$
One can extend the above computations to any $t,s \geq 0$: it suffices to take the minimum between $s$ and $t$ in the above formulae, and multiply by $\sign(t-s)$ in the exponentials. Denoting with $c_N(s,t)$ and $d_N(s,t)$ the covariance functions of $(x_j^N(t))_{t \in [0,T]}$ and $(Y_j(t))_{t \in [0,T]}$ (i.e.\! the process on the right hand side of \eqref{eqn:equality_in_law}), the above computations on $Y_j$ and the expressions \eqref{eqn:rel_imp} and \eqref{eqn:relations_impo} show that, for any $T > 0$, $c_N(s,t) = d_N(s,t)$, so that $(i)$ is proved.
For the proof of $(ii)$, recall that $\text{Law}\Big(x_j^N(t)\Big) = \mathcal{N}\left(0,\frac{\sigma^2}{2\alpha_2}(1+2\alpha_2 t)\right)$.
On the other hand, note that, integrating in $dX$, recalling \eqref{eqn:transition_kernel}, \eqref{eqn:mu_tilde_final} and $\mu_\infty(dx;X) = \mathcal{N}\left(X, \frac{\sigma^2}{2\alpha_2}\right)$,
\begin{align*}
\tilde{\mu}^t(dx) & = \int_{\mathbb{R}}Q_t(0,dX)\mu_{\infty}(dx; X) = \left[\int_{\mathbb{R}}\frac{1}{\sqrt{2\pi \sigma^2 t}} e^{-\frac{X^2}{2\sigma^2 t}}\frac{1}{\sqrt{\frac{\pi \sigma^2}{\alpha_2}}}e^{-\frac{(x-X)^2}{\sigma^2/\alpha_2}}dX\right]dx \\
& = \frac{1}{\sqrt{\frac{\pi \sigma^2}{\alpha_2}} \sqrt{1 + 2 \alpha_2 t}}e^{-\frac{\alpha_2 x^2}{\sigma^2 (1+2\alpha_2t)}}dx = \mathcal{N}\left(0,\frac{\sigma^2}{2\alpha_2}(1+2\alpha_2 t)\right)(dx),
\end{align*}
that is
\begin{equation}
\label{eqn:conv_for_a_fixed_j}
\text{Law}\Big(x_j^N(t)\Big) = \tilde{\mu}^t,
\end{equation}
for every $N \in \mathbb{N}$, with $\tilde{\mu}^t$ as in \eqref{eqn:mu_tilde_final}.
We now check the validity of $(ii)$ for bidimensional vectors $(x_i^N(t),x_j^N(t))$, as the assertion then follows by the Gaussianity of the processes in play. 
By the computations developed for the proof of $(i)$, we know that $(x_i^N(t), x_j^N(t))$ is normally distributed, with $\mathbb{E}[x_i^N(t)] = \mathbb{E}[x_j^N(t)] = 0$, $\text{Var}(x_i^N(t))= A(t) = \frac{\sigma^2 (1+2\alpha_2 t)}{2\alpha_2}$, and $\text{Cov}(x_i^N(t),x_j^N(t)) = B(t) = \sigma^2 t$. Then, we just need to check that $\tilde{\mu}^{t,2}(dx)$, as defined in $(ii)$, has the same moments. Let $(X_1,X_2) \sim \tilde{\mu}^{t,2}$. As one can check (e.g.\! via Mathematica): 
\begin{align*}
\mathbb{E}[X_1 X_2] = \int_{\mathbb{R}^3} x_1 x_2 \frac{1}{\sqrt{2\pi \sigma^2 t}} e^{-\frac{X^2}{2 \sigma^2 t}}\left(\frac{1}{\sqrt{\frac{\pi \sigma^2}{\alpha_2}}}\right)^2e^{-\frac{(x_1-X)^2}{\sigma^2/\alpha_2}}e^{-\frac{(x_2-X)^2}{\sigma^2/\alpha_2}}dXdx_1dx_2 = \sigma^2 t,\\ 
\end{align*}
while the other moments were already verified.

For the proof of $(iii)$, we note that for fixed $j \in \left\{1,\dots,N\right\}$ and any $T>0$ with $t \in [0,T]$, 
\begin{align*}
\text{Law}\Big(x_j^N(t) \big| |X^N(t)-X| \leq \varepsilon_N\Big) = \text{Law}\Bigg(\xi_j^N(t) - \overline{\xi}_N(t) + U_N(t) \Bigg| \Big|U_N(t) - X\Big| \leq \varepsilon_N\Bigg),
\end{align*}
since $X^N(t) \stackrel{\mathcal{D}}{=} U_N(t)$. By noting that $\mathbb{E}\Big[\xi_j^N(t) - \overline{\xi}_N(t) + U_N(t)\big|U_N(t)\Big] = U_N(t)$, and $\text{Var}\Big(\xi_j^N(t) - \overline{\xi}_N(t) + U_N(t)\big|U_N(t)\Big)= \left(1-\frac{1}{N}\right)\frac{\sigma^2}{2\alpha_2}$, we find that 
\begin{align*}
\lim_{N\to\infty}\text{Law}\Big(x_j^N(t) \big| |X^N(t) - X| \leq \varepsilon_N \Big) = \lim_{N\to\infty}\mathcal{N}\left(X,\left(1-\frac{1}{N}\right)\frac{\sigma^2}{2\alpha_2}\right) = \mu_\infty(\cdot;X).
\end{align*}
Furthermore, computing
\begin{align*}
\text{Cov}\Big(\xi_i^N(t) \!-\! \overline{\xi}_N(t) \!+\! U_N(t),\xi_j^N(t) \!-\! \overline{\xi}_j^N(t) \!+\! U_N(t) \Big| \overline{\xi}_N\Big) \!=\! -\frac{2}{N}\mathbb{E}[(\xi_i^N(t))^2] + \mathbb{E}[\overline{\xi}_N^2(t)]= -\frac{1}{N}\frac{\sigma^2}{2\alpha_2},
\end{align*}
which tends to $0$ when $N \to +\infty$, we can deduce the conditional law of bidimensional vectors $(x_i^N(t), x_j^N(t))$, so that $(iii)$ is verified.
\end{proof}

\begin{lem}[Averaging property]
\label{averaging}
Under the notation above, let $f : \mathbb{R}^3 \times [-1,1] \to [-1,1]$ be globally Lipschitz continuous in each variable. Let $L$ be the Lipschitz constant with respect to its fourth argument, i.e., for any $M, M' \in [-1,1]$, 
$$
|f(x_1,x_2,x_3,M) - f(x_1,x_2,x_3, M')| \leq L |M-M'|,
$$ 
for every $(x_1,x_2,x_3) \in \mathbb{R}^3$, and suppose $L <1$.
Let $\mu(du) = \mathcal{N}\left(0,\frac{\sigma^2}{2\alpha_2}\right)(du)$. Then, for any $T > 0$ we have that
\begin{itemize}
\item[(i)] For every $N \in \mathbb{N}$ and $t \in [0,T]$, the equation
\begin{equation}
\label{eqn:implicit_fixedpoint}
M_N(t) = \frac{1}{N}\sum_{j=1}^N f(\xi_j^N(t), \overline{\xi}_N(t), U_N(t), M_N(t))
\end{equation}
has a unique solution almost surely.
\item[(ii)] Let $\Big(B(t)\Big)_{t \geq 0}$ a Brownian motion. For every finite $k$-tuple of times $(t_1,\dots,t_k) \in [0,T]^k$,
\begin{equation}
\label{eqn:joint_conv}
\text{Law}\Big(M_N(t_1),\dots,M_N(t_k)\Big) \xrightarrow{N \to +\infty} \text{Law}\Big(M(t_1),\dots,M(t_k)\Big),
\end{equation}
where the process $\Big(M(t)\Big)_{t \geq 0}$ is defined by 
\begin{equation}
\label{eqn:limit_conv}
M(t) := \int_{\mathbb{R}}f(u,0, \sigma^2 B(t), M(t))\mu(du).
\end{equation}
\item[(iii)] For every fixed $t \in [0,T]$, 
\begin{equation}
\label{eqn:cond_conv_det}
\text{Law}\Big(M_N(t) \Big| |U_N(t) - z| \leq \varepsilon_N\Big) \xrightarrow{N \to +\infty} \delta_{M(t)},
\end{equation}
and $\varepsilon_N \xrightarrow{N \to +\infty} 0$, with 
\begin{equation}
\label{eqn:limit_conditional}
M(t):= \int_\mathbb{R}f(u,0,z,M(t))\mu(du).
\end{equation}
\end{itemize}
\end{lem}
\begin{proof}
The map $m \mapsto \frac{1}{N}\sum_{j=1}^Nf(\xi_j^N(t),\overline{\xi}_N(t),U_N(t),m)$
is $L$-Lipschitz continuous with $L <1$. Thus, $(i)$ follows by a contraction argument (e.g.\! Banach-Caccioppoli Theorem).

For the proof of $(ii)$ we make some preliminary remarks. First, note that by definition of $\overline{\xi}_N(t)$, we have
\begin{equation*}
\begin{cases}
d \overline{\xi}_N(t) = -\alpha_2 N\overline{\xi}_N(t)dt + \sigma dW^N(t),\\
\overline{\xi}_N(0) \sim \mathcal{N}\left(0,\frac{\sigma^2}{2\alpha_2 N}\right),
\end{cases}
\end{equation*}
with $W^N(t) := \frac{1}{\sqrt{N}}\sum_{j=1}^N W_j(t)$, with the $W_j$'s appearing in dynamics \eqref{eqn:ou_auxiliary}. The solution of the above equation is
$$
\overline{\xi}_N(t) = \overline{\xi}_N(0) e^{-\alpha_2 N t} + \sigma \int_0^t e^{-\alpha_2 N(t-s)} dW^N(s),
$$
which implies, for any $T > 0$,
\begin{equation}
\label{eqn:overline_aux}
\mathbb{E}\left[\sup_{t \in [0,T]}|\overline{\xi}_N(t)|\right] \xrightarrow{N \to +\infty} 0.
\end{equation}
Moreover, recalling Eq. \eqref{eqn:U_N} for $\Big(U_N(t)\Big)_{t \geq 0}$ and the definition of $\Big(M_N(t)\Big)_{t \geq 0}$ \eqref{eqn:implicit_fixedpoint}, we have the almost sure equality between the processes $\Big(M_N(t)\Big)_{t \geq 0}$ and $\Big(M_N^*(t)\Big)_{t \geq 0}$, the latter being defined by
$$
M^*_N(t) = \frac{1}{N}\sum_{j=1}^N f\left(\xi_j^N(t), \overline{\xi}_N(t), \sigma^2W(t) + U_N(0), M_N^*(t)\right).
$$
Let $\Big(\hat{M}_N(t)\Big)_{t \geq 0}$ be the process defined by
\begin{equation}
\hat{M}_N(t) = \frac{1}{N}\sum_{j=1}^N f\left(\xi_j^N(t), 0, \sigma^2W(t), \hat{M}_N(t)\right).
\end{equation}
In light of \eqref{eqn:overline_aux}, the trivial convergence $\mathbb{E}\left[\sup_{t \in [0,T]} |U_N(t) - \sigma^2 W(t)|\right] \xrightarrow{N \to +\infty} 0$ and the Lipschitz assumptions on $f$, we obtain
\begin{equation}
\label{eqn:lei_10}
\mathbb{E}\left[\sup_{t \in [0,T]}|M_N(t) - \hat{M}_N(t)|\right] \xrightarrow{N \to +\infty} 0.
\end{equation}
In particular $\Big(M_N(t)\Big)_{t \in [0,T]}$ and $\Big(\hat{M}_N(t)\Big)_{t \in [0,T]}$ share the same limit in distribution, provided it exists.

Now we fix a $t \in [0,T]$ and prove $(ii)$ for all the one-dimensional distributions. Let $\hat{M}_N(t)(z)$ be the unique solution to
$$
\hat{M}_N(t)(z) = \frac{1}{N}\sum_{j=1}^N f(\xi_j^N(t),0,z, \hat{M}_N(t)(z)),
$$
and $M(t)(z)$
$$
M(t)(z) = \int_\mathbb{R}f(u,0,z,M(t)(z))\mu(du).
$$
If we show that, for every $z \in \mathbb{R}$,
\begin{equation}
\label{eqn:lei_finale}
\hat{M}_N(t)(z) \xrightarrow{N \to +\infty} M(t)(z),
\end{equation}
almost surely, then we have $\hat{M}_N(t) = \hat{M}_N(t)(\sigma^2 W(t)) \xrightarrow{N \to +\infty} M(t)(\sigma^2 W(t)) = M(t)$ almost surely, and thus the one-dimensional version of $(ii)$ follows by \eqref{eqn:lei_10}. For the proof of \eqref{eqn:lei_finale}, we omit for the moment the arguments $0$ and $z$, and rewrite
$$
\hat{M}_N(t) = \frac{1}{N}\sum_{j=1}^N f(\xi_j^N(t),\hat{M}_N(t)) = \int_{\mathbb{R}}f(u, \hat{M}_N(t)) \mu_N(t)(du),
$$
where $\mu_N(t) := \frac{1}{N}\sum_{j=1}^N \delta_{\xi_j^N(t)}$ is the empirical measure of the $\xi_j^N(t)$'s. We now set $F:[-1,1]\times \mathcal{M}_1(\mathbb{R}) \to \mathbb{R}$ to be given by $F(m, \mu):= \int_{\mathbb{R}} f(u,m)\mu(du)$,
endowing $\mathcal{M}_1(\mathbb{R})$ with the BL (bounded-Lipschitz) metric
$$
||\mu - \nu||_{\text{BL}} = \sup\left\{\Bigg|\int_\mathbb{R}g d\mu - \int_\mathbb{R} g d\nu\ \Bigg| : ||g||_\infty \leq 1, \  g \ 1-\text{Lip.}\right\}.
$$
Note that $m \mapsto F(m,\mu)$ is $L$-Lipschitz, so that there exists a unique $m(\mu)$ such that $m(\mu) = F(m(\mu),\mu)$.
Moreover, we have
\begin{align*}
|m(\mu) - m(\nu)| & = \Bigg| \int_{\mathbb{R}}f(u,m(\mu)) \mu(du) - \int_\mathbb{R} f(u,m(\nu))\nu(du)\Bigg| \\
& \leq \int_\mathbb{R}\Big|f(u,m(\mu)) - f(u,m(\nu))\Big|\mu(du) + \Bigg| \int_\mathbb{R} f(u,m(\nu))\mu(du) - \int_\mathbb{R} f(u,m(\nu))\nu(du)\Bigg|\\
& \leq L |m(\mu) - m(\nu)| + ||\mu- \nu||_{\text{BL}},
\end{align*}
so that $|m(\mu) - m(\nu)| \leq \frac{||\mu-\nu||_{\text{BL}}}{1-L}$.
In particular, $m(\mu)$ is continuous in $\mu$. Finally, since $\hat{M}_N(t) = m(\mu_N(t))$ and by a LLN $\mu_N(t) \to \mu = \mathcal{N}\left(0,\frac{\sigma^2}{2\alpha_2}\right)$ almost surely, we have that, restoring the dependence on $z$ in the previous expression,
$$
\hat{M}_N(t)(z) \to m(t)(\mu) = \int_\mathbb{R} f(u,0,z,m(t)(\mu))\mu(du) = M(t)(z),
$$
so that \eqref{eqn:lei_finale} is proved. Recall that, for any $t$, the above implies
\begin{equation}
\label{eqn:needed}
\hat{M}_N(t)(\sigma^2 W(t)) \xrightarrow{N \to +\infty} M(t)(\sigma^2 W(t))
\end{equation}
almost surely. The same conclusion for the finite dimensional distributions follows by the continuity of the processes with respect to time.
Assertion $(iii)$ follows directly by $(ii)$. It is indeed the corresponding conditional statement of the one-dimensional version of $(ii)$ noting, as we did above, that $U_N(t) \stackrel{\mathcal{D}}{=} \sigma^2 W(t) + U_N(0)$ for every $N$, with $U_N(0) \to 0$ for $N\to +\infty$, and 
\begin{equation*}
\mathbb{E}[\xi_j^N(t)|U_N(t)] = \mathbb{E}[\xi_j^N(t)]= 0,\quad \text{Var}(\xi_j^N(t)|U_N(t)) = \text{Var}(\xi_j^N(t))= \frac{\sigma^2}{2\alpha_2}.
\end{equation*}
The limit distribution is a delta in $M(t)$ since $M(t) = \int_\mathbb{R}f(u,0,z,M(t))\mu(du)$ is deterministic.
\end{proof}
\begin{proof}[Proof of Theorem \ref{chaos_ordN^2}]
Consider the set of $N$ processes $(\tilde{m}_j^N(t), \tilde{M}^N(t))_{j=1,\dots,N}$, coupled with $(m_j^N(t),M^N(t))_{j=1,\dots,N}$, defined by
\begin{equation}
\label{eqn:coupled_tilda}
\begin{cases}
\tilde{m}^N_j(t) = \tanh\left(\beta_1(x_j^N(t) + \tilde{m}^N_j(t)) + \beta_2(X^N(t) + \tilde{M}^N(t))\right),\\
\tilde{m}^N_j(0) = m^N_j(0),\\
 \tilde{M}^N(t) =\frac{1}{N}\sum_{j=1}^N\tanh\left(\beta_1(x_j^N(t) + \tilde{m}^N_j(t)) + \beta_2(X^N(t) + \tilde{M}^N(t))\right).
\end{cases}
\end{equation}
By the contraction estimates \eqref{eqn:strong_collapse} for $k=1,m=2$, we know that both $m_j^N(t) - \tilde{m}_j^N(t)$ and $M^N(t) - \tilde{M}^N(t) \to 0$ in strong norm, for $N \to +\infty$. Indeed, \eqref{eqn:strong_collapse} can be trivially adapted to show that $M^N(t)$ collapses onto the empirical mean of the processes laying on the invariant curve. 
It is then sufficient to study the convergence in distribution of $\Big(\tilde{M}^N(t)\Big)_{t \in [0,T]}$. We first observe that, by $(i)$ of Lemma \ref{coordinates}, for every $N \in \mathbb{N}$, it holds $(\tilde{m}_j^N(t),\tilde{M}^N(t))_{t \in [0,T]} \stackrel{\mathcal{D}}{=} (\hat{m}_j^N(t),\hat{M}^N(t))_{t\in[0,T]}$,
with $(\hat{m}_j^N(t), \hat{M}^N(t))_{j=1,\dots,N}$ given by
\begin{equation}
\label{eqn:utile_dimN^2}
\begin{cases}
\hat{m}^N_j(t) = \tanh\left(\beta_1\big(\xi_j^N(t) - \overline{\xi}_N(t) + U_N(t) + \hat{m}^N_j(t)\big) + \beta_2\big(U_N(t) + \hat{M}^N(t)\big)\right),\\
\hat{m}^N_j(0) = m^N_j(0),\\
\hat{M}^N(t) = \! \frac{1}{N}\sum_{j=1}^N\tanh \! \Big(\beta_1\big(\xi_j^N(t) - \overline{\xi}_N(t) \!+ U_N(t) + \hat{m}^N_j(t)\big) + \beta_2\big(U_N(t) \!+\! \hat{M}^N(t)\big)\Big),
\end{cases}
\end{equation}
with $\xi_j^N(t)$ and $U_N(t)$ given by \eqref{eqn:ou_auxiliary} and \eqref{eqn:U_N} respectively.  Now, we note that the function
\begin{align*}
\varphi(\xi, \overline{\xi},U,M) := & \tanh\Big(\beta_1\big(\xi - \overline{\xi} + U + \varphi(\xi,\overline{\xi},U,M)\big)+ \beta_2\big(U+ M\big)\Big)
\end{align*}
satisfies the Lipschitz properties of Lemma \ref{averaging} for any choice of $\beta_1,\beta_2 > 0$ such that $\beta_1 + \beta_2 < 1$. Indeed, the Lipschitz continuity in the first three variables follows from the regularity of $\tanh(\cdot)$. For the last argument of $\varphi$, for any $M, M' \in [-1,1]$, we estimate
$$
\Big|\varphi(\xi,\overline{\xi},U,M) - \varphi(\xi,\overline{\xi},U,M')\Big| \leq \beta_1 \Big|\varphi(\xi,\overline{\xi},U,M) - \varphi(\xi,\overline{\xi},U,M')\Big| + \beta_2 |M-M'|,
$$
so that $\Big|\varphi(\xi,\overline{\xi},U,M) - \varphi(\xi,\overline{\xi},U,M')\Big|\leq \frac{\beta_2}{1-\beta_1}|M - M'|$.
Thus, $\varphi$ is $L$-Lipschitz continuous in $M$ with $L := \frac{\beta_2}{1-\beta_1} < 1$ if and only if $\beta_1 + \beta_2 < 1$.
We can then apply $(ii)$ of Lemma \ref{averaging} to $\hat{M}^N(t)$, to get, for all the finite dimensional distributions $(t_1,\dots,t_k) \in [0,T]^k$,
$$
\text{Law}\Big(\hat{M}^N(t_1), \dots, \hat{M}^N(t_k)\Big) \xrightarrow{N \to +\infty} \text{Law}\Big(M^*(t_1),\dots,M^*(t_k)\Big),
$$
where 
$$
M^*(t):= \int_\mathbb{R}\varphi(u,0,\sigma^2W(t),M^*(t))\mu(du),
$$
with $\mu = \mathcal{N}\left(0,\frac{\sigma^2}{2\alpha_2}\right)$. 
Assertion $(i)$ is then implied by
\begin{equation}
\label{eqn:last_thm}
\text{Law}\Big(\big(M^*(t)\big)_{t \in [0,T]}\Big) = \text{Law}\Big(\big(M(t)\big)_{t \in [0,T]}\Big),
\end{equation}
with $M(t)$ as in \eqref{eqn:full_law_M(t)}.
We start by proving \eqref{eqn:last_thm} for all the one-dimensional time distributions. For the purpose, we note that, for $t \in [0,T]$, $\sigma^2W(t)\stackrel{\mathcal{D}}{=} X(t) \sim \mathcal{N}(0,\sigma^2t)$,
with $X(t)$ the limit in distribution of $X^N(t)$. Substituting in $M^*(t)$, we have the equality in distribution
\begin{align*}
M^*(t) & = \int_{\mathbb{R}} \tanh\Big(\beta_1(u + X(t) + \varphi(u,0,X(t),M^*(t)))+ \beta_2(X(t) + M^*(t))\Big)\mu(du).
\end{align*}
Finally, with the change of variable $x := u + X(t)$, noting that, by the computations in Lemma \ref{coordinates}, the random variable $x(t) := \xi + \sigma^2W(t) \stackrel{\mathcal{D}}{=} \xi + X(t)$, with $\xi \sim \mathcal{N}\left(0,\frac{\sigma^2}{2\alpha_2}\right)$, is such that $\text{Law}(x(t))(dx) = \tilde{\mu}^t(dx)$,
the relation \eqref{eqn:last_thm} is proved for the one-dimensional time marginal distributions. The analogous conclusion is immediately obtained for all the finite dimensional distributions, by using properties $(i)$ and $(ii)$ of Lemma \ref{coordinates}, which hold for any $N$ and thus also for the limit. The equality in law for the whole process follows, since both $\big(M^*(t)\big)_{t \in [0,T]}$ and $\big(M(t)\big)_{t \in [0,T]}$ are functions of the same Gaussian process $\big(X(t)\big)_{t \in [0,T]}$. 

Assertion $(ii)$ is a consequence of property $(iii)$ of Lemma \ref{averaging} and of a change of coordinates as above. In details, we know that, with the above notation
$$
\text{Law}\Big(\tilde{M}^N(t) \Big| |X^N(t) - X| \leq \varepsilon_N\Big) = \text{Law}\Big(\hat{M}^N(t) \Big| |U_N(t) - X|\leq \varepsilon_N\Big),
$$ 
for some $\varepsilon_N \xrightarrow{N \to +\infty} 0$. By property $(iii)$ of Lemma \ref{averaging} and for the above couplings, this implies
$$
\text{Law}\Big(M^N(t) \Big| |X^N(t) - X| \leq \varepsilon_N\Big) \to \delta_{M(t)},
$$
where
\begin{align*}
M(t) = \int_\mathbb{R}\varphi(u,X,M(t))\mu(du). 
\end{align*}
With the change of coordinates $x = u + X$ we get \eqref{eqn:convergenza_a_delta}.

For the proof of $(iii)$, consider a single process $\hat{m}_j^N(t)$, as given in \eqref{eqn:utile_dimN^2}, for a fixed $t \in [0,T]$. Combining assertion $(iii)$ of Lemma \ref{coordinates} with $(ii)$ of this theorem, it follows directly
$$
\text{Law}\Big(\tilde{m}_j^N(t)\Big| |X^N(t) - X(t)| \leq \varepsilon_N \Big) \to \text{Law}\Big(\tilde{m}_j(t)\Big).
$$
The asymptotic independence among the magnetizations, i.e.\! $(iii)$, follows by noting that the $m_j^N$'s (resp.\! $\tilde{m}_j^N$) are functions of $x_j^N$, $X^N$, and $M^N$ (resp.\! $\tilde{M}^N$). When we condition with respect to $X^N(t)$, we have that the $x_j^N(t)$'s are asymptotically independent by $(iii)$ of Lemma \ref{coordinates}, and $M^N(t)$ tends to a deterministic value by $(ii)$. Thus the dependence among the magnetizations is deleted in the limit.
\end{proof}

\begin{rem}[Long-time behavior]
In this timescale the long-time behavior cannot be determined, as the process $(X(t))_{t \geq 0}$ does not admit an invariant measure on the whole space. However, it is clear that for big positive values of $X(t)$ the second-level magnetization $M(t)$ will be close to $+1$, while for big negative values it will be close to $-1$, as it is shown in Fig.\! \ref{M(t)_betasmall}.
\end{rem}

\begin{figure}%
    \centering
    \subfloat[$(t,M^N(t))$]{{\includegraphics[width=5.5cm]{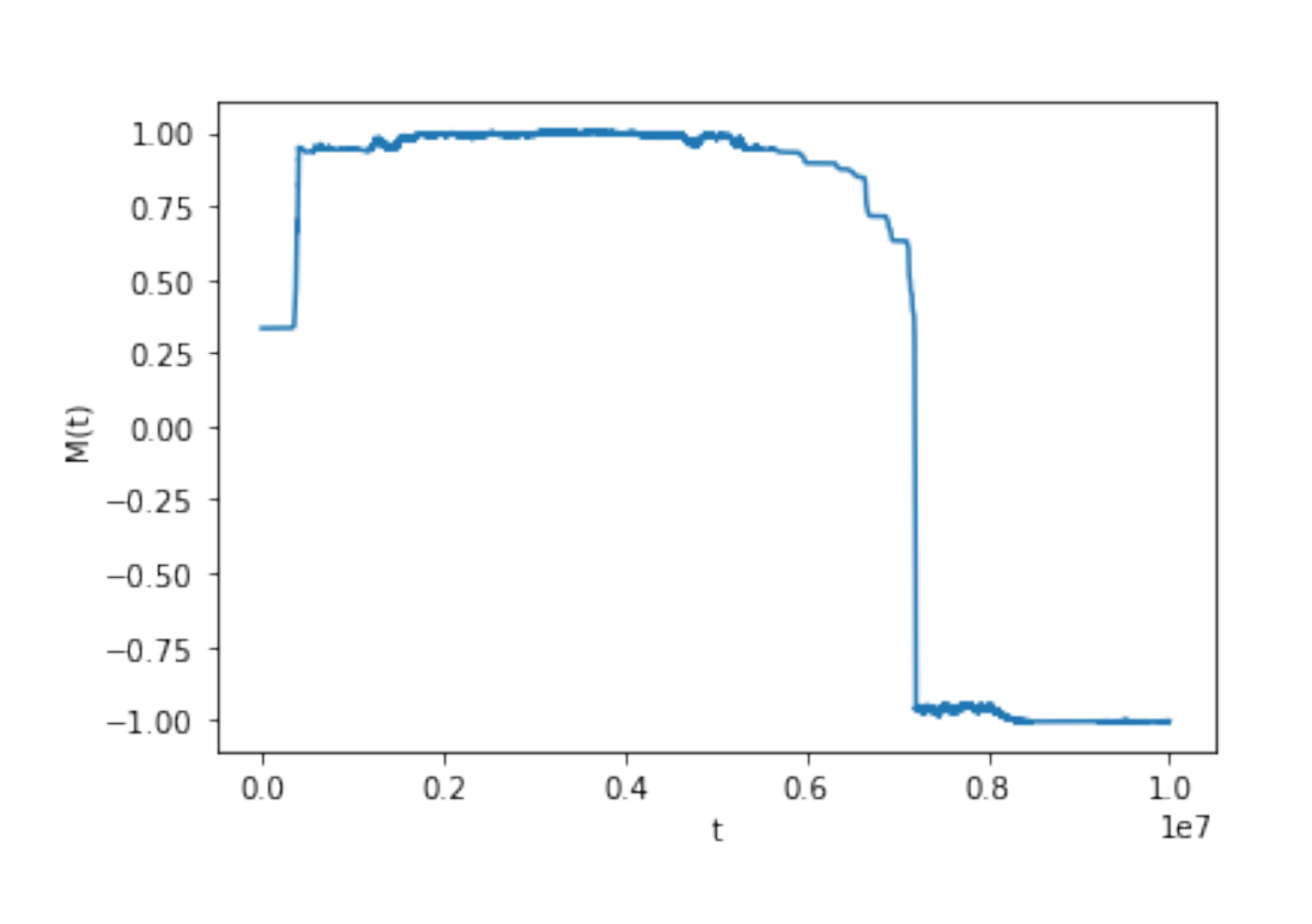} }}%
    \qquad
    \subfloat[$(t,X^N(t))$]{{\includegraphics[width=5.5cm]{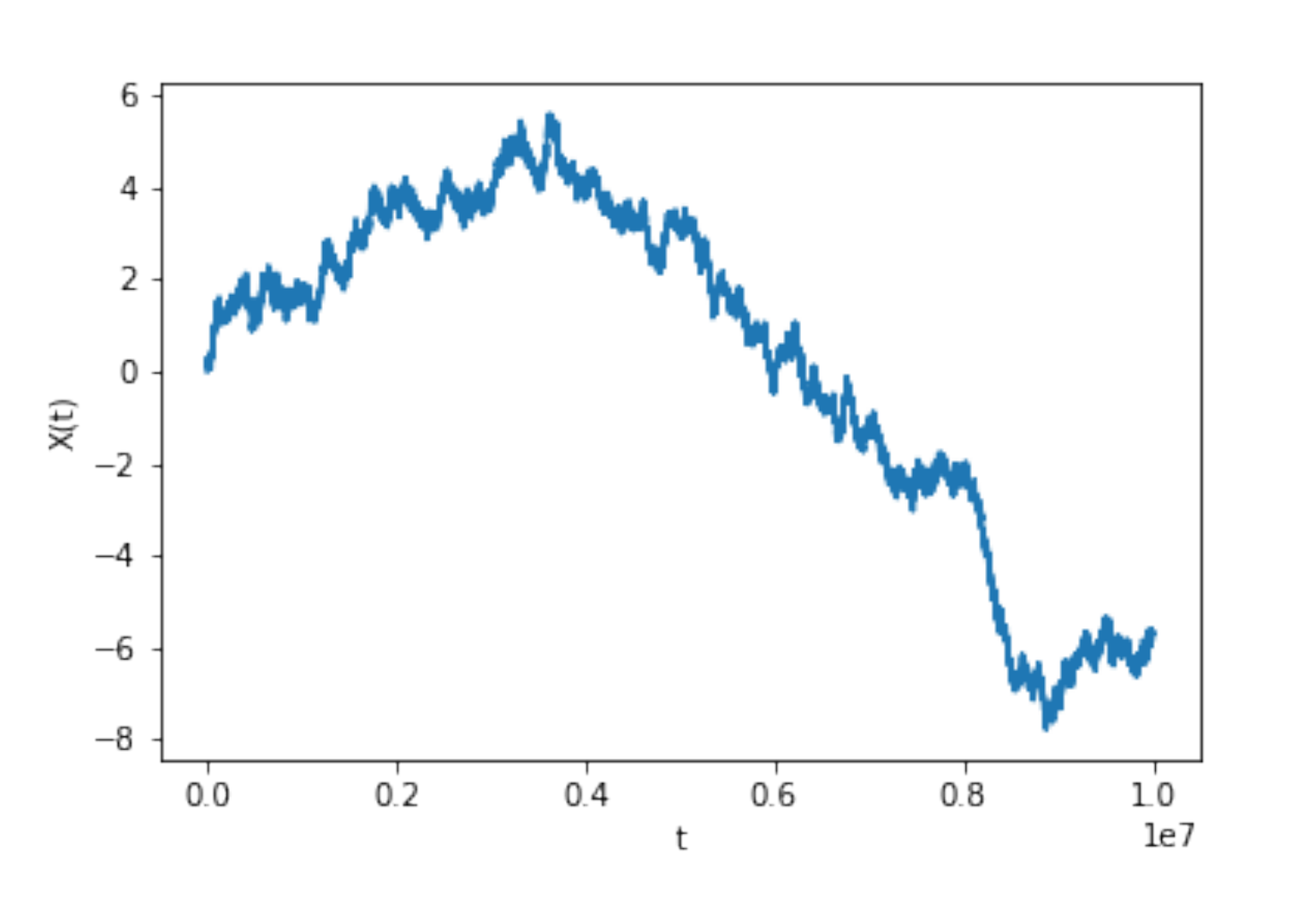} }}%
    \caption{Simulation of the finite particle system subcritical dynamics at a timescale of order $N^2$, for $N = 200$, $\beta_1= \beta_2=0.3$, $\alpha_1 = \alpha_2 = \sigma = 1$, $T = 10^7$.}   
    \label{M(t)_betasmall}
\end{figure}


\subsection{Renormalization theory: the subcritical case}
\label{renormalization_theory}
The results of the previous section are expected to be generalizable to the $k$-th hierarchical level for any $k > 0$ finite. In this section we state what we think should be the corresponding assertion, in form of a conjecture, as we did not work out a proof yet. The goal is to define inductively a \textit{renormalization} map $\varphi_d$, with $d = 1,\dots,k$, which allows one to describe the limit dynamics for the aggregated magnetizations at each timescale $N^dt$ in terms of the corresponding aggregated diffusions.

In this case, the model is defined on the set $V:= \left\{1,\dots,N\right\}^k$. Any of the $N^k$ individuals in the population is identified by a $k$-tuple $i = (i_1,i_2,\dots,i_k)$. For any two individuals $i,j \in V$, define the \textit{hierarchical distance} as
\begin{equation}
\label{eqn:hierarchical_distance}
d(i,j) := \min\left\{ d \ \big| \ 0 \leq d \leq k-1, \  \ (i_{d+1},\dots, i_k) = (j_{d+1}, \dots, j_k)\right\}.
\end{equation}
If in \eqref{eqn:hierarchical_distance} $(i_{d+1},\dots,i_k) \neq (j_{d+1},\dots,j_k)$ for any $0 \leq d \leq k-1$, then we set $d(i,j) := k$.
The interaction among individuals $(i,j)$ at distance $d(i,j) = d$ now scales as
\begin{equation}
\label{eqn:hier_inter}
J_{ij} = \frac{\beta_d}{N^d}, \qquad J'_{ij} = \frac{\alpha_d}{N^{2d -1}}.
\end{equation} 
For $d = 1,2,\dots,k-1$ and $i \in V$, set 
$$
i^d := (i_{d+1},\dots,i_{k}) \in \left\{1,2,\dots,N\right\}^{k-d}, \qquad i^k := \emptyset.
$$
Denote the $N^{k-d}$ \textit{$d$-th level magnetizations}, for any $d < k$, and the \textit{$k$-th level magnetization} as
$$
m^{N}_{i^d}(t) := \frac{1}{N^d}\sum_{j \in V : j^d = i^d}\!\!\!\!\!\!\! \mu_j(t), \quad m^{N}_{i^k}(t) := \frac{1}{N^k}\sum_{j \in V} \mu_j(t). 
$$
Moreover, denote for any $d < k$ the limit \textit{$d$-th level diffusion}, 
\begin{equation}
\label{eqn:d_th_diff}
\begin{cases}
d X^{(d)}(t) = - \alpha_{d+1} X^{(d)}(t) dt + \sigma d W^{(d)}(t),\\
X^{(d)}(0) = 0,
\end{cases}
\end{equation}
and the limit \textit{$k$-th level diffusion}
\begin{equation}
\label{eqn:k_th_diff}
\begin{cases}
d X^{(k)}(t) = \sigma d W(t),\\
X^{(k)}(0) = 0.
\end{cases}
\end{equation}
Let $Q_t^{(d)}(0,dX)$ and $Q_t^{(k)}(0,dX)$ be the transition kernels of the diffusions \eqref{eqn:d_th_diff} and \eqref{eqn:k_th_diff} respectively, and $\nu_y^{(d)}$ the stationary distribution of
\begin{equation}
\begin{cases}
dz(t) = -\alpha_{d+1}(z(t) - y)dt + \sigma dB(t),\\
z(0) = 0,
\end{cases}
\end{equation}
where $B(t)$ is a Brownian motion. Note that, in the notation of the previous section $\nu_y^{(1)}(\cdot) = \mu_{\infty}(\cdot;y)$.
Then, we have
\begin{conj}
\label{conj_2}
Assume $\beta_1 + \dots + \beta_k < 1$ and that $x_j(0) \sim \mathcal{N}(0,1)$ i.i.d.\! for any $j \in V$. Then, for any $i^d$, $d \in \left\{1,2,\dots,k\right\}$ and $T > 0$,
\begin{equation}
\label{eqn:general_subcritical}
(m_{i^d}^N(N^d t))_{t \in [0,T]} \to (m^{(d)}(t))_{t \in [0,T]}
\end{equation}
in the sense of weak convergence of stochastic processes, where
\begin{equation}
\label{eqn:renorm_eq}
m^{(d)}(t) = \varphi_d(X^{(d)}(t),0),
\end{equation}
with $X^{(d)}$ the solution to \eqref{eqn:d_th_diff} (resp.\! \eqref{eqn:k_th_diff}) for $d < k$ (resp.\! $d=k$). The function $\varphi_d = \varphi_d(x,y)$, for $d=2,\dots,k$, is the unique solution to
\begin{equation}
\label{eqn:recursive_hier}
\varphi_d = \int_{\mathbb{R}}\varphi_{d-1}(z,\beta_d(\varphi_d + x) + y)\tilde{\mu}^{t,d}(dz),
\end{equation}
with, for $d = 2,\dots,k$,
\begin{equation}
\tilde{\mu}^{t,d}:= \int_\mathbb{R} Q_t^{(d)}(0,dx)\nu^{(d-1)}_{x},
\end{equation}
and $\varphi_1(x,y)$ is the unique solution to $\varphi_1 = \tanh(\beta_1(\varphi_1 + x) + y)$.
\end{conj}
\noindent
We conclude the section justifying the $k$-th level subcritical regime condition 
\begin{equation}
\label{eqn:subcritical_condition}
\beta_1 + \dots + \beta_k < 1
\end{equation}
in the above conjecture, which is in accordance with the first two hierarchical levels of the previous sections. Let $L_{d-1}$ be the Lipschitz constant of $\varphi_{d-1}$ in its second variable, with $L_{d-1} < 1$. At the $d$-th hierarchical level, Eq. \eqref{eqn:recursive_hier} has a unique solution, provided that the right hand side is a contraction in terms of $\varphi_d$. This is true if
\begin{equation}
\label{eqn:rec_hier_d-1}
L_{d-1}\beta_d < 1.
\end{equation}
On the other hand, computing the $d$-th level Lipschitz constant $L_d$ we find, for $y, y' \in \mathbb{R}$,
\begin{align*}
|\varphi_d(x,y) - \varphi_d(x,y')| & \leq L_{d-1}(1+\beta_d L_d)|y-y'|,
\end{align*}
and thus 
\begin{equation}
\label{eqn:rec_hier_d}
L_d = \frac{L_{d-1}}{1-L_{d - 1}\beta_d}.
\end{equation}
Using \eqref{eqn:rec_hier_d}, the subcriticality condition \eqref{eqn:subcritical_condition} implies inductively the validity of \eqref{eqn:rec_hier_d-1} for $d=1,\dots, k$, starting from $L_0 = 1$.

\subsection{The limit case: $\left[\beta_1=\beta_2 \to \infty\right]$}
\label{the_limit_case}
In this section we develop heuristics for dealing with the limit case of null temperatures. 
For convenience, here we use the notation
\begin{equation}
\label{eqn:normal_distrs}
\begin{aligned}
\mu_0(dx) &:= \mathcal{N}(0,\rho^2),\qquad \mu_0^t(dx) := \mathcal{N}(0,\rho^2(t)),\\
\mu_X(dx) &:= \mathcal{N}(X,\rho^2),\qquad \mu^\infty_X(dx):= \mu_\infty(dx;X) = \mathcal{N}\left(X,\frac{\sigma^2}{2\alpha_2}\right),
\end{aligned}
\end{equation}
for the normal distributions we consider, where $\rho^2$ and $\rho^2(t)$ depend on the diffusion parameters $\sigma$ and $\alpha_2$. We analyze the limit dynamics at any timescale, formally replacing $\beta_1=\beta_2=\infty$ in the equations. Substituting $\tanh(\beta_1 z + \beta_2 w)$ with $\sign(z + w)$, the main focus of the section is on the study of the deterministic dynamics
\begin{equation}
\label{eqn:limit_intro}
\begin{cases}
\dot{m}(t)(x) = 2\text{sign}(x + m(t)(x) + X + M(t)) - 2m(t)(x),\\
m(0)(x) = m_0(x),\\
M(t) = \int_{\mathbb{R}}m(t)(x)\nu(dx),
\end{cases}
\end{equation}
and of its equilibria, where $X$ has to be intended as a fixed value of the second level diffusion, and the measure $\nu(dx)$ is a normal distribution. Equation \eqref{eqn:limit_intro} describes the zero-temperature dynamics at a timescale of order $1$, possibly with initial data at larger timescales, realized by choosing  an initial value of the second level diffusion $X \neq 0$, which is kept fixed for the whole evolution. The spirit of our approach is the following:
\begin{itemize}
\item we identify all the equilibria reached at an order $1$ timescale, showing that they are provided by \textit{staircase} functions
\begin{equation*}
m^{x_0}(x):=
\begin{cases}
+1, \ \ \ \forall x > x_0,\\
-1, \ \ \ \forall x < x_0,
\end{cases}
\end{equation*}
where the discontinuity point $x_0\in \mathbb{R}$ belongs to a certain interval which we refer to as \textit{fixed points region}. In particular, we show that the fixed points region depends on the current value of the macroscopic quantity $X$ and on the diffusion parameters $\sigma$ and $\alpha_2$ (Propositions \ref{scalino}, \ref{scalinoX});
\item we deduce the local stability of the above configurations (Proposition \ref{stability_scalino});
\item we update the macroscopic time (either with an infinitesimal change at order $N$ or $N^2$), and evolve $X$ to a new value $\overline{X}$;
\item we quantify the corresponding adaptation of the staircase profiles to the change in the environment, distinguishing between the order $N$ dynamics (Section \ref{ordN_subsub}), where the first level diffusions $x_j$'s already evolve non-trivially, and the order $N^2$ dynamics (Section \ref{ordN^2_subsub}), where the first level diffusions have reached a stationary equilibrium with the environment;
\item in case the previous step brings the magnetization profile to a \textit{non-equilibrium} configuration, we quantify the way it approaches again the fixed points region (Propositions \ref{attractor}, \ref{attractorX});
\item the reiteration of the above steps allows for a heuristic description of the order $N^2$ dynamics;
\item we show simulations of the finite particle system at any timescale, confirming the above facts and highlighting the remaining open problems.
\end{itemize}
In particular, we observe the following phenomenon: the $N^2$ dynamics undergoes a \textit{phase transition}, depending on the diffusion parameters $\sigma$ and $\alpha_2$. Specifically, for big values of $\frac{\sigma^2}{2\alpha_2}$, the pair $(X^N(t),M^N(t))$ approximately evolves as a regular two-dimensional diffusion inside the fixed points region. When $\frac{\sigma^2}{2\alpha_2}$ is small instead, $(X^N(t),M^N(t))$ still evolves inside the fixed points region, but behaves as a two-dimensional diffusion with jumps (see Fig.\! \ref{complete_N^2}). The motivation for this is a loss of stability, not yet fully understood, of certain areas of the fixed points region which happens already at an order $N$ timescale.

\subsubsection{Order $1$ dynamics}
When $\beta_1 = \beta_2 \to \infty$, we have that the dynamics at order $1$ is given by 
\begin{equation}
\label{eqn:order1_inf}
\begin{cases}
\dot{m} = 2\text{sign}(2m) - 2m,\\
m(0) = 2p-1,
\end{cases}
\end{equation}
which, when $t \to \infty$ reaches the equilibrium point $m = \text{sign}(2m) = \text{sign}(m)$.
Clearly, Eq. \eqref{eqn:order1_inf} has the same behavior as the low temperature limit of the Curie--Weiss model. It is easy to see that the solution $m = 0$ is unstable and the two polarized solutions $m = \pm 1$ are stable, where the one getting picked asymptotically is determined by the initial sign of $m(0)$. We denote the three equilibria of \eqref{eqn:order1_inf} by $m_0, m_{\pm}$.
We now let the dynamics evolve until a time of order $N$, when the dynamics of the diffusions is not trivial anymore. Let this time be our new initial time, and let the system evolve again at times of order $1$. The mean field equations are now 
\begin{equation}
\label{eqn:limit_inf}
\begin{cases}
\dot{m}(t)(x) = 2\text{sign}(x + m(t)(x) + M(t)) - 2m(t)(x),\\
m(0)(x) = m_0(x),\\
M(t) = \int_{\mathbb{R}}m(t)(x)\mu_0(dx),
\end{cases}
\end{equation}
with $\mu_0$ as in the first line of \eqref{eqn:normal_distrs}, for some $\rho > 0$ which depends on the previous evolution of the diffusions and $m_0(x)$ is close (but not necessarily equal) to the function constantly equal to one of the three equilibria $m_0, m_{\pm}$. The study of the asymptotic profile of Eq. \eqref{eqn:limit_inf} helps us in understanding what the dynamics at longer timescales will be. Indeed, as we already stressed, the further timescales dynamics are expected to be described by motions across the different equilibria profiles of order $1$, triggered by the dynamics of the diffusions.

For studying the asymptotic profiles of the magnetization $m(t)(x)$ we make an \textit{ansatz} on their shape, motivated by the following preliminary remark
\begin{rem}
Any asymptotic equilibrium $m^*(x)$ of Eq. \eqref{eqn:limit_inf} is such that
\begin{equation*}
m^*(x) :=
\begin{cases}
+1, \ \ \ \forall x > 2,\\
-1, \ \ \ \forall x < -2.
\end{cases}
\end{equation*}
Indeed, for any $t> 0$, we have $ -2 \leq m(t)(x) + M(t) \leq 2$, and thus for $x > 2$, $\dot{m}(t)(x) > 0$ and, symmetrically, for $x < -2$, $\dot{m}(t)(x) < 0$. 
\end{rem}

Even though a full proof of the validity of the below ansatz is not established, as we expect it to hold we state it as a
\begin{prop}[Shape of the equilibria] 
\label{scalino}
Every equilibrium of Eq. \eqref{eqn:limit_inf} is a staircase function $m^{x_0}(x)$ of the form
\begin{equation}
\label{eqn:scalino}
m^{x_0}(x):=
\begin{cases}
+1, \ \ \ \forall x > x_0,\\
-1, \ \ \ \forall x < x_0,
\end{cases}
\end{equation}
for some $x_0 \in \mathbb{R}$ satisfying 
\begin{equation}
\label{eqn:posneg}
- 2\mu_0(x_0, +\infty) \leq x_0 \leq 2 \mu_0(-\infty, x_0).
\end{equation}
\end{prop}
\begin{proof}
We restrict ourselves to prove one direction of the ansatz, which is easy. Indeed, a profile $m^{x_0}(x)$ is an equilibrium for the dynamics \eqref{eqn:limit_inf} if
\begin{equation}
\label{eqn:staz}
m^{x_0}(x) = \text{sign}(x + m^{x_0}(x) + M),
\end{equation}
with 
\begin{equation}
\label{eqn:stazM}
M = \int_\mathbb{R} m^{x_0}(x) \mu_0(dx) = - \mu_0(-\infty,x_0) + \mu_0(x_0, \infty) = 1 - 2 \mu_0(-\infty,x_0).
\end{equation}
For \eqref{eqn:staz} to be satisfied it must be, when $x < x_0$, $x - 1 + M < 0$, while, for $x > x_0$, $x + 1 + M > 0$.
Using \eqref{eqn:stazM}, the inequalities become
$$
x -1 + 1- 2\mu_0(-\infty, x_0) = x - 2\mu_0(-\infty,x_0) < 0,
$$ 
for $x < x_0$, and
$$
x + 1 + 1 - 2\mu_0(-\infty,x_0) = x + 2\mu_0(x_0,+\infty) > 0,
$$
where in the second equality we have used $1 - 2\mu_0(-\infty,x_0) = -1 + 2\mu_0(x_0,+\infty)$. Because of the monotonicity of the above conditions with respect to $x$, they can be equivalently stated respectively as 
$$
x_0 - 2 \mu_0(-\infty, x_0) \leq 0,\qquad x_0 + 2\mu_0(x_0,+\infty) \geq 0.
$$
Finally, observe that, when $x_0 \geq 0$, the second inequality is trivially true and thus the inequality on the right in \eqref{eqn:posneg} is the equilibrium condition for this case, while for $x_0 \leq 0$ the first is the trivial one, so that we obtain the left inequality in \eqref{eqn:posneg} as a necessary condition for the equilibrium.
\end{proof}

\begin{rem}
\label{lim_case}
Note that in \eqref{eqn:posneg}, both at a timescale of order $1$ and $N$, $\mu_0$ is a normal distribution centered in $0$ with variance $\rho^2$ (possibly depending on the (macroscopic) time and on the diffusion parameters $\sigma$ and $\alpha_2$). Condition $\eqref{eqn:posneg}$ restricts to 
$$
-2 \leq x_0 \leq 2, \quad -1 \leq x_0 \leq 1,
$$
respectively when $\rho \to 0$ and $\rho \to \infty$. Moreover, the fixed points interval is monotonically decreasing with $\rho$, since $\mu_0(-\infty,x_0)$ is so.
\end{rem}

As an example of convergence to the equilibrium, let us fix a constant initial datum $m(0)(x) \equiv \overline{m}$ (with $0 < \overline{m} < \frac{1}{2}$) for \eqref{eqn:limit_inf} and reason heuristically by small variations of time. Since at the initial time $M(0) = \overline{m}$, for every $x > -2\overline{m}$ one has that $\frac{d}{dt}m(t)(x)\Big|_{t=0} > 0$, and symmetrically, for every $x < -2\overline{m}$, we have $\frac{d}{dt}m(t)(x)\Big|_{t=0} < 0$. One can expect that these considerations should keep being true for any $t > 0$ as the quantities inside the sign function increase/decrease monotonically with time (this is not precise because of the term $M(t)$ inside the sign). The same argument works for $-\frac{1}{2} < \overline{m} < 0$, and for the symmetric case $\overline{m} = 0$. The limit configuration, denoted by $m^*(x)$, is thus given by 
\begin{equation}
\label{eqn:limit_profile}
m^*(x):=
\begin{cases}
-1, \ \ \text{for } x < -2\overline{m},\\
0, \ \ \text{for } x = -2\overline{m},\\
+1, \ \ \text{for } x > -2\overline{m}.  
\end{cases}
\end{equation}
By integrating \eqref{eqn:limit_profile} over the diffusion's distribution we obtain the asymptotic value of $M$,
\begin{equation}
\label{eqn:limit_prof_M}
M = \int_{-2\overline{m}}^{2\overline{m}}\mu_0(dx).
\end{equation}
Depending on the variance parameter of the distribution $\mu_0$, the resulting asymptotic value of $M$ can either be greater or smaller than the initial one (or equal to in the symmetric case $\overline{m} = 0$). The bigger the variance of $\mu_0$, the more $M$ would tend to be depolarized in this limit.


Proposition \ref{scalino} asserts that there exists a whole region of fixed points for Eq. \eqref{eqn:limit_inf}. Concerning the stability properties of these equilibria, we have that 
\begin{prop}[Stability of the equilibria]
\label{stability_scalino}
The equilibrium $m^{x_0}(x)$ is locally stable for the dynamics \eqref{eqn:limit_inf} if inequality \eqref{eqn:posneg} holds. 
\end{prop}
\begin{proof}
The proof is non-rigorous. Fix e.g.\! $x_0 > 0$. Choose as initial condition for \eqref{eqn:limit_inf} $m_0(x) = \tilde{m}(x)$, the perturbation of $m^{x_0}(x)$ in a point $\tilde{x} > x_0$, given by
\begin{equation*}
\tilde{m}(x):=
\begin{cases}
\tilde{m}(x) = m^{x_0}(x), \ \ \forall x \neq \tilde{x}\\
\tilde{m}(\tilde{x}) = m^{x_0}(\tilde{x}) - \varepsilon.
\end{cases}
\end{equation*}
Then we have that, heuristically, $\frac{d}{dt}m(t)(\tilde{x})\Big|_{t=0} = \varepsilon > 0 $. Analogously, if $\tilde{x} < x_0$ we consider $\tilde{m}(x)$, defined as
\begin{equation*}
\tilde{m}(x):=
\begin{cases}
\tilde{m}(x) = m^{x_0}(x), \ \ \forall x \neq \tilde{x}\\
\tilde{m}(\tilde{x}) = m^{x_0}(\tilde{x}) + \varepsilon, 
\end{cases}
\end{equation*}
so that  $\frac{d}{dt}m(t)(\tilde{x})\Big|_{t=0} = -\varepsilon < 0$.
\end{proof}

To sum up, as we saw above, when we start the dynamics with a constant initial datum $m_0(x) \equiv \overline{m}$ we soon get attracted (at times of order $1$) to a staircase equilibrium $m^{x_0}(x)$ for some $x_0 \in \mathbb{R}$.
The next proposition, for which we do not have a proof (but is motivated by Proposition \ref{stability_scalino} and supported by numerics), describes what happens at a timescale of order $1$ when we start the dynamics \eqref{eqn:limit_inf} with a staircase initial datum $m_0(x) = m^{x_0}(x)$ with $x_0$ not belonging to the fixed points region given by \eqref{eqn:posneg}.

\begin{prop}[Stable attractors of the dynamics]
\label{attractor}
Let $m(t)(x)$ be the solution to Eq. \eqref{eqn:limit_inf} with initial datum $m_0(x) = m^{x_0}(x)$, with $x_0 > 0$ (resp.\! $x_0 < 0$) such that $x_0 > 2 \mu_0(-\infty, x_0)$ (resp.\! $x_0 < - 2\mu_0(x_0, + \infty)$). Then, we have
$$
\lim_{t \to \infty} m(t)(x) = m^{\overline{x_0}}(x),
$$
with $\overline{x_0} = 2\mu_0(-\infty,\overline{x_0})$ (resp.\! $\overline{x_0} = - 2\mu_0(\overline{x_0},+\infty)$).
\end{prop}
Proposition \ref{attractor} turns out to be very useful in describing the dynamics at order $N$ and $N^2$ by infinitesimal (of order $1$) variations of time. Indeed, the presence of a non-zero $X(t)$ can move the magnetization profile to be outside of the fixed points region. The above proposition thus quantifies how the dynamics gets attracted again towards the fixed points region, at least for times of order $1$. Unfortunately we were not able to prove this result, which can be motivated heuristically by saying that the out-of-equilibrium dynamics approaches the nearest possible stable equilibrium. An illustration of this phenomenon is given in Fig.\! \ref{scalino_ord1}. 
\begin{figure}%
    \centering
    \subfloat[Initial configuration at order $1$]{{\includegraphics[width=5.5cm]{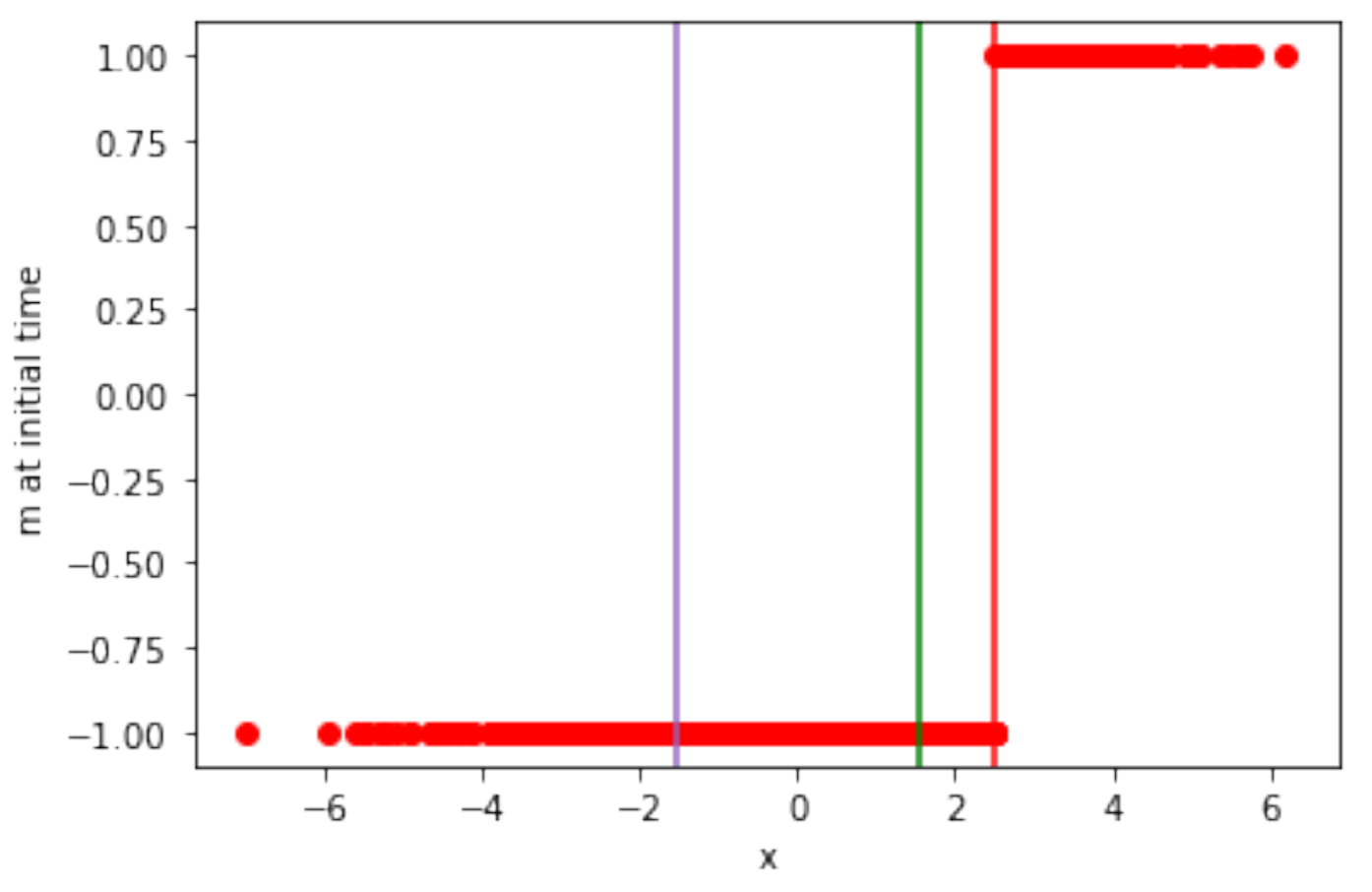}}}%
    \qquad
    \subfloat[Final configuration at order $1$]{{\includegraphics[width=5.5cm]{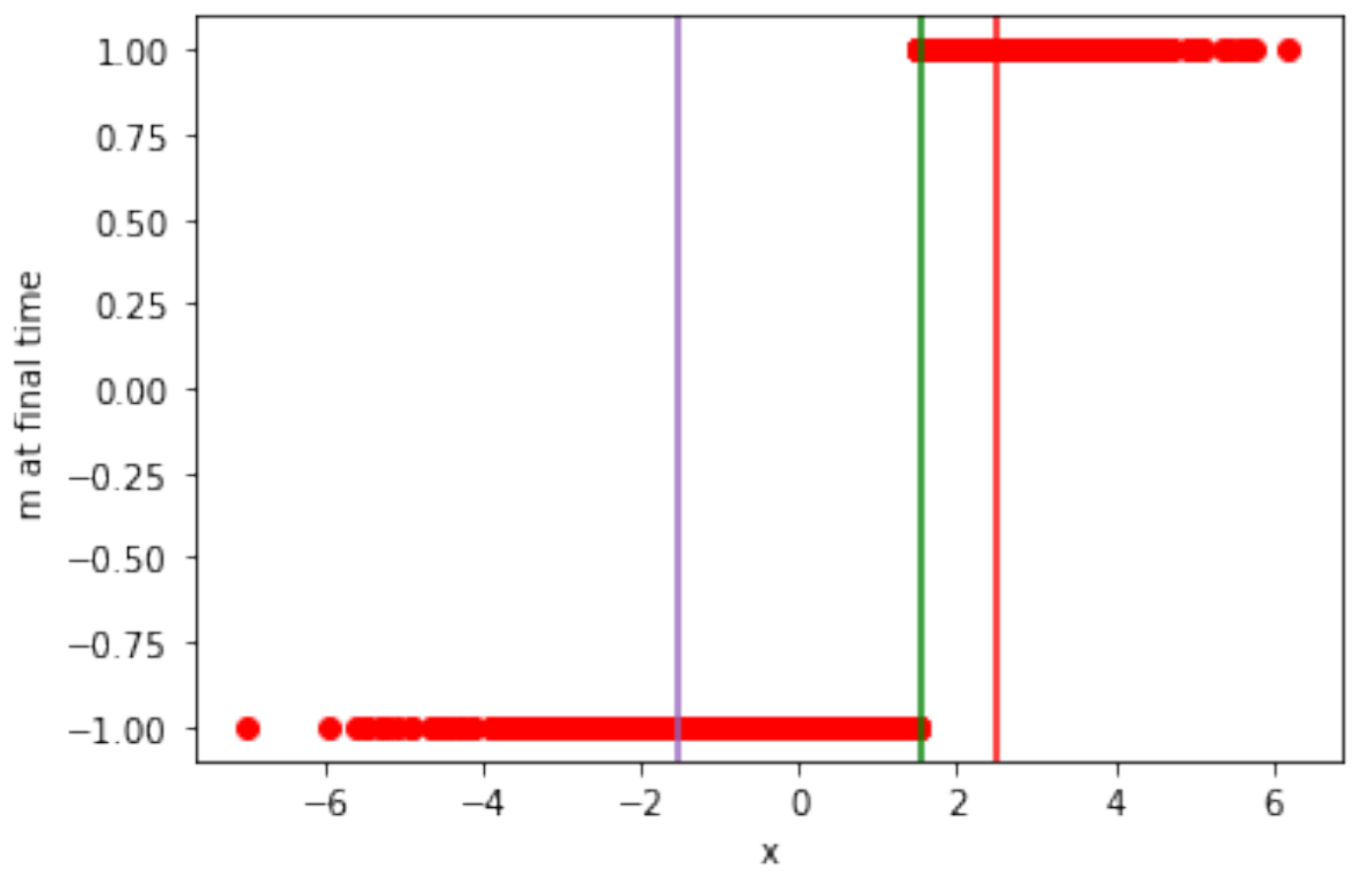} }}%
    \caption{Simulation of the finite particle system's dynamics at a timescale of order $1$, for $N = 1000$, $\beta_1= \beta_2=\infty$, $\alpha_1 = \alpha_2 = 1$, $\sigma = 3$. We start the dynamics with a staircase function (the red line) outside the fixed points region's band (purple and green lines). We take $x_j(0) \sim \mathcal{N}\left(0,\frac{\sigma^2}{2\alpha_2}\right)$.}   
    \label{scalino_ord1}
\end{figure}
\subsubsection{Order $N$ dynamics}
\label{ordN_subsub}
In the timescale of order $N$ the only additional dynamics which takes place is due to the fact that $\mu_0^t(dx)$ now depends on time (it is a normal distribution centered around $0$, with variance depending on the macroscopic timescale $t$ and proportional to $\frac{\sigma^2}{2\alpha_2}$), because of the dynamics of the Ornstein-Uhlenbeck diffusions. In this scale we thus expect to see the same staircase equilibrium previously reached, with some movement of the points close to $x_0$ (caused by the motion of the diffusions at order $N$) in between the region of fixed points described in Proposition \ref{scalino}. At the finite particle system level indeed, the motion of the diffusions at order $N$ should produce a coexistence of phases around $x_0$, with some magnetizations being $+1$ and others $-1$. An illustration of this is shown in Fig.\! \ref{scalino_ordN}, the analogous to Fig.\! \ref{scalino_ord1} at order $N$. The bigger the diffusive coefficient $\sigma$ (for a fixed $\alpha_2$), the wider the range of the diffusions and the area with coexistence of phases are: in Fig.\! \ref{scalino_ordN} the coexistence area fills all the fixed points region.
The deterministic limit dynamics becomes
\begin{equation}
\label{eqn:limit_inf_ordN}
\begin{cases}
m(t)(x) = \text{sign}(x + m(t)(x) + M(t)),\\
m(0)(x) = m^{x_0}(x),\\
M(t) = \int_{\mathbb{R}}m(t)(x)\mu_0^t(dx).
\end{cases}
\end{equation}
At this timescale the diffusion parameters play an important role. 
For $\frac{\sigma^2}{2\alpha_2}$ big, simulations suggest the presence of a very mild interaction among the magnetizations: see Fig.\! \ref{magn_ordN_bigsigma}, where we plot the path of the second level empirical magnetization relative to the same simulation of Fig.\! \ref{scalino_ordN}. We see that, after starting from a rather polarized value, after a short time $M^N(t)$ becomes very small and from that time on it just wanders around $0$, so that the single magnetization's processes are subject to a very low interaction among themselves, which could eventually tend to zero for $N \to +\infty$; in fact, the interaction among the diffusions is also tending to $0$. In other words, the presence of a big $\sigma > 0$ (for a fixed $\alpha_2 > 0$) might render the particles asymptotically independent with $M \equiv 0$. Assuming this is the case, the limit process for each magnetization should be given by independent copies of a non-Markovian spin with jump times distributed as the hitting times of the Ornstein-Uhlenbeck. Moreover, from Fig.\! \ref{scalino_ordN} we see that the jumps should occur precisely at the borders of the fixed points region (the purple and green lines). This regime appears then to be related to the mean field scenario of Section \ref{mean_field_hier} for $\beta \to \infty$, highlighted in Remark \ref{limiting_case_b=infty}.
\begin{figure}%
    \centering
    \subfloat[Initial configuration at order $N$]{{\includegraphics[width=5.5cm]{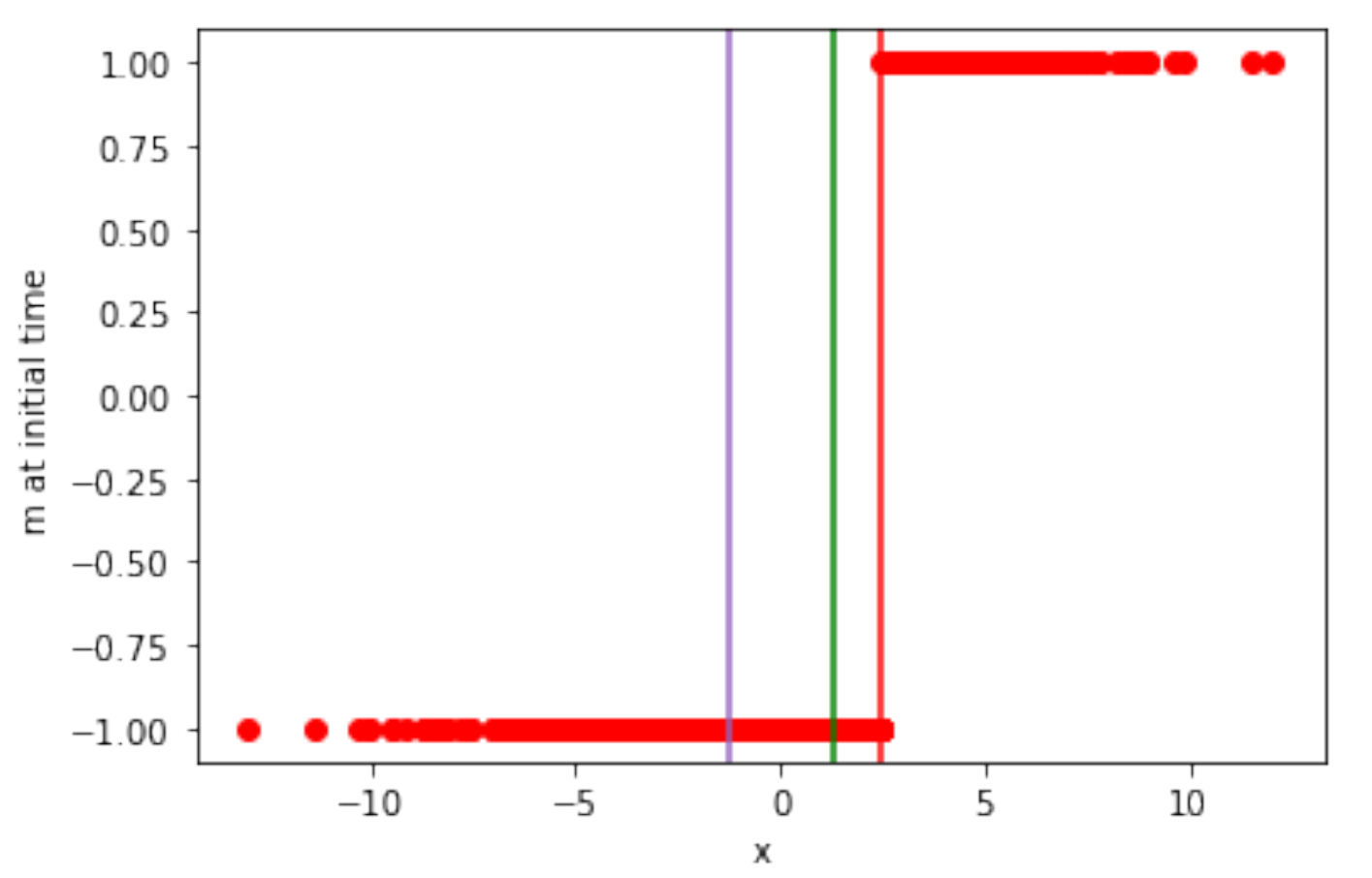}}}%
    \qquad
    \subfloat[Final configuration at order $N$]{{\includegraphics[width=5.5cm]{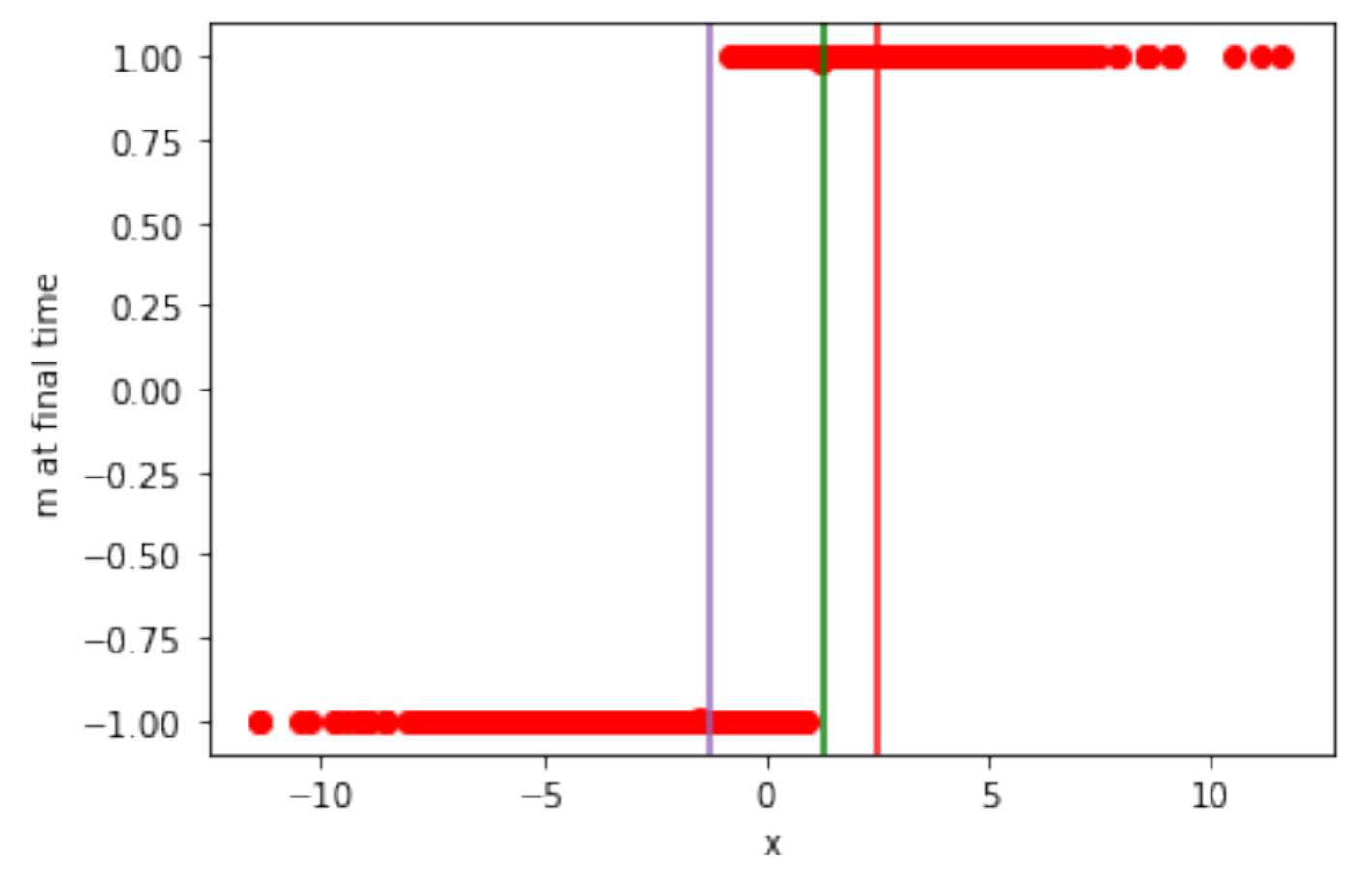} }}%
    \caption{Simulation of the finite particle system's dynamics at a timescale of order $N$, for $N = 1000$, $\beta_1= \beta_2=\infty$, $\alpha_1 = \alpha_2 = 1$, and $\sigma = 5$.  As above, we take $x_j(0) \sim \mathcal{N}\left(0,\frac{\sigma^2}{2\alpha_2}\right)$.}   
    \label{scalino_ordN}
\end{figure}

\begin{figure}%
\centering
\subfloat[$(t,M^N(t))$]{{\includegraphics[width=5.5cm]{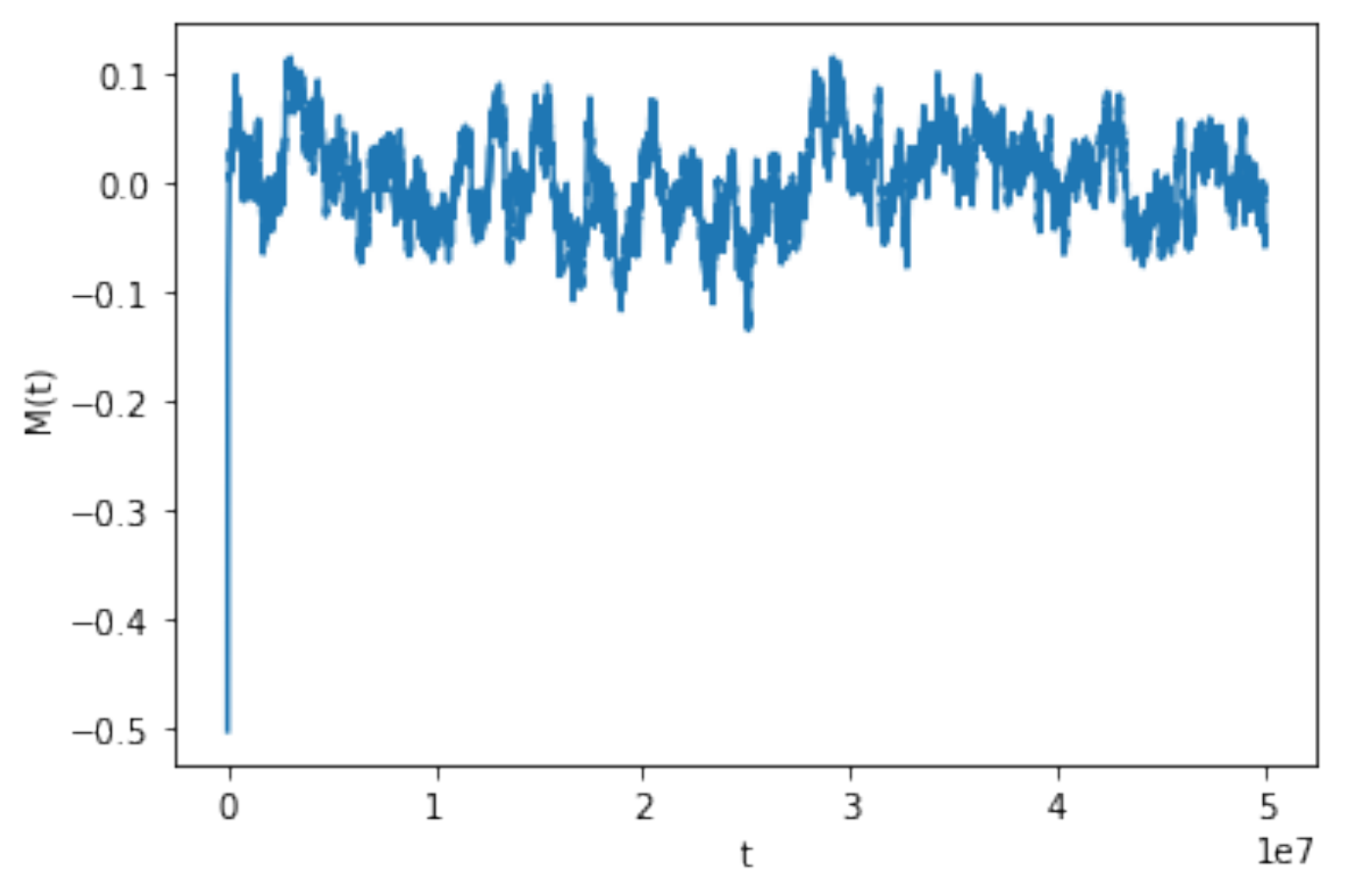}}}%
\caption{The path of the empirical $2$-level magnetization for the same simulation of Fig.\! \ref{scalino_ordN}.}   
\label{magn_ordN_bigsigma}
\end{figure}

When the parameter $\frac{\sigma^2}{2\alpha_2}$ is small, we instead witness the loss of stability of certain areas of the fixed points region. For a description of this we refer to the next section, where we describe the full dynamics at order $N^2$.

\subsubsection{Order $N^2$ dynamics}
\label{ordN^2_subsub}

In order to describe the order $N^2$ dynamics, we proceed as above by looking at the conditional dynamics with respect to the values of the macroscopic limit diffusion $X(t)$. We fix an initial condition with $X(0) \neq 0$ and evolve the dynamics at times of order $1$. The latter  converges soon to some staircase equilibrium in the fixed points region and stays put for all times of order $1$. At times of order $N$ we then see some diffusive behavior of the equilibrium around the fixed $X$, until the process $X(t)$ changes again. The reiteration of this procedure for the updated value of $X$ describes an infinitesimal time step in the order $N^2$ timescale.
The order $1$ conditional dynamics is given by
\begin{equation}
\label{eqn:limit_infX}
\begin{cases}
\dot{m}(t)(x) = 2\text{sign}(x + m(t)(x) + M(t)+ X) - 2m(t)(x),\\
m(0)(x) = m^{x_0}(x),\\
M(t) = \int_{\mathbb{R}}m(t)(x)\mu_X^{\infty}(dx),
\end{cases}
\end{equation}
for some fixed $X \in \mathbb{R}$ and some staircase initial condition $m^{x_0}$ for $m$, which was reached at the previous timescale long-time limit.
In \eqref{eqn:limit_infX}, $\mu_X^{\infty} = \mathcal{N}\left(X, \frac{\sigma^2}{2\alpha_2}\right)$ is the asymptotic distribution of the (sped up) Ornstein-Uhlenbeck processes for a fixed value of $X$.
Proposition \ref{scalino} generalizes to
\begin{prop}[Shape of the equilibria]
\label{scalinoX}
Every equilibrium of Eq. \eqref{eqn:limit_infX} is a staircase function $m^{x_0}(x)$ of the form
\begin{equation}
\label{eqn:scalinoX}
m^{x_0}(x):=
\begin{cases}
+1, \ \ \ \forall x > x_0,\\
-1, \ \ \ \forall x < x_0,
\end{cases}
\end{equation}
for some $x_0 \in \mathbb{R}$ satisfying 
\begin{equation}
\label{eqn:posnegX}
- 2\mu_0^{\infty}(x_0 - X, +\infty) \leq x_0 + X \leq 2 \mu_0^{\infty}(-\infty, x_0-X),
\end{equation}
\end{prop}
\begin{proof}
The proof follows the same steps as in the proof of Proposition \ref{scalino}, observing that
$$
\mu_X^{\infty}(-\infty, x_0) = \mu_0^{\infty}(-\infty,x_0-X),
$$
with $\mu_0^\infty = \mathcal{N}\left(0,\frac{\sigma^2}{2\alpha_2}\right)$.
\end{proof}

Proposition \ref{attractor} generalizes to
\begin{prop}[Stable attractors of the dynamics]
\label{attractorX}
Let $m(t)(x)$ be the solution to Eq. \eqref{eqn:limit_infX} with initial datum $m_0(x) = m^{x_0}(x)$, with $x_0 + X > 0$ (resp.\! $x_0  + X < 0$) such that $x_0+ X > 2 \mu_X^{\infty}(-\infty, x_0)$ (resp.\! $x_0 + X < - 2\mu_X^{\infty}(x_0, + \infty)$). Then, we have
$$
\lim_{t \to \infty} m(t)(x) = m^{\overline{x_0}}(x),
$$
with $\overline{x_0} = 2\mu_X^{\infty}(-\infty,\overline{x_0})$ (resp.\! $\overline{x_0} = - 2\mu_X^{\infty}(\overline{x_0},+\infty)$).
\end{prop}

\begin{rem}
The borders of the fixed points region of Proposition \ref{attractorX} can be expressed in terms of $(X,M)$, $M$ being the asymptotically stable value of $M(t)$ in Eq. \eqref{eqn:limit_infX}:
\begin{equation}
\label{eqn:left_right_border}
\overline{x_0}(X,M) = -1- X- M, \qquad \overline{x_0}(X,M) = 1- X- M,
\end{equation}
respectively for the left border (i.e.\! for $\overline{x_0} + X < 0$), and the right border (i.e.\! for $\overline{x_0} + X > 0$).
The expressions in \eqref{eqn:left_right_border} can be derived by using that $M = 1 - 2 \mu_X^{\infty}(-\infty, \overline{x_0})$.
\end{rem}
In Fig.\! \ref{fixed_points_region} we plot the fixed points region as a parametric function of $X$ and $M$, with the two borders respectively given by
\begin{equation}
\label{eqn:l_r_bord}
M =  1 - 2\mu_X^{\infty}(-\infty, -1-X-M), \quad M = 1 -2\mu_X^{\infty}(-\infty, 1-X-M).
\end{equation}
We can distinguish two regimes depending on the diffusion parameters: for large values of $\frac{\sigma^2}{2\alpha_2}$, both expressions in \eqref{eqn:l_r_bord} define the graph of a function $M = \psi(X)$, while this is not the case when $\frac{\sigma^2}{2\alpha_2}$ is small. Unfortunately, we were not able to determine the precise value of $\sigma$ and $\alpha_2$ where this transition takes place, due to the implicit character of the equations in play.
\begin{figure}%
    \centering
    \subfloat[Fixed points region for large $\frac{\sigma^2}{2\alpha_2}$]{{\includegraphics[width=5.5cm]{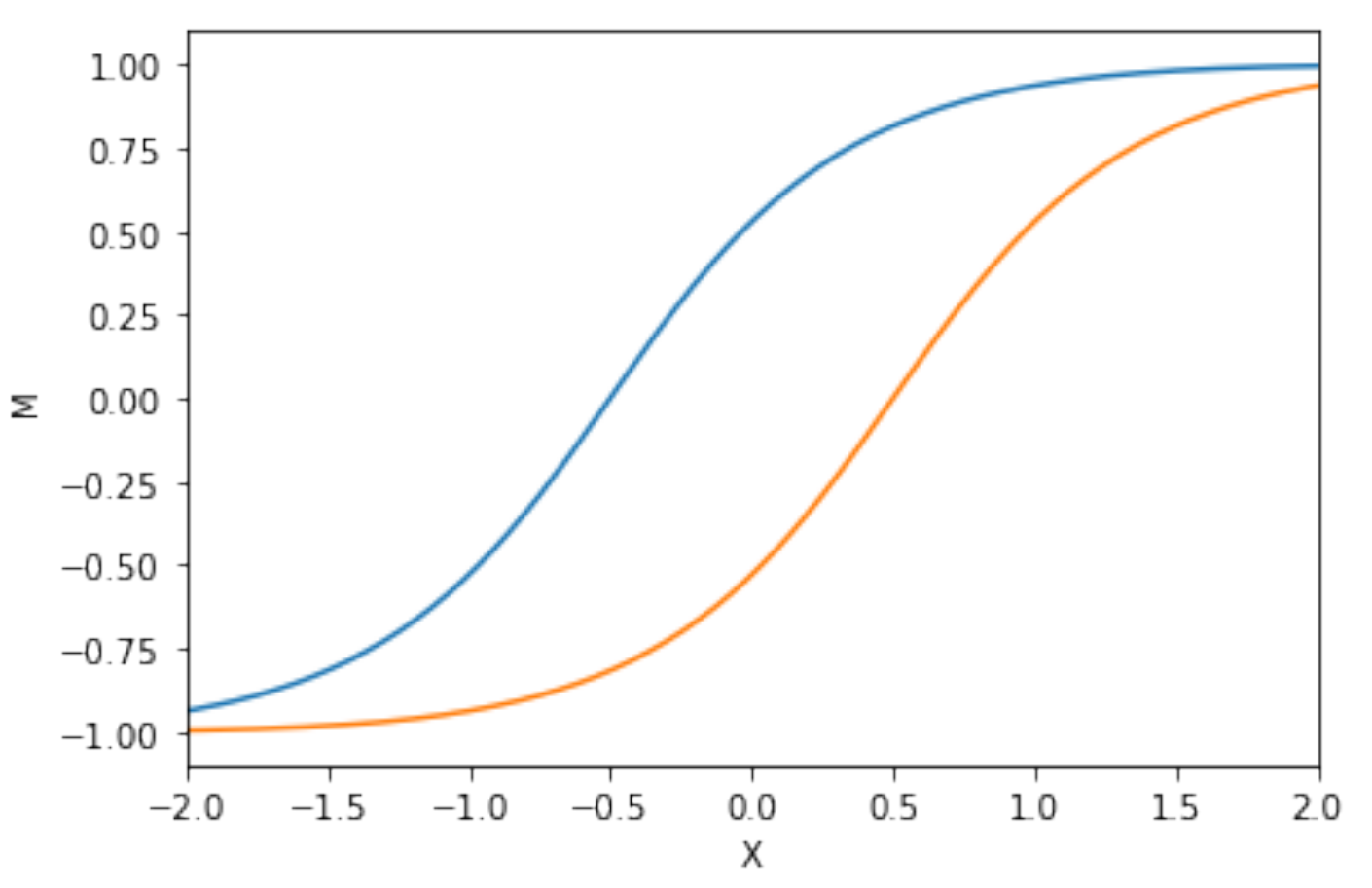}}}%
    \qquad
    \subfloat[Fixed points region for small $\frac{\sigma^2}{2\alpha_2}$]{{\includegraphics[width=5.5cm]{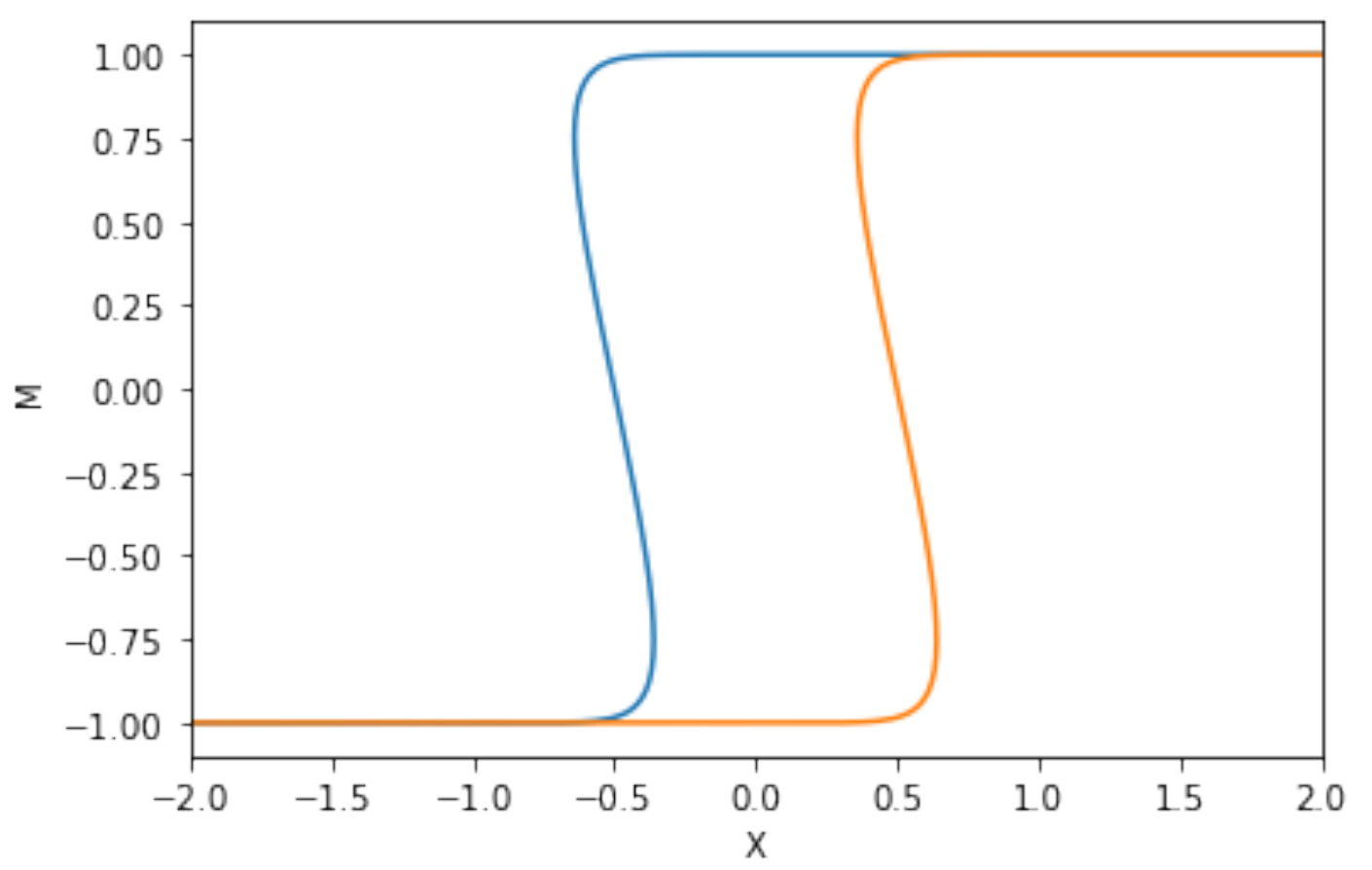} }}%
    \caption{Fixed points region for $\alpha_2 = 1$, $\sigma=3$ (left), and for $\alpha_2 = 3$, $\sigma = 1$ (right).}   
    \label{fixed_points_region}
\end{figure}
Simulations suggest that the dynamics of $(X^N(t),M^N(t))$ is substantially different in the two cases: for big values of $\frac{\sigma^2}{2\alpha_2}$, we observe a diffusive motion onto the fixed points region, while for small $\frac{\sigma^2}{2\alpha_2}$ the dynamics resembles a diffusion with jumps. In both cases, as we noted for the simpler case $X^N(t) \equiv 0$, the single $1$-level magnetizations $m_i^N(t)$'s should be evolving as non-Markovian spins, this time interacting since $X^N(t) \neq 0$.
This situation appears to be comparable to its mean field counterpart shown in Fig.\! \ref{mf-sim}, with the diffusive parameters playing the role of (the inverse of) $\beta$.
\begin{figure}%
    \centering
    \subfloat[$X^N(t) \equiv X_0 = 0.4$, $M_0 = 0.802$ ]{{\includegraphics[width=5.5cm]{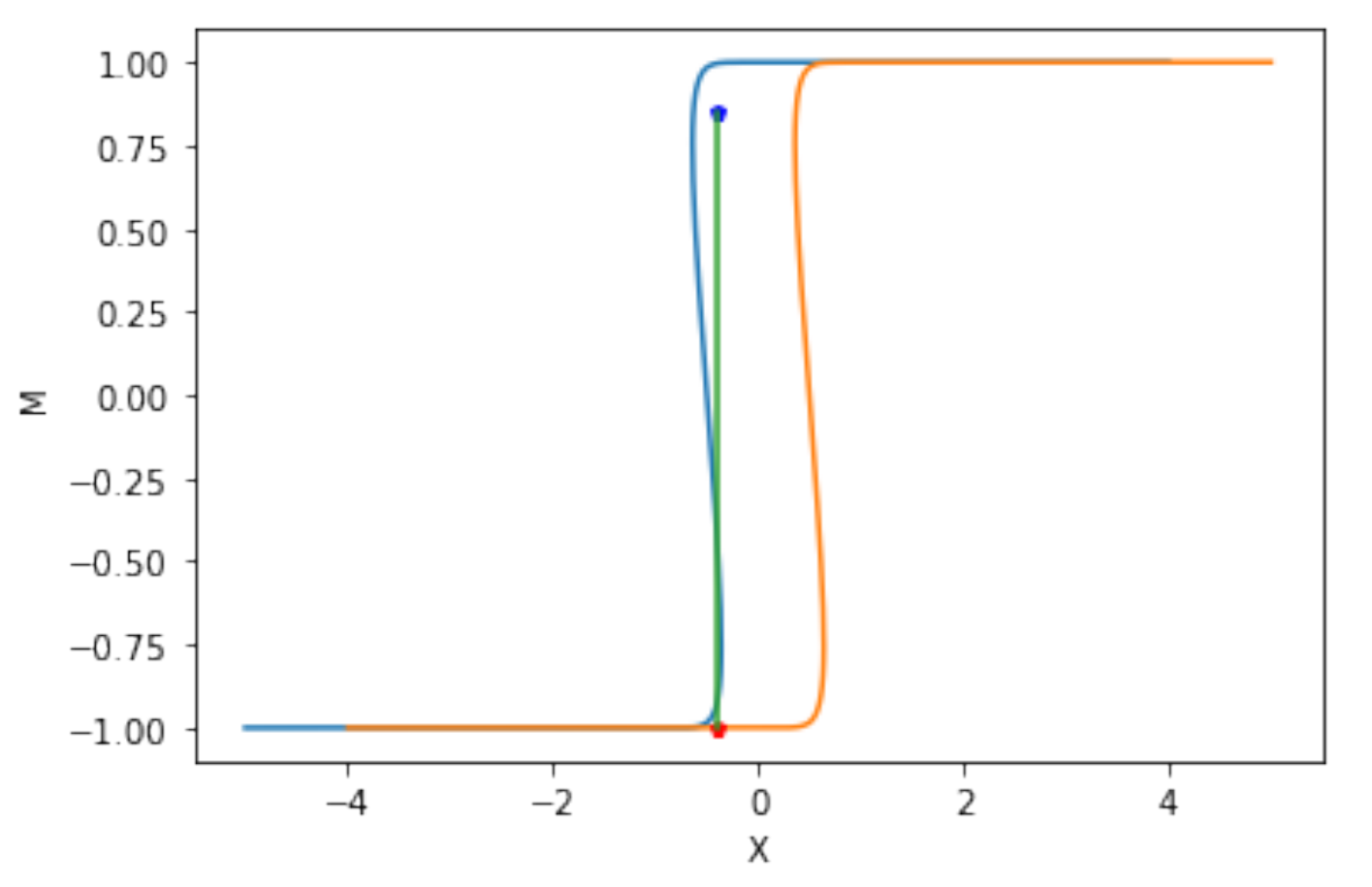}}}%
    \qquad
    \subfloat[$X^N(t) \equiv X_0 = -0.4$, $M_0 = -0.802$]{{\includegraphics[width=5.5cm]{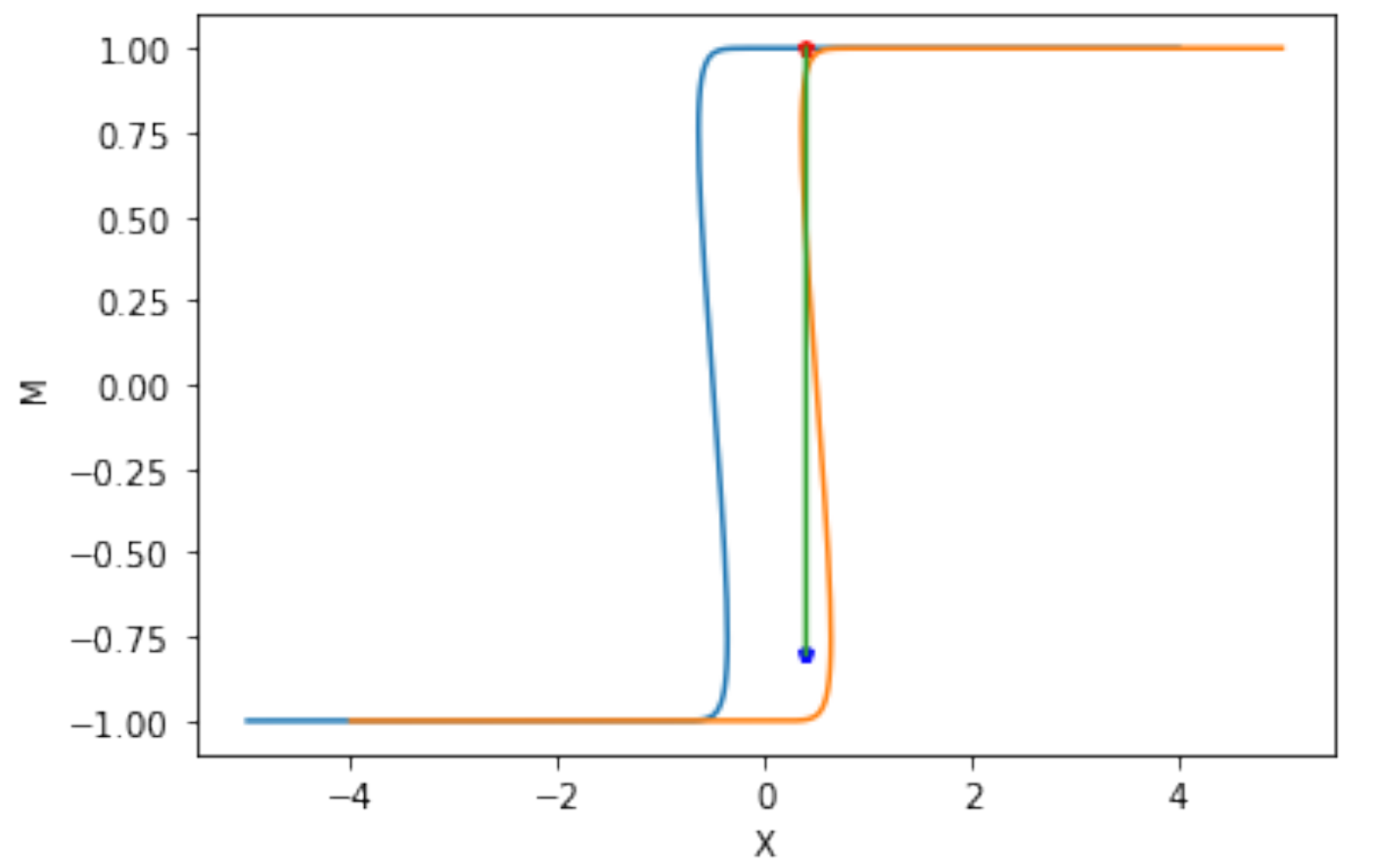} }}%
    \caption{Two simulations of the order $N$ dynamics, with symmetric initial conditions in the two unstable regions, with $\alpha_2 = 3$, $\sigma=1$, $\beta_1=\beta_2 = \infty$. The blue (red) dot is the initial (final) point of each trajectory.}   
    \label{unstable}
\end{figure}
As in the mean field case, the jumps seem to be occuring because of a loss of stability of the fixed points in certain areas of the phase-space. This loss of stability seems to originate at an order $N$ timescale (this is also in parallel with the mean field case, where it was originating at an order $1$ timescale). Indeed, for $\frac{\sigma^2}{2\alpha_2}$ small, starting the dynamics \eqref{eqn:limit_infX} from a staircase equilibrium which belongs to a certain area of the fixed points region, and letting it evolve for times of order $N$ when the $x_i$'s start their motion, we see a fast trajectory which very soon gets attracted to an area close to the opposite border from which it started. An example of this is shown in Fig.\! \ref{unstable}, where we have kept fixed $X^N(t) \equiv X(0)$ to simulate the dynamics at a timescale of order $N$: we indeed see two fast transient trajectories starting from two initial points (the blue dots) which seem to belong to the unstable regions of fixed points. The same simulation at an order $1$ timescale would have instead shown a trivial dynamics constantly equal to the initial datum, for any choice of the latter inside the two-dimensional manifold of fixed points.  Finally, in Fig.\! \ref{complete_N^2} we show a complete simulation of the trajectories $(X^N(t),M^N(t))$ at a timescale of order $N^2$, showing the different behavior depending on the value of $\frac{\sigma^2}{2\alpha_2}$. In the right plot, the white areas at the borders of the fixed points region which are not hit by any trajectory should approximate the unstable regions of fixed points. 
\begin{figure}%
    \centering
    \subfloat[$\frac{\sigma^2}{2\alpha_2}$ large]{{\includegraphics[width=5.5cm]{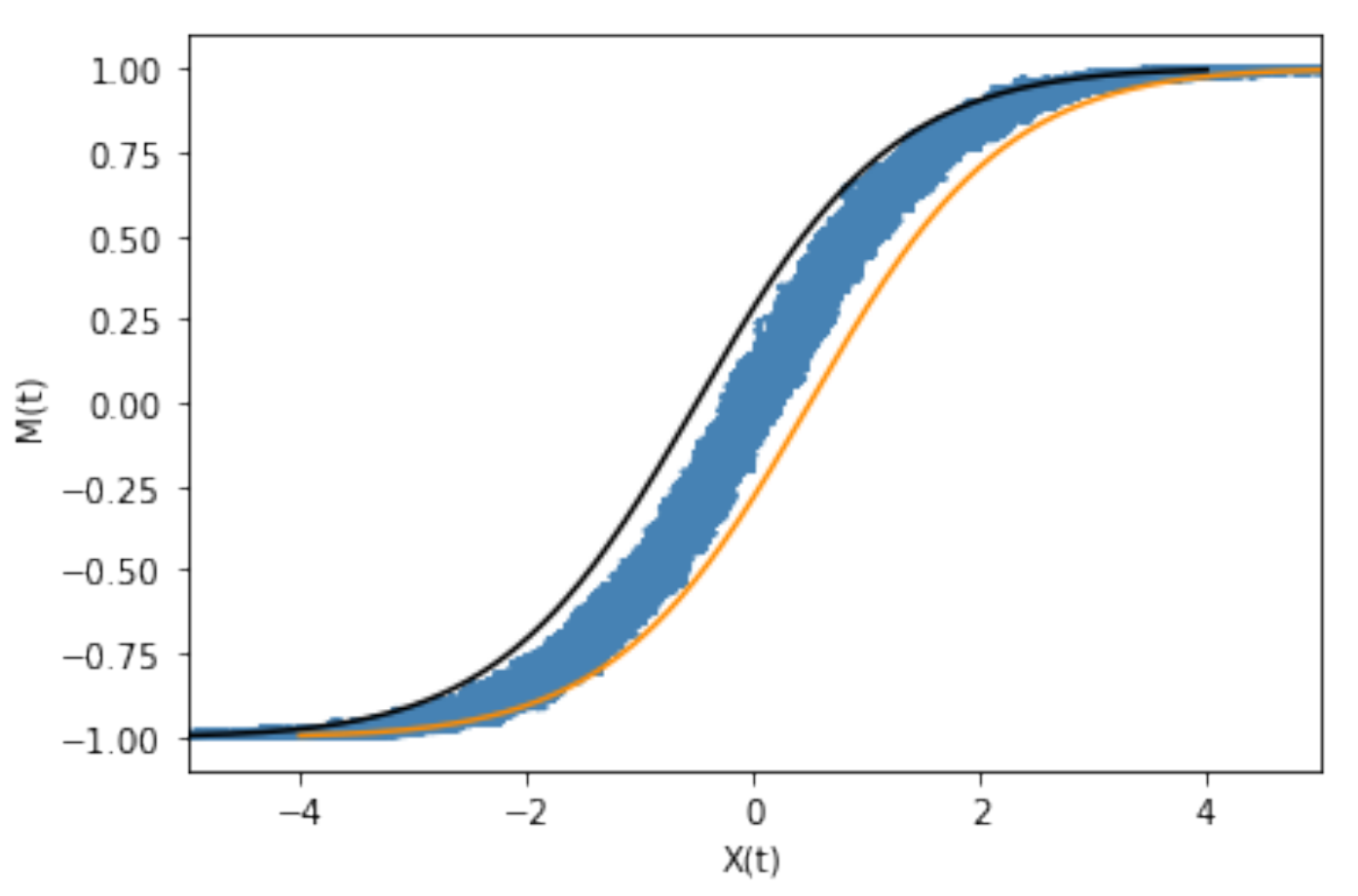}}}%
    \qquad
    \subfloat[$\frac{\sigma^2}{2\alpha_2}$ small]{{\includegraphics[width=5.5cm]{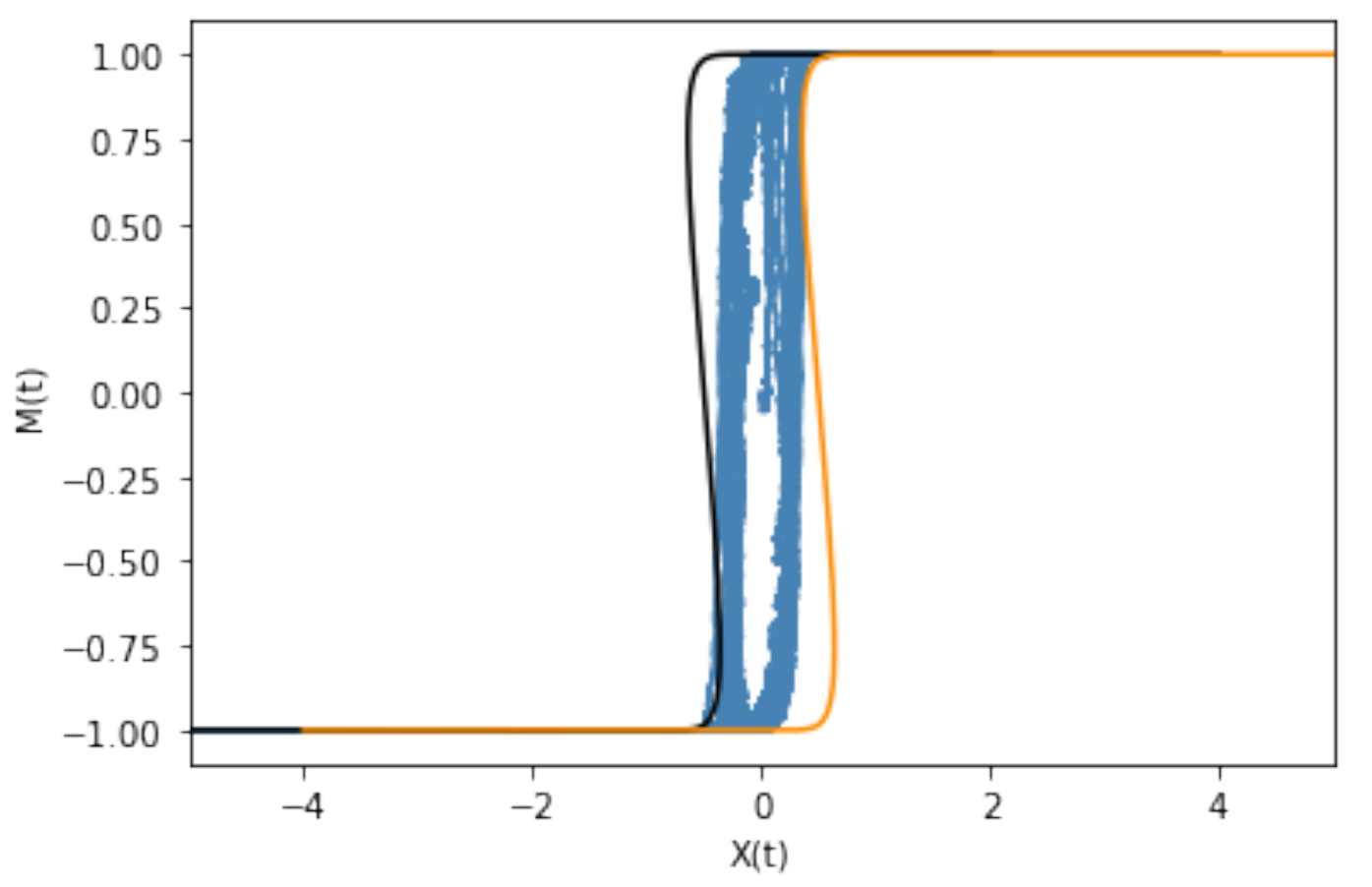} }}%
    \caption{Two simulated trajectories of $(X^N(t), M^N(t))$ at the order $N^2$ timescale for $N = 500$ and $T = 5\times10^8$ , with $\beta_1=\beta_2 = \infty$. On the left $\alpha_2 = 1$, $\sigma=5$, while on the right $\alpha_2 = 3$, $\sigma=1$.}   
    \label{complete_N^2}
\end{figure}


\begin{thebibliography}{9}
\bibitem{athreya}
S. R. Athreya, J. M. Swart. Survival of contact processes on the hierarchical group,
\emph{Probab. Theory Relat. Fields}, 147:529–563, 2010.

\bibitem{bleher}
P.M. Bleher, P. Major. Critical phenomena and universal exponents in statistical physics. On Dyson’s hierarchical model, 
\emph{Ann. Probab.}, 15(2), 431–477, 1987.

\bibitem{bremaud}
P. Bremaud. \emph{Markov Chains. Gibbs Fields, Monte Carlo Simulation, and Queues}, Springer-Verlag New York, 1998.

\bibitem{collet}
F. Collet, P. Dai Pra. The role of disorder in the dynamics of critical fluctuations of mean field models, \emph{Electron. J. Probab.}, 17(26):1–40, 2012.

\bibitem{comets}
F. Comets, Th. Eisele.
Asymptotic Dynamics, Non-Critical and Critical Fluctuations for a Geometric Long-Range Interacting Model, \emph{Commun. Math. Phys.}, 118: 531-567, 1988.

\bibitem{cox_longtime}
T. Cox, A. Greven. On the long term behaviour of some finite particle systems, \emph{Probab. Theory Relat. Fields}, 85, 195-237, 1990.

\bibitem{daipra_reg_diss}
P. Dai Pra, M. Fischer, and D. Regoli. A Curie-Weiss model with dissipation, \emph{Journal of Statistical Physics}, 152(1):37-53, 2013.

\bibitem{dawson_greven_hier}
D. A. Dawson, A. Greven. Hierarchical models of interacting diffusions: Multiple time scale phenomena, phase transition and pattern of cluster-formation, \emph{Probab. Theory Relat. Fields}, 96:435M73, 1993.

\bibitem{dawson_greven_1lev}
D. A. Dawson, A. Greven. Multiple time scale analysis of interacting diffusions,
\emph{Probab. Theory Relat. Fields}, 95, 467-508, 1993.

\bibitem{dyson}
F.J. Dyson. Existence of a phase transition in a one-dimensional Ising ferromagnet, 
\emph{Commun. Math. Phys. 12}, 91–107, 1969.

\bibitem{derrida}
B. Derrida, L. De Seze, and C. Itzykson. Fractal Structure of Zeros in Hierarchical Models, 
\emph{J. Stat. Phys.}, 33: 559, 1983.

\bibitem{desimoi}
J. De Simoi, S. Marmi. Potts models on hierarchical lattices and Renormalization Group dynamics, 
\emph{Journal of Physics A Mathematical and Theoretical}, 42. 10.1088/1751-8113/42/9/095001, 2007. 

\bibitem{duong}
M. H. Duong, G. A. Pavliotis. Mean field limits for non-Markovian interacting particles: convergence to equilibrium, GENERIC formalism, asymptotic limits and phase transitions, \texttt{arXiv:1805.04959}, 2018.

\bibitem{ethier}
S. N. Ethier, T. G. Kurtz. \emph{Markov processes, Characterization and Convergence}, Wiley, Hoboken (NJ), 1986.

\bibitem{Graham} 
C. Graham. McKean-Vlasov It\^o-Skorokhod equations, and nonlinear diffusions with
discrete jump sets, \emph{Stochastic Processes Appl.}, 40(1):69-82, 1992.

\bibitem{hara}
T. Hara, T. Hattori, and H. Watanabe. Triviality of hierarchical Ising model in four dimensions, 
\emph{Commun. Math. Phys.}, 220(1), 13–40, 2001.

\bibitem{denholl}
W.Th.F. den Hollander. Renormalization of interacting diffusions : a program and four examples, \emph{Report Eurandom; Vol. 2005025}. Eindhoven: Eurandom, 2005.

\bibitem{kaufman}
M. Kaufman, R. Griffiths. Exactly soluble Ising models on hierarchical lattices, 
\emph{Physical Review B}, 24. 1-1981. 10.1103/PhysRevB.24.496, 1981.

\bibitem{parsons} 
T. L. Parsons, T. Rogers. Dimension reduction for stochastic dynamical systems forced onto a manifold by large drift: a constructive approach with examples from theoretical biology, \emph{J. Phys. A: Math. Theor.} 50, 41, 2017.
\end{thebibliography}
\end{document}